\documentclass[11pt]{amsart}
\usepackage{amssymb, amsmath, amscd, eucal, rotating}
\usepackage{pstricks}
\usepackage{mathrsfs}

\input xy
 \xyoption{all}

\numberwithin{equation}{section}

\theoremstyle{plain}

\newtheorem{proposition}{Proposition}[section]
\newtheorem{theorem}[proposition]{Theorem}		
\newtheorem{corollary}[proposition]{Corollary}
\newtheorem{lemma}[proposition]{Lemma}

\theoremstyle{definition}

\newtheorem{definition}[proposition]{Definition}
\newtheorem{remark}[proposition]{Remark}
\newtheorem{example}[proposition]{Example}

\newcommand{\C}{\mathbb C}
\newcommand{\R}{\mathbb R}

\newcommand{\Gr}{\mathop{\rm Gr}\nolimits}

\newcommand{\End}{\mathop{\rm End}\nolimits}
\newcommand{\coker}{\mathop{\rm coker}\nolimits}

\newcommand{\Hom}{\mathop{\rm Hom}\nolimits}

\newcommand{\ad}{\mathop{\rm ad}\nolimits}
\newcommand{\Ad}{\mathop{\rm Ad}\nolimits}

\newcommand{\tr}{\mathop{\rm Tr}\nolimits}

\newcommand{\GL}{\mathsf{GL}}

\newcommand{\U}{\mathsf{U}}

\newcommand{\doubleslash}{\bigr/ \negthinspace\negthinspace \bigr/}

\def\overlaprel{\mathrel{\mkern-4mu}}

\newcommand{\longlongrightarrow}{\ensuremath{\relbar\overlaprel\longrightarrow}}

\DeclareMathOperator{\HNJH}{HNJH}
\DeclareMathOperator{\JH}{JH}

\DeclareMathOperator{\Rep}{Rep}
\DeclareMathOperator{\Vect}{Vect}
\DeclareMathOperator{\id}{id}
\DeclareMathOperator{\im}{im}
\DeclareMathOperator{\rank}{rank}
\DeclareMathOperator{\slope}{slope}

\DeclareMathOperator{\grad}{grad}

\DeclareMathOperator{\Rel}{Rel}

\begin{document}

% Topmatter

\title{Moment map flows and the Hecke correspondence for quivers}

\author{Graeme Wilkin}
\address{Department of Mathematics,
National University of Singapore, 
Singapore 119076}
\email{graeme@nus.edu.sg}
\date{\today}

\thanks{This research was partially supported by grant number R-146-000-200-112 from the National University of Singapore. The author also acknowledges support from NSF grants DMS 1107452, 1107263, 1107367 ``RNMS GEometric structures And Representation varieties'' (the GEAR Network).}


\begin{abstract}
In this paper we investigate the convergence properties of the upwards gradient flow of the norm-square of a moment map on the space of representations of a quiver. The first main result gives a necessary and sufficient algebraic criterion for a complex group orbit to intersect the unstable set of a given critical point. Therefore we can classify all of the isomorphism classes which contain an initial condition that flows up to a given critical point. As an application, we then show that Nakajima's Hecke correspondence for quivers has a Morse-theoretic interpretation as pairs of critical points connected by flow lines for the norm-square of a moment map. The results are valid in the general setting of finite quivers with relations. 
\end{abstract}

% End Topmatter

%%%%%%%%%%%%%%%%%%%%
% Disclaimer
%%%%%%%%%%%%%%%%%%%%
%\begin{center}
%\framebox{
%{\Large\bf DRAFT (\today): DO NOT DISTRIBUTE.}}
%\end{center}

\maketitle

%\tableofcontents

\thispagestyle{empty}

\baselineskip=16pt
%\thispagestyle{plain}
%\setcounter{page}{1}

%\setcounter{section}{-1}

%%%%%%%%%%%%%%%%%%%%%%%%%%%%%%%%%%%%%%%%%%%%%%%%

\section{Introduction}

There is a well-known correspondence between quotients in symplectic and algebraic geometry. For example, the Kempf-Ness theorem relates GIT quotients and symplectic quotients of affine spaces \cite{KempfNess79} and the Donaldson-Uhlenbeck-Yau theorem relates moduli spaces of polystable holomorphic bundles to moduli spaces of Yang-Mills minima \cite{Donaldson85, Donaldson87, UhlenbeckYau86}. 

In many examples of interest, there is also a symplectic-algebraic correspondence between GIT unstable points and critical points of a moment map functional, given by taking the limit of the downwards gradient flow of the norm-square of a moment map. This originated in Kirwan's work \cite{Kirwan84} for projective varieties, and the case of the Yang-Mills flow is studied in \cite{Daskal92}, \cite{DaskalWentworth04} and \cite{Sibley15}. A version of this theorem for quiver varieties is proved in \cite[Thm. 3]{HaradaWilkin11}, which has since been generalised in \cite{Hoskins14} to the case of reductive group actions on affine spaces.

The goal of this paper is to further extend this correspondence between symplectic and algebraic geometry to spaces of flow lines between critical sets on the space of representations of a quiver with relations. The symplectic side of the picture determines the critical points and the flow lines for the norm-square of the moment map and one of the main theorems of this paper gives an algebraic criterion for the existence of a flow line connecting two critical points. Using this criterion we then show that critical points connected by flow lines can be interpreted in terms of the Hecke correspondence for quivers defined in \cite{Nakajima94}, \cite{Nakajima98}.

The setup is explained in detail in Section \ref{subsec:quiver-background}; here we provide a summary of the points relevant to the rest of the introduction. The vector space $\Rep(Q, {\bf v})$ of complex representations of a quiver with a fixed dimension vector ${\bf v}$  has a natural symplectic structure determined by the Hermitian structure on the complex vector space at each vertex of the quiver. Associated to this is a Hamiltonian action of a group $K_{\bf v}$ and a moment map which we denote by $\mu$. The complexification of $K_{\bf v}$ is denoted $G_{\bf v}$. In the paper \cite{HaradaWilkin11} we studied the downwards gradient flow of $\| \mu - \alpha \|^2$ for any stability parameter $\alpha$ and showed the existence of a Morse stratification which coincides with the algebraic Harder-Narasimhan stratification defined by Geometric Invariant Theory (cf. \cite{Reineke03}). The main theorem of \cite{HaradaWilkin11} identifies the isomorphism class of the limit of the downwards flow with the graded object of the Harder-Narasimhan-Jordan-H\"older filtration associated to the initial condition.

The entire setup described above restricts to any closed subset $Z \subset \Rep(Q, {\bf v})$ which is preserved by the action of $G_{\bf v}$. An important special case of this is the variety of representations of a quiver satisfying a finite set of relations in the path algebra. Moreover, the critical sets of the norm-square of the moment map have a simple interpretation in terms of representations which are the direct sum of subrepresentations of different slopes minimising the norm-square of the moment map (cf. \eqref{eqn:critical-splitting}). Therefore the critical sets can be classified by the dimension and slope of each of the subrepresentations in this splitting in analogy with the classification of the critical sets of the Yang-Mills functional described by Atiyah and Bott in \cite{AtiyahBott83}. 

Given a critical point $x$, let $W_x^-$ denote the \emph{unstable set} of initial conditions for which the upwards flow of $\| \mu - \alpha \|^2$ converges to $x$. Associated to the critical point is another space called the \emph{negative slice} $S_x^-$, which is defined in terms of a group action around the critical point (cf. Definition \ref{def:neg-slice-def}). When the ambient space is smooth then $S_x^-$ is simply the exponential image of the negative eigenspace of the Hessian. Given a critical set $C$, we have fibrations $\pi_w : W_C^- \rightarrow C$ and $\pi_s : S_C^- \rightarrow C$ for which $\pi_w^{-1}(x) = W_x^-$ and $\pi_s^{-1}(x) = S_x^-$. The first main result of the paper is Theorem \ref{thm:homeo-slice-bundle}.

\begin{theorem}
There are neighbourhoods $U$ and $V$ of $C \subset Z$ and a $K_{\bf v}$-equivariant homeomorphism $\psi :S_C^- \cap U \rightarrow W_C^- \cap V$. Moreover, for each $y \in S_C^- \cap U$ there exists $g(y) \in G_{\bf v}$ such that $\psi(y) = g(y) \cdot y \in W_C^- \cap V$. 
\end{theorem}

Since the negative slice $S_x^-$ is defined algebraically, then it is possible to classify the isomorphism classes of representations in $S_x^-$ in terms of filtrations for which the quotients are semistable, the slopes are increasing and the graded object is isomorphic to the critical point $x$ (cf. Lemma \ref{lem:neg-slice-filtration}). Therefore, using the above theorem we can completely classify the isomorphism classes in $W_x^-$. 

\begin{theorem}\label{thm:algebraic-unstable-intro}
Let $x$ be a critical point of $\| \mu - \alpha \|^2$ on $\nu^{-1}(0)$ and let $x = x_1 \oplus \cdots \oplus x_n$ be the decomposition as a direct sum of stable representations as in \eqref{eqn:critical-decomposition}, ordered so that $\slope_\alpha(x_j) < \slope_\alpha(x_k)$ for all $j < k$. Then every $y \in W_x^-$ admits a  filtration $y_1 \subset \cdots \subset y_n$ such that each quotient $y_k / y_{k-1}$ is isomorphic to $x_k$ for $k=1, \ldots, n$. Conversely, let $y \in \nu^{-1}(0)$ be a representation admitting a filtration $y_1 \subset \cdots \subset y_n$ such that each quotient $y_k / y_{k-1}$ is isomorphic to $x_k$ for $k=1, \ldots, n$. Then there exists $g \in G_{\bf v}$ such that $g \cdot y \in W_x^-$.
\end{theorem}

In contrast to the Harder-Narasimhan filtration, where the slopes of the semistable quotients are decreasing, filtrations of the type in the above theorem are not necessarily unique. This reflects the fact that a single isomorphism class may contain different initial conditions that flow up to different critical points. Equivalently, if $y \in W_x^-$ then there may exist $g \in G_{\bf v}$ such that $g \cdot y \in W_{x'}^-$ for some higher critical point $x'$, or perhaps the upwards flow with initial condition $g \cdot y$ may not converge to any critical point. Therefore we see the difference between the properties of the upwards and downwards flow: the results of \cite{HaradaWilkin11} show that the isomorphism class of the limit of the downwards flow is an isomorphism invariant of the initial condition, however Theorem \ref{thm:algebraic-unstable-intro} shows that the upwards flow is more delicate in the sense that a small perturbation within an isomorphism class may completely change the convergence properties of the upwards flow.

Together with known results on the limit of the downwards flow from \cite{HaradaWilkin11}, Theorem \ref{thm:algebraic-unstable-intro} leads to an algebraic criterion for two critical points to be connected by a flow line, which we can then exploit to give a Morse-theoretic construction of the Hecke correspondence in Theorem \ref{thm:hecke-intro} below.

The precise statement is as follows. Let $Q$ be a finite quiver, choose a dimension vector ${\bf v}$, and let $\alpha$ denote the stability parameter from \cite{Nakajima94}, \cite{Nakajima98} (cf. Definition \ref{def:aasp}). Let $\Rep(Q, {\bf v})$ denote the affine space of representations with dimension vector ${\bf v}$, let $G_{\bf v}$ be the associated complex reductive group acting on $\Rep(Q, {\bf v})$ (cf. \eqref{eqn:group-action}) and let $Z \subset \Rep(Q, {\bf v})$ be any closed subset such that $G_{\bf v} \cdot Z \subset Z$. As mentioned above, a particular case of interest is when $Z$ is the subvariety of representations satisfying a finite set of relations on the quiver. Define $f : \Rep(Q, {\bf v}) \rightarrow \R$ by $f(x) = \| \mu(x) - \alpha \|^2$. The gradient flow for $f$ is the solution of \eqref{eqn:neg-grad-flow}. 

The above choice of stability parameter implies that any critical point $x$ for $f$ splits into two subrepresentations which we denote $x = x_1 \oplus x_2$, where $x_1$ is stable of negative slope, and $x_2$ is semistable of positive slope. Given dimension vectors ${\bf v_1^u} < {\bf v_1^\ell} < {\bf v}$, let $C_{\bf v_1^u}$ (resp. $C_{\bf v_1^\ell}$) denote the critical sets on $Z$ for which the associated stable subrepresentation of negative slope has dimension vector ${\bf v_1^u}$ (resp. ${\bf v_1^\ell}$). Since ${\bf v_1^u} < {\bf v_1^\ell}$ then $f(C_{\bf v_1^u}) > f(C_{\bf v_1^\ell})$. Now let $\mathcal{P}_{{\bf v_1^u}, {\bf v_1^\ell}} \subset C_{\bf v_1^u} \times C_{\bf v_1^\ell}$ denote the subset of pairs of critical points which are connected by a gradient flow line. There are projection maps $\mathcal{P}_{{\bf v_1^u}, {\bf v_1^\ell}} \rightarrow C_{\bf v_1^u}$ and $\mathcal{P}_{{\bf v_1^u}, {\bf v_1^\ell}} \rightarrow C_{\bf v_1^\ell}$ defined by projecting a flow line to its upper and lower endpoints. Lemma \ref{lem:quotient-critical} shows that there are also natural projection maps $C_{\bf v_1^u} \rightarrow \mathcal{M}(Q, {\bf v_1^u})$ (resp. $C_{\bf v_1^\ell} \rightarrow \mathcal{M}(Q, {\bf v_1^\ell})$) onto the moduli space of $\alpha$-stable representations with dimension vector ${\bf v_1^u}$ (resp. ${\bf v_1^\ell}$). 

Since the flow is equivariant with respect to the maximal compact subgroup $K_{\bf v} \subset G_{\bf v}$, then there is an induced correspondence variety $\mathcal{M}_{{\bf v_1^u}, {\bf v_1^\ell}}$ which fits into the diagram below
\begin{equation*}
\xymatrix{
& \ar[dl] \mathcal{P}_{{\bf v_1^u}, {\bf v_1^\ell}} \ar[dr] \ar@{-->}[d] & \\
C_{\bf v_1^u} \ar[d] & \ar[dl] \mathcal{M}_{{\bf v_1^u}, {\bf v_1^\ell}} \ar[dr] & C_{\bf v_1^\ell} \ar[d] \\
\mathcal{M}(Q, {\bf v_1^u}) & & \mathcal{M}(Q, {\bf v_1^\ell})
}
\end{equation*}

Now suppose that ${\bf v_1^u} = {\bf v_1^\ell} - {\bf e_k}$, where ${\bf e_k}$ denotes the dimension vector which is equal to one at the $k^{th}$ vertex of $Q$ and zero elsewhere. The following is the main result of the paper.
\begin{theorem}[Theorem \ref{thm:miniscule-flow-hecke}]\label{thm:hecke-intro}
$\mathcal{M}_{{\bf v_1^u}, {\bf v_1^\ell}}$ is the Hecke correspondence.
\end{theorem}

In the case where ${\bf v_1^u} < {\bf v_1^\ell} - {\bf e_k}$, one has to distinguish between broken and unbroken flow lines. For the Yang-Mills-Higgs flow on the space of Higgs bundles, this was done in \cite{wilkin-YMH-flow-lines} in terms of certain secant varieties of the space of Hecke modifications inside the negative slice at a critical point. For the Yang-Mills flow, this involves studying secant varieties of the projectivisation of the underlying bundle inside a space of bundle extensions, and for the Yang-Mills-Higgs flow the corresponding picture involves studying secant varieties of the spectral curve.

To distinguish between broken and unbroken flow lines in the space of representations of quivers, the analogous idea also involves secant varieties of the space of Hecke modifications, and it is natural to ask whether one can use algebraic methods to explicitly describe the compactification of the space of unbroken flow lines by broken flow lines. In \cite{wilkin-quiver-broken-flow-lines} we further develop the methods of this paper to answer this question.

\subsection{Organisation of the paper}
Section \ref{sec:smooth-space} contains the background theory for the properties of the norm-square of the moment map on the vector space of complex representations of a quiver. In Section \ref{sec:singular-space} we show how these properties restrict to a singular subset invariant under the group action and define the moduli spaces of flow lines. Section \ref{sec:gradient-Nakajima} contains the main results of the paper leading to the classification of critical points connected by flow lines in terms of the Hecke correspondence.

{\bf Acknowledgements.} The author would like to thank George Daskalopoulos, Richard Wentworth and Matthew Young for useful discussions and Hiraku Nakajima for pointing out the stability parameter from Definition \ref{def:aasp}.

\section{Background results for quivers without relations}\label{sec:smooth-space}

This section contains the basic results and notational setup used in the rest of the paper. The goal is to study the gradient flow, the structure of the critical sets and the eigenspaces of the Hessian for the function $\| \mu - \alpha \|^2$ on the vector space $\Rep(Q, {\bf v})$ (where we can apply theorems for smooth manifolds) before restricting to the singular subvariety associated to a quiver with relations in Section \ref{sec:singular-space}.  In Sections \ref{subsec:quiver-background}--\ref{subsec:symplectic} we set up the notation and summarise known results used in the rest of the paper, in Section \ref{subsec:critical-sets} we derive some useful formulae for the gradient flow on the space of metrics and in Section \ref{subsec:hessian} we prove results about the Hessian of $\| \mu - \alpha \|^2$ at a critical point and construct the unstable bundle and negative slice bundle over a critical set.

\subsection{Quiver varieties}\label{subsec:quiver-background}

\begin{definition}
A \emph{quiver} $Q$ is a directed graph, consisting of vertices $\mathcal{I}$, edges $\mathcal{E}$, and head/tail maps $h,t : \mathcal{E} \rightarrow \mathcal{I}$. 

A \emph{complex representation of a quiver} consists of a collection of complex vector spaces $\{ V_i \}_{i \in \mathcal{I}}$, and $\C$-linear homomorphisms $\{ x_a : V_{t(a)} \rightarrow V_{h(a)} \}_{a \in \mathcal{E}}$. The \emph{dimension vector} of a representation is the vector ${\bf v} := ( \dim_\C V_i )_{i \in \mathcal{I}} \in \mathbb{Z}_{\geq 0}^\mathcal{I}$. The vector space of all representations with fixed dimension vector is denoted 
\begin{equation*}
\Rep(Q, {\bf v}) := \bigoplus_{a \in \mathcal{E}} \Hom(V_{t(a)}, V_{h(a)}) .
\end{equation*}
\end{definition}

The group 
\begin{equation}\label{eqn:reductive-group-def}
G_{\bf v} := \prod_{i \in \mathcal{I}} \GL(V_i, \C)
\end{equation}
acts on the space $\Rep(Q, {\bf v})$ via the induced action on each factor $\Hom(V_{t(a)}, V_{h(a)})$ 
\begin{equation}\label{eqn:group-action}
( g_i )_{i \in \mathcal{I}} \cdot (x_a)_{a \in \mathcal{E}} := \left( g_{h(a)} x_a g_{t(a)}^{-1} \right)_{a \in \mathcal{E}}  
\end{equation}

The infinitesimal action of the Lie algebra $\mathfrak{g}_{\bf v}$ at a representation $x \in \Rep(Q, {\bf v})$ is denoted $\rho_x^\C : \mathfrak{g}_{\bf v} \rightarrow T_x \Rep(Q, {\bf v}) \cong \Rep(Q, {\bf v})$. A calculation shows that
\begin{equation}\label{eqn:complex-inf-action}
\rho_x^\C (u) := \left. \frac{d}{dt} \right|_{t=0} e^{tu} \cdot x = \bigoplus_{a \in \mathcal{E}} \left( u_{h(a)} x_a - x_a u_{t(a)} \right) \in \bigoplus_{a \in \mathcal{E}} \Hom(V_{t(a)}, V_{h(a)}).
\end{equation}

The direct sum of all the vector spaces is denoted
\begin{equation*}
\Vect(Q, {\bf v}) := \bigoplus_{i \in \mathcal{I}} V_i .
\end{equation*}
Given a representation $x \in \Rep(Q, {\bf v})$, we can consider each component $x_a$ as a homomorphism $\Vect(Q, {\bf v}) \rightarrow \Vect(Q, {\bf v})$ via the inclusion $\Hom(V_{t(a)}, V_{h(a)}) \subseteq \End \left( \Vect(Q, {\bf v}) \right)$.

There is a notion of slope-stability for quivers introduced by King in \cite{King94}, which corresponds to the usual definition of stability from GIT. Recall from \cite[Lemma 2.2]{King94} that GIT-stability on $\Rep(Q, {\bf v})$ is equivalent to defining a lift of the $G_{\bf v}$-action to a line bundle over $\Rep(Q, {\bf v})$. In contrast to the case of GIT on a projective variety (where the line bundle is determined by the projective embedding), in this case the line bundle is the trivial bundle $\Rep(Q, {\bf v}) \times \C$, and the lift of the action is determined by the choice of a stability parameter.

Given $\alpha = (\alpha_i)_{i \in \mathcal{I}} \in \mathbb{Z}^\mathcal{I}$, define the lift of the $G_{\bf v}$-action to $\Rep(Q, {\bf v}) \times \C$ by
\begin{equation}\label{eqn:lifted-action}
g \cdot (x, \xi) := \left( g \cdot x, \chi_\alpha(g) \xi \right) ,
\end{equation}
where the character $\chi_\alpha : G_{\bf v} \rightarrow \C$ is defined to be
\begin{equation*}
\chi_\alpha(g) = \prod_{i \in \mathcal{I}} \left( \det g_i \right)^{\alpha_i} .
\end{equation*}

\begin{definition}
An \emph{admissible stability parameter} for $\Rep(Q, {\bf v})$ is a choice of $\alpha = (\alpha_i)_{i \in \mathcal{I}} \in \mathbb{Z}^\mathcal{I}$ such that
\begin{equation*}
\sum_{i \in \mathcal{I}} \alpha_i v_i = 0 .
\end{equation*}
\end{definition}

\begin{remark}
The subgroup $\{ (\lambda \cdot \id_{V_i})_{i \in \mathcal{I}} \, : \, \lambda \in \C^* \} \subseteq G_{\bf v}$ acts trivially on $\Rep(Q, {\bf v})$. An equivalent definition of admissibility is that $\alpha$ is an admissible stability parameter if and only if the subgroup of scalar multiples of the identity in $G_{\bf v}$ also acts trivially on the line bundle $\Rep(Q, {\bf v}) \times \C$ via \eqref{eqn:lifted-action}. This is essential for the definition of stability in Definition \ref{def:stability}, since all points would be unstable if the parameter is not admissible. 
\end{remark}

The definition of GIT stability and semistability with respect to an admissible stability parameter $\alpha$ is then the usual one (first described for representations of quivers in \cite{King94}), which we recall in the following.

\begin{definition}\label{def:stability}
A representation $x \in \Rep(Q, {\bf v})$ is \emph{$\alpha$-semistable} if, for all nonzero $\xi \in \C$, the closure of the $G_{\bf v}$-orbit of $(x, \xi)$ in the trivial line bundle $\Rep(Q, {\bf v}) \times \C$ does not intersect the zero section, i.e.
\begin{equation*}
\overline{G_{\bf v} \cdot (x, \xi)} \cap \left( \Rep(Q, {\bf v}) \times \{ 0 \} \right) = \emptyset .
\end{equation*}

A representation $x \in \Rep(Q, {\bf v})$ is \emph{$\alpha$-polystable} if $x$ is $\alpha$-semistable and the $G_{\bf v}$-orbit of $(x, \xi)$ in $\Rep(Q, {\bf v})$ is closed for all nonzero $\xi \in \C$, .

A representation $x \in \Rep(Q, {\bf v})$ is \emph{$\alpha$-stable} if $x$ is $\alpha$-polystable and the isotropy group of $x$ in $G_{\bf v}$ consists only of the scalar multiples of the identity.
\end{definition}

The space of $\alpha$-stable (respectively $\alpha$-semistable and $\alpha$-polystable) representations is denoted $\Rep(Q, {\bf v})^{\alpha-st}$ (respectively $\Rep(Q, {\bf v})^{\alpha-ss}$ and $\Rep(Q, {\bf v})^{\alpha-polyst}$).

\begin{definition}
The \emph{GIT quotient} of $\Rep(Q, {\bf v})$ by $G_{\bf v}$ with respect to the stability parameter $\alpha$ is
\begin{equation*}
\mathcal{M}_\alpha(Q, {\bf v}) = \Rep(Q, {\bf v}) \doubleslash_\alpha G_{\bf v} := \Rep(Q, {\bf v})^{\alpha-ss} \doubleslash G_{\bf v} = \Rep(Q, {\bf v})^{\alpha-polyst} / G_{\bf v},
\end{equation*}
where the quotient $\doubleslash$ identifies $S$-equivalent orbits (those whose closures intersect) in the usual way.
\end{definition}

\begin{remark}
It is sometimes more convenient to divide out by the scalar multiples of the identity (which act trivially) and use the projectivisation $P G_{\bf v}$ instead. The quotients $\Rep(Q, {\bf v})^{\alpha-ss} \doubleslash G_{\bf v}$ and $\Rep(Q, {\bf v})^{\alpha-ss} \doubleslash PG_{\bf v}$ have the same underlying topological space, although when computing the equivariant cohomology of $\Rep(Q, {\bf v})^{\alpha-ss}$ with respect to $G_{\bf v}$ one has to remember the extra factor of $\C^*$ that acts trivially.
\end{remark}

When $\alpha = 0$, then the lift of the $G_{\bf v}$ action to $\Rep(Q, {\bf v}) \times \C$ is the trivial one, hence all representations $x \in \Rep(Q, {\bf v})$ are semistable. Therefore, in this case the GIT quotient $\mathcal{M}_0(Q, {\bf v})$ is just the affine quotient $\Rep(Q, {\bf v}) \doubleslash G_{\bf v}$. Every $G_{\bf v}$ orbit in $\Rep(Q, {\bf v})$ has a unique closed orbit in its closure (see \cite[Theorem 4, p19]{Nagata65} and \cite[Sec. 8]{Nagata64}), and the points in the affine quotient correspond to these closed orbits. Therefore there is a well-defined projection map
\begin{equation}\label{eqn:affine-projection}
\pi : \mathcal{M}_\alpha(Q, {\bf v}) \rightarrow \mathcal{M}_0(Q, {\bf v})
\end{equation}
taking an orbit to the unique closed orbit in its closure (where we take the closure in $\Rep(Q, {\bf v})$).

In analogy with holomorphic bundles, one can also define slope-stability of a representation in terms of the degree and rank (cf. \cite{King94}). 

\begin{definition}
A \emph{subrepresentation} of a representation $x \in \Rep(Q, {\bf v})$ consists of vector spaces $\{ V_i' \subseteq V_i \}_{i \in \mathcal{I}}$ such that $x_a (V_{t(a)}') \subseteq V_{h(a)}'$ for all edges $a \in \mathcal{E}$, and homomorphisms $\{ x_a' : V_{t(a)}' \rightarrow V_{h(a)}' \}_{a \in \mathcal{E}}$ such that $x_a'$ is the restriction of $x_a$ to $V_{t(a)}'$ for all $a \in \mathcal{E}$. 
\end{definition}

For a given subrepresentation $x'$ of $x \in \Rep(Q, {\bf v})$, let ${\bf v}' := \left( \dim_\C V_i' \right)_{i \in \mathcal{I}}$ be the associated dimension vector. Then $x' \in \Rep(Q, {\bf v}') \subseteq \Rep(Q, {\bf v})$.  We can now define the degree and rank of a subrepresentation.

\begin{definition}
Let $Q$ be a quiver, $\alpha = ( \alpha_i )_{i \in \mathcal{I}}$ an admissible stability parameter, and ${\bf v}' = (v_i)_{i \in \mathcal{I}} \in \mathbb{Z}_{\geq 0}^\mathcal{I}$ a dimension vector. The \emph{$\alpha$-degree} of $(Q, {\bf v}')$ is
\begin{equation*}
\deg_\alpha(Q, {\bf v}') : = \sum_{i \in \mathcal{I}} \alpha_i v_i ,
\end{equation*}
and the \emph{rank} is
\begin{equation*}
\rank(Q, {\bf v}') := \sum_{i \in \mathcal{I}} v_i .
\end{equation*}
The \emph{$\alpha$-slope} of $(Q, {\bf v}')$ is 
\begin{equation*}
\slope_\alpha(Q, {\bf v}') := \deg_\alpha(Q, {\bf v}') / \rank(Q, {\bf v}') .
\end{equation*}
\end{definition}

\begin{remark}
The stability parameter $\alpha$ is admissible for $\Rep(Q, {\bf v})$ if and only if $\deg_\alpha(Q, {\bf v}) = 0$.
\end{remark}

The following theorem of King then shows that, in analogy with holomorphic bundles, $\alpha$-stability and $\alpha$-semistability have an interpretation in terms of the slopes of subrepresentations.
\begin{proposition}[Proposition 3.1 of \cite{King94}]\label{prop:slope-stability}
Let $Q$ be a quiver, ${\bf v}$ a dimension vector, and $\alpha$ an admissible stability parameter. A representation $x \in \Rep(Q, {\bf v})$ is $\alpha$-stable (resp. $\alpha$-semistable) if and only if every proper non-zero subrepresentation satisfies
\begin{equation*}
\slope_\alpha(Q, {\bf v}') < 0 \quad \text{(respectively, $\slope_\alpha(Q, {\bf v}') \leq 0$)} .
\end{equation*}
\end{proposition}

When classifying the critical sets of $\| \mu-\alpha\|^2$ in Sections \ref{subsec:critical-sets} and \ref{sec:singular-critical-sets} it is necessary to choose a stability parameter for a given subrepresentation. In general it is not possible to use the same stability parameter $\alpha$, since $\deg_\alpha(Q, {\bf v'})$ may not be zero, and therefore $\alpha$ may not be admissible for $(Q, {\bf v'})$. Instead, the correct definition involves subtracting a scalar multiple of the vector $(1)_{j \in \mathcal{I}}$, where the scalar is chosen so that $(Q, {\bf v'})$ has degree zero with respect to the new parameter.

\begin{definition}\label{def:induced-parameter}
Let $Q$ be a quiver, ${\bf v}$ a dimension vector, and $\alpha = (\alpha_j)_{j \in \mathcal{I}}$ an admissible stability parameter for $(Q, {\bf v})$. Given any dimension vector ${\bf v'} \leq {\bf v}$, the \emph{induced stability parameter} on $(Q, {\bf v'})$ is 
\begin{equation*}
\alpha' = \left( \alpha_j - \slope_\alpha(Q, {\bf v'}) \right)_{j \in \mathcal{I}} .
\end{equation*}

\end{definition}

Note that it is easy to see that the induced stability parameter is admissible on $\Rep(Q, {\bf v'})$, since $\deg_{\alpha'}(Q, {\bf v'}) = \deg_\alpha(Q, {\bf v'}) - \deg_\alpha(Q, {\bf v'}) = 0$. Finally, we show that the operation of taking Hermitian adjoint of a representation preserves stability with respect to a change in parameter from $\alpha$ to $-\alpha$.
\begin{lemma}\label{lem:adjoint-stability}
Let $Q$ be a quiver and let $\bar{Q}$ denote the quiver with the same vertices, but with the direction of all edges reversed. Fix Hermitian structures on the vector spaces $\{ V_k \}_{k \in \mathcal{I}}$. Then $x \in \Rep(Q, {\bf v})$ is $\alpha$-stable (resp. semistable, polystable) if and only if the adjoint $x^* \in \Rep(\bar{Q}, {\bf v})$ is $-\alpha$-stable (resp. semistable, polystable).
\end{lemma}

\begin{proof}
Suppose that there is a subrepresentation with dimension vector ${\bf v'}$ preserved by $x^*$ and let $\mu = \slope_{-\alpha}(\bar{Q}, {\bf v'}) = - \slope_\alpha(Q, {\bf v'})$. Then the orthogonal complement of the subrepresentation is preserved by $x$ and so $\slope_\alpha(Q, {\bf v}-{\bf v'}) < 0$ (resp. $\leq 0$) since $x$ is $\alpha$-stable (resp. semistable). Therefore, since $\alpha$ is an admissible stability parameter, then $\slope_\alpha(Q, {\bf v'}) > 0$ (resp. $\geq 0$) and so $\slope_{- \alpha}(\bar{Q}, {\bf v'}) < 0$ (resp. $\leq 0$).  Therefore $x$ is $\alpha$-stable (resp. semistable) if and only if $x^*$ is $-\alpha$-stable (resp. semistable). 

Since $x$ is a direct sum of subrepresentations if and only if the adjoint $x^*$ is also a direct sum, then the above argument shows that $x$ is $\alpha$-polystable iff $x^*$ is $-\alpha$-polystable.
\end{proof}

\subsection{The algebraic stratification}\label{subsec:algebraic-strata}

The Harder-Narasimhan stratification for quivers is defined in analogy with the case of holomorphic bundles (see \cite{AtiyahBott83} and \cite{HarderNarasimhan74} for holomorphic bundles, and \cite[Section 2]{Reineke03} for quivers). The filtration is denoted by the sequence
\begin{equation}\label{eqn:HN-rep-filtration}
0 = x_0 \subset x_1 \subset \cdots \subset x_n = x 
\end{equation}
of subrepresentations such that for each $j=1, \ldots, n$, the quotient $x_j / x_{j-1}$ is the maximal semistable subrepresentation of $x / x_{j-1}$ (where the stability parameter is the one induced on the quotient using Definition \ref{def:induced-parameter}). The associated dimension vectors induce a canonical filtration 
\begin{equation}\label{eqn:HN-filtration}
\{ 0 \} = \Vect(Q, {\bf v_0}) \subset \Vect(Q, {\bf v_1}) \subset \cdots \subset \Vect(Q, {\bf v_n}) = \Vect(Q, {\bf v})
\end{equation}
called the \emph{Harder-Narasimhan filtration}, and the dimension vectors ${\bf v^*} = ({\bf v_1}, {\bf v_2} - {\bf v_1}, \ldots, {\bf v_n} - {\bf v_{n-1}})$ form a vector called the \emph{Harder-Narasimhan type} of the filtration. Note that the inclusion maps in \eqref{eqn:HN-filtration} are induced from the representation $x$, so that the spaces $\Vect(Q, {\bf v_j}) \subseteq \Vect(Q, {\bf v})$ are all $x$-invariant.

\begin{definition}\label{def:HN-length}
The \emph{length} of the Harder-Narasimhan filtration \eqref{eqn:HN-filtration} is equal to $n$, the number of non-trivial terms in the filtration.
\end{definition}

\begin{definition}\label{def:HN-strata}
The \emph{Harder-Narasimhan stratum} with Harder-Narasimhan type ${\bf v^*}$ is
\begin{equation}\label{eqn:HN-strata-def}
B_{\bf v^*} := \left\{ x \in \Rep(Q, {\bf v}) \, : \, \text{$x$ has H-N type ${\bf v^*}$} \right\} \subseteq \Rep(Q, {\bf v}).
\end{equation}
\end{definition}

Since the filtration is canonical then a representation belongs to exactly one Harder-Narasimhan stratum, and so we have a disjoint union
\begin{equation*}
\Rep(Q, {\bf v}) = \bigcup_{\text{HN types }{\bf v^*}}  B_{\bf v^*} .
\end{equation*}

There is a partial ordering on the strata given in \cite[Definition 3.6]{Reineke03} (analogous to that for holomorphic bundles described by Shatz in \cite{Shatz77}), and \cite[Proposition 3.7]{Reineke03} shows that the stratification has good properties in the sense that the closure of each stratum $B_{\bf v^*}$ is contained in the union of all $B_{\bf w^*}$ such that ${\bf w^*} \geq {\bf v^*}$. 

Any semistable representation also has a Jordan-H\"older filtration, given by the following
\begin{definition}
Let $x \in \Rep(Q, {\bf v})^{\alpha-ss}$ be an $\alpha$-semistable representation. A filtration 
\begin{equation*}
\{0\} = \Vect(Q, {\bf v_0}) \subset \Vect(Q, {\bf v_1}) \subset \cdots \subset \Vect(Q, {\bf v_m}) = \Vect(Q, {\bf v})
\end{equation*}
with induced subrepresentations of $x$
\begin{equation*}
0 = x_0 \subset x_1 \subset \cdots \subset x_m = x ,
\end{equation*}
is called a \emph{Jordan-H\"older filtration} if each quotient representation $x_j / x_{j-1}$ is stable with respect to the stability parameter on $\Rep(Q, {\bf v_j}-{\bf v_{j-1}})$ induced by $\alpha$, and each subrepresentation has the same slope.
\end{definition}

In contrast to the Harder-Narasimhan filtration, the Jordan-H\"older filtration is not necessarily unique, but the \emph{graded object}
\begin{equation*}
\Gr^{\JH}(x) = \bigoplus_{j=1}^m x_j / x_{j-1}
\end{equation*}
is unique up to isomorphism. Combining the Harder-Narasimhan filtration with the Jordan-H\"older filtration, for any representation $x \in \Rep(Q, {\bf v})$ we obtain a double filtration called the \emph{Harder-Narasimhan-Jordan-H\"older filtration} (cf. \cite[Sec. 5]{HaradaWilkin11} for quivers and \cite{DaskalWentworth04} for holomorphic bundles). Again, this is not necessarily unique, but the graded object $\Gr^{\HNJH}(x)$ is unique up to isomorphism.

\subsection{The symplectic quotient}\label{subsec:symplectic}

Another theorem of King (\cite[Theorem 6.1]{King94}) identifies the GIT quotient of $\Rep(Q, {\bf v})$ with the symplectic quotient. Since this equivalence is central to this paper, then we recall the details here.

Let $Q$ be a quiver with dimension vector ${\bf v} = ( v_i )_{i \in \mathcal{I}}$, and fix a Hermitian structure on the vector spaces $V_i \cong \C^{v_i}$. There is an associated symplectic structure on $\Rep(Q, {\bf v})$, defined as follows. Given tangent vectors $\delta x_1, \delta x_2 \in T_x \Rep(Q, {\bf v}) \cong \Rep(Q, {\bf v})$, define the metric
\begin{equation}\label{eqn:metric}
g(\delta x_1, \delta x_2) := \sum_{a \in \mathcal{E}} \Re \tr \left( (\delta x_1)_a (\delta x_2)_a^* \right) ,
\end{equation}
and symplectic structure
\begin{equation}\label{eqn:symplectic-structure}
\omega(\delta x_1, \delta x_2) := \sum_{a \in \mathcal{E}} \Im \tr \left( (\delta x_1)_a (\delta x_2)_a^* \right) .
\end{equation}
Note that $\omega(\delta x_1, \delta x_2) = g(i \delta x_1, \delta x_2)$, in other words the complex structure $I = i \cdot \id$ is compatible with the metric. With this complex structure and metric, the space $\Rep(Q, {\bf v})$ has the structure of a K\"ahler manifold.

With respect to the Hermitian structure on each $V_i$, one can define the unitary group $\U(V_i) \subset \GL(V_i, \C)$, and therefore the compact subgroup
\begin{equation*}
K_{\bf v} := \prod_{i \in \mathcal{I}} \U(V_i) \subset G_{\bf v}.
\end{equation*}

The induced action of $K_{\bf v}$ on $\Rep(Q, {\bf v})$ is given by
\begin{equation*}
(g_j)_{j \in \mathcal{I}} \cdot (x_a)_{a \in \mathcal{E}} = \left( g_{h(a)} x_a g_{t(a)}^{-1} \right)_{a \in \mathcal{E}} ,
\end{equation*}
and the infinitesimal action of the Lie algebra $\mathfrak{k}_{\bf v}$, denoted $\rho_x : \mathfrak{k}_{\bf v} \rightarrow T_x \Rep(Q, {\bf v}) \cong \Rep(Q, {\bf v})$, is given by
\begin{equation}\label{eqn:inf-action}
\rho_x(u) := \left. \frac{d}{dt} \right|_{t = 0} e^{tu} \cdot x = \bigoplus_{a \in \mathcal{E}} \left( u_{h(a)} x_a - x_a u_{t(a)}  \right) .
\end{equation}

This action is \emph{Hamiltonian}, i.e. it preserves the symplectic structure and has an associated moment map
\begin{align}\label{eqn:moment-map-def}
\begin{split}
\mu : \Rep(Q, {\bf v}) & \rightarrow \mathfrak{k}_{\bf v}^* \\
 (x_a)_{a \in \mathcal{E}} & \mapsto \frac{1}{2i} \sum_{a \in E} [x_a, x_a^*] 
\end{split}
\end{align}
that satisfies $d \mu (\delta x) \cdot u = \omega(\rho_x(u), \delta x)$ for all $\delta x \in \Rep(Q, {\bf v}) \cong T_x \Rep(Q, {\bf v})$, and $u \in \mathfrak{k}_{\bf v}$. In the above definition the commutator $[x_a, x_a^*]$ is defined via the inclusion $\Hom(V_{t(a)}, V_{h(a)}) \hookrightarrow \End \left( \Vect(Q, {\bf v}) \right)$. In the following we fix an $\Ad$-invariant inner product on $\mathfrak{k}_{\bf v}$ and use this to identify $\mathfrak{k}_{\bf v} \cong \mathfrak{k}_{\bf v}^*$, so that we can consider $\mu$ as a map into $\mathfrak{k}_{\bf v}$. 

Note also that the above definition implies that $\tr \mu(x) = 0$, since $\mu(x)$ is constructed from commutators. Therefore, for the symplectic quotient to make sense, we need the following definition.

\begin{definition}
Let $Q$ be a quiver, and ${\bf v} = (v_j)_{j \in \mathcal{I}} \in \mathbb{Z}_{\geq 0}^\mathcal{I}$ a dimension vector. The central element $\alpha = (i \alpha_j \cdot \id_{V_j})_{j \in \mathcal{I}} \in Z(\mathfrak{k}^*)$ is an \emph{admissible central element} if $\displaystyle{\sum_{j \in \mathcal{I}} \alpha_j v_j = 0}$. 
\end{definition}

The \emph{symplectic quotient} with respect to an admissible central element $\alpha$ is
\begin{equation*}
\mathcal{M}_\alpha(Q, {\bf v}) := \mu^{-1}(\alpha) / K_{\bf v} .
\end{equation*}

\begin{remark}

\begin{enumerate}

\item The parameter $\alpha$ is admissible if and only if $\alpha$ is a central element of the dual of the Lie algebra of $PK$. 

\item Given an admissible stability parameter $(\alpha_j)_{j \in \mathcal{I}} \in \mathbb{Z}^{\mathcal{I}}$ one can construct an admissible central element $(i \alpha_j \cdot \id_{V_j})_{j \in \mathcal{I}} \in Z(\mathfrak{k}^*)$ and vice-versa. In the rest of the paper both of these will be denoted $\alpha$, and the meaning will be clear from the context.
\end{enumerate}
\end{remark}

A result of King from \cite{King94} shows that the GIT quotient and the symplectic quotient are bijective and that there is a continuous map $\mu^{-1}(\alpha) / K_{\bf v} \rightarrow \Rep(Q, {\bf v})^{\alpha-ss} \doubleslash G_{\bf v}$. Using the gradient flow of $\| \mu - \alpha \|^2$, Hoskins has shown in \cite[Theorem 4.2]{Hoskins14} that the inverse of this map is also continuous.

\begin{proposition}\label{prop:King-Kempf-Ness}
Let $Q$ be a quiver, ${\bf v}$ a dimension vector, and $\alpha$ an admissible stability parameter. Then the GIT quotient $\Rep(Q, {\bf v})^{\alpha-ss} \doubleslash G_{\bf v}$ is homeomorphic to the symplectic quotient $\mu^{-1}(\alpha) / K_{\bf v}$.
\end{proposition}

\subsection{Critical points of $\| \mu - \alpha \|^2$ and the gradient flow} \label{subsec:critical-sets}

Recall from \cite[Sec. 3.2]{HaradaWilkin11} that a representation $x \in \Rep(Q, {\bf v})$ is a critical point for $\| \mu - \alpha \|^2$ if and only if the infinitesimal action of $K_{\bf v}$ satisfies
\begin{equation}\label{eqn:critical-rep}
\rho_x (\mu(x)-\alpha) = 0   \quad \text{for all $a \in \mathcal{E}$} .
\end{equation}
More explicitly, this is equivalent to the condition that
\begin{equation}
\quad x_a (\mu(x)-\alpha)_{t(a)} - (\mu(x)-\alpha)_{h(a)} x_a = 0 \quad \text{for all $a \in \mathcal{E}$} .
\end{equation}

This equation implies that the representation $x$ splits into subrepresentations, each of which corresponds to an eigenspace of $i(\mu(x)-\alpha)$ (the factor of $i$ is used so that the eigenvalues are real; see \eqref{eqn:critical-eigenvalues} below). In other words, if $\lambda_1, \ldots, \lambda_n$ are the eigenvalues of $i(\mu(x)-\alpha)$, then for each eigenvalue $\lambda_j$ there exists a dimension vector ${\bf v_j}$ such that ${\bf v_1} + \cdots + {\bf v_n} = {\bf v}$, and 
\begin{equation}\label{eqn:critical-splitting}
x = \bigoplus_{j=1}^n x_j, \quad \Vect(Q, {\bf v}) \cong \bigoplus_{j=1}^n \Vect(Q, {\bf v_j})
\end{equation}
where $x_j \in \Rep(Q, {\bf v_j})$ for each $j$. Since $\mu(x)$ is constructed from commutators, then $\tr \mu(x) = 0$ on each subrepresentation, and therefore taking the trace of $i(\mu(x) - \alpha)$ shows that
\begin{equation}\label{eqn:critical-eigenvalues}
\lambda_j = \slope_\alpha(Q, {\bf v_j}) .
\end{equation}
Moreover, restricting to a subrepresentation with dimension vector ${\bf v_j}$ induces a new stability parameter $\alpha_j$ on $\Rep(Q, {\bf v_j)}$ (see Definition \ref{def:induced-parameter}), and a direct sum of representations such as that described in \eqref{eqn:critical-splitting} is critical if and only if each $x_j$ is a minimum for $\| \mu - \alpha \|^2$ on $\Rep(Q, {\bf v_j})$. See \cite[Proposition 1]{HaradaWilkin11} for more details.

For each critical set there is a corresponding decomposition ${\bf v} = {\bf v_1} + \cdots + {\bf v_n}$. The \emph{critical type} of a critical point is the vector ${\bf v^*} = \left( {\bf v_1}, \ldots, {\bf v_n} \right)$, where the dimension vectors are ordered by decreasing slope, i.e. $\slope_\alpha(Q, {\bf v_i}) > \slope_\alpha(Q, {\bf v_j})$ if and only if $i < j$. The set of all critical points with critical type ${\bf v^*}$ is denoted $C_{\bf v^*}$. 

In terms of the infinitesimal action of $G_{\bf v}$ on $\Rep(Q, {\bf v})$, the gradient of $f(x) = \| \mu(x) - \alpha \|^2$ has the form
\begin{equation}\label{eqn:grad-def}
\grad f(x) = I \rho_x(\mu(x) - \alpha)
\end{equation}

\begin{definition}
Let $\phi(x_0, t) \in \Rep(Q, {\bf v})$ denote the solution to the downwards gradient flow equation
\begin{equation}\label{eqn:neg-grad-flow}
\frac{d}{dt} \phi(x_0, t) = - I \rho_{\phi(x_0, t)}(\mu(\phi(x_0, t))) - \alpha)
\end{equation}
with initial condition $\phi(x_0, 0) := x_0 \in \Rep(Q, {\bf v})$.
\end{definition}
The $K_{\bf v}$-equivariance of a solution to \eqref{eqn:neg-grad-flow} follows immediately from the $K_{\bf v}$-equivariance of the moment map. In the same way as for moment map flows on K\"ahler manifolds studied by Kirwan in \cite{Kirwan84}, the flow is generated by the action of $G_{\bf v}$. Given $x_0 \in \Rep(Q, {\bf v})$, let $g(t)$ be the solution of
\begin{equation}\label{eqn:group-neg-flow}
\frac{dg}{dt} g(t)^{-1} = -i (\mu(g(t) \cdot x_0) - \alpha), \quad g(0) = \id
\end{equation}
A calculation then shows that $\phi(x_0, t) = g(t) \cdot x_0$ satisfies \eqref{eqn:neg-grad-flow}. 

It follows from Sjamaar's compactness result in \cite[Lem. 4.10]{Sjamaar98} and the Lojasiewicz inequality technique of Simon in \cite{Simon83} that the flow $\phi(x_0, t)$ exists for all time $t \geq 0$ and converges to a unique limit $\displaystyle{x_\infty = \lim_{t \rightarrow \infty} \phi(x_0, t)}$. The main theorem of \cite{HaradaWilkin11} gives an algebraic description of the limit of the downwards gradient flow of $\| \mu - \alpha \|^2$ (see also \cite{Hoskins14}) and the main result of this paper is to give conditions on $x_0$ such that $\lim_{t \rightarrow - \infty} \phi(x_0, t)$ exists. This will be used in Section \ref{subsec:flow-classification} to characterise the pairs of critical points connected by a flow line.

\begin{theorem}\label{thm:algebraic-flow-limit}
Let $Q$ be a quiver and $\alpha$ a stability parameter for $Q$. Given a dimension vector ${\bf v}$ for $Q$, let $x \in \Rep(Q, {\bf v})$. Then
\begin{enumerate}

\item (\cite[Theorem 8, p336]{HaradaWilkin11}) The limit $x_\infty := \lim_{t \rightarrow \infty} \phi(x_0, t)$ is isomorphic to the graded object of the Harder-Narasimhan-Jordan-H\"older double filtration of $x$.

\item (\cite[Proposition 2, p320]{HaradaWilkin11}) The gradient flow defines a continuous $K_{\bf v}$-equivariant deformation retract of each Harder-Narasimhan stratum $B_{\bf v^*}$ onto the associated critical set $C_{\bf v^*}$.

\end{enumerate}
\end{theorem}

The remaining results of this section are related to the flow on $G_{\bf v} / K_{\bf v}$ induced by \eqref{eqn:group-neg-flow}. Let $H(n)$ denote the space of $n \times n$ Hermitian matrices and let $H(n)^+$ denote the subset of positive definite Hermitian matrices. Given a quiver $Q$ with set of vertices $\mathcal{I}$ and dimension vector ${\bf v} = (v_i)_{i \in \mathcal{I}}$, define $H_{\bf v} := \times_{i \in \mathcal{I}} H(v_i)$ and $H_{\bf v}^+ := \times_{i \in \mathcal{I}} H(v_i)^+$. Recall from \cite[Sec. VI.1]{Kobayashi87} that there is an identification $H_{\bf v}^+ \cong G_{\bf v} / K_{\bf v}$. Given two metrics $h_1, h_2 \in H_{\bf v}^+$, let $\lambda_{i, \ell}$ denote the eigenvalues of $h_1^{-1} h_2$ at vertex $i \in \mathcal{I}$ for $\ell = 1, \ldots, v_i$. Then the geodesic distance between $h_1$ and $h_2$ is (cf. \cite[Sec. VI.1]{Kobayashi87})
\begin{equation}\label{eqn:metric-distance}
d(h_1, h_2) = \sum_{i \in \mathcal{I}} \sum_{\ell = 1}^{v_i} \left( (\log \lambda_{i, \ell})^2 \right)^{\frac{1}{2}} .
\end{equation}

Define $\sigma : H_{\bf v}^+ \cong G_{\bf v} / K_{\bf v} \rightarrow \R_{\geq 0}$ by
\begin{equation}\label{eqn:def-sigma}
\sigma(h) = \tr(h) + \tr(h^{-1}) - 2 \rank(Q, {\bf v}) .
\end{equation}
In the sequel we will use the following two properties of $\sigma$

\begin{enumerate}

\item $\sigma(h) = 0$ if and only if $h = \id$, and

\item $\lim_{t \rightarrow \infty} \sigma(h_t h_\infty^{-1}) = 0$ if and only if $h_t \rightarrow h_\infty$ with respect to the metric \eqref{eqn:metric-distance}.

\end{enumerate}

The function $\sigma$ is more convenient than the distance measure on $G_{\bf v} / K_{\bf v}$ given by \eqref{eqn:metric-distance} since it has better properties with respect to the gradient flow \eqref{eqn:group-neg-flow}. This was first observed by Donaldson \cite{Donaldson85} in the context of the Yang-Mills flow on K\"ahler manifolds. Given any $x \in \Rep(Q, {\bf v})$ and any $g \in G_{\bf v}$, define $h = g^* g$ and 
\begin{equation}
\mu_h(x) = \Ad_{g^{-1}} \mu(g \cdot x)
\end{equation}
Note that the $K_{\bf v}$-equivariance of the moment map implies that $\mu_h(x)$ depends only on $h \in G_{\bf v} / K_{\bf v}$. Then the results of \cite[Sec. 3.3]{HaradaWilkin11} show that
\begin{align}\label{eqn:mu-h-inequality}
\begin{split}
- 2i \tr \left( (\mu_h(x) - \mu(x)) h \right) & \leq 0 \\
2i \tr \left( h^{-1} (\mu_h(x) - \mu(x)) \right) & \leq 0 
\end{split}
\end{align}
Let $x \in \Rep(Q, {\bf v})$ and $g_0 \in G_{\bf v}$. Then the flow \eqref{eqn:group-neg-flow} with initial conditions $x$ and $g_0 \cdot x$ has respective solutions $g_1(t)$ and $g_2(t)$. Define $g_t : = g_2(t) g_0 g_1(t)^{-1}$ and $h_t = g_t^* g_t$.  In analogy with the Yang-Mills flow studied in \cite{Donaldson85}, the result of \cite[Thm. 5]{HaradaWilkin11} shows that
\begin{equation}\label{eqn:distance-decreasing}
\frac{d}{dt} \sigma(h_t) \leq 0 .
\end{equation}

As explained in the proof of Corollary 15 in \cite{Donaldson85}, the geodesic distance $d$ in the homogeneous space $\GL(n, \C) / \U(n)$ compares uniformly with $\sigma$, in the sense that there exist continuous monotone functions $f, F : \R_{\geq 0} \rightarrow \R_{\geq 0}$ with $f(0) = F(0) = 0$ such that $d(\id, h) \leq F(\sigma(h))$ and $\sigma(h) \leq f(d(\id, h))$. In Section \ref{subsec:scattering} we need a more precise estimate, which is contained in the following lemma.

\begin{lemma}\label{lem:mu-h-Lipschitz}
Given $x \in \Rep(Q, {\bf v})$ and a bounded neighbourhood $U$ of $x$ there exists a bounded neighbourhood $V$ of the identity in $G_{\bf v} / K_{\bf v}$ and a positive constant $C$ such that
\begin{equation}\label{eqn:mu-h-Lipschitz}
\| \mu_h(y) - \mu(y) \| \leq C \sqrt{\sigma(h)}
\end{equation}
for all $y \in U$ and $h \in V$.
\end{lemma}

\begin{proof}
Since $\mu_h(x) : \Rep(Q, {\bf v}) \times G_{\bf v} / K_{\bf v} \rightarrow \mathfrak{g}_{\bf v}$ is a smooth function and the derivative with respect to $h$ is uniformly bounded in a bounded neighbourhood $U$ of $x$, then $\mu_h$ is uniformly Lipschitz in $h$ on $U$. Therefore $\| \mu_h(y) - \mu(y) \| \leq C' d(\id, h)$ for some constant $C'$, where $d$ denotes the geodesic distance in the homogeneous space $G_{\bf v} / K_{\bf v}$. Let $\{ \nu_k \} \subset \R_{>0}$ be the eigenvalues of $h$. Then the calculation in \cite[Ch. VI.1]{Kobayashi87} shows that
\begin{equation*}
d(\id, h) = \left( \sum_k (\log \nu_k)^2 \right)^{\frac{1}{2}} .
\end{equation*}
There is a neighbourhood $V$ of $\id$ in $G_{\bf v} / K_{\bf v}$ and a constant $K < \frac{1}{2}$ such that for all metrics $h$ in $V$ the eigenvalues satisfy $| \nu_k - 1 | \leq K$ for all $k$, and so there is a constant $C$ such that $| \log \nu_k | \leq C | \nu_k - 1 |$ for all $k$. Therefore there is a constant $C_1$ such that
\begin{equation}\label{eqn:distance-sigma1}
d(\id, h) \leq \sum_k  \left| \log \nu_k \right|  \leq \sum_k \left| \nu_k - 1 \right| \leq C_1 \sqrt{ \tr((h-\id)^2)}
\end{equation}
We also have 
\begin{equation}\label{eqn:distance-sigma2}
\sigma(h) = \tr(h + h^{-1} - 2 \id) = \tr\left( h^{-1} (h-\id)^2 \right) \geq C_ 2\tr ((h-\id)^2)
\end{equation}
for some constant $C_2$ since $h^{-1}$ is positive and bounded below in the given neighbourhood $V$ of the identity. Combining all of these estimates gives us a positive constant $C = \frac{C' C_1}{\sqrt{C_2}}$ such that
\begin{equation*}
\| \mu_h(y) - \mu(y) \| \leq C' d(\id, h) \leq C' C_1 \sqrt{\tr((h-\id)^2)} \leq C \sqrt{\sigma(h)}
\end{equation*}
for all $h \in V$.
\end{proof}

\subsection{The eigenspaces of the Hessian at a critical point}\label{subsec:hessian}

Recall from \eqref{eqn:critical-rep} that the critical point equation for $\| \mu - \alpha \|^2$ on $\Rep(Q, {\bf v})$ is
\begin{equation*}
I \rho_x(\mu(x) - \alpha) = 0 \quad \Leftrightarrow \quad \left[ (\mu(x) - \alpha), x_a \right] = 0 \quad \text{for all $a \in \mathcal{E}$} ,
\end{equation*}
and recall from \eqref{eqn:critical-splitting} that any representation satisfying this equation must split into the direct sum of subrepresentations
\begin{equation}\label{eqn:critical-decomposition}
x = x_1 \oplus \cdots \oplus x_n ,
\end{equation}
where each $x_j \in \Rep(Q, {\bf v_j})$, and ${\bf v_1} + \cdots + {\bf v_n} = {\bf v}$. In addition, each $x_j$ minimises the function $\| \mu - \alpha_j \|^2$ on $\Rep(Q, {\bf v_j})$, where $\alpha_j$ is the stability parameter on $\Rep(Q, {\bf v_j})$ induced from $\alpha$.

Also recall from \eqref{eqn:critical-eigenvalues} that $i (\mu(x_j) - \alpha) = \slope_\alpha(Q, {\bf v_j}) \cdot \id$ for each $j=1, \ldots, n$, and so $i( \mu(x) - \alpha)$ has the block-diagonal form
\begin{equation}\label{eqn:mom-map-block-diagonal}
i (\mu(x) - \alpha) = \left( \begin{matrix} \lambda_1 \cdot \id & 0 & 0 & \cdots & 0 \\ 0 & \lambda_2 \cdot \id & 0 & \cdots & 0 \\ 0 & 0 & \lambda_3 \cdot \id & \cdots & 0 \\ \vdots & \vdots & \vdots & \ddots & \vdots \\ 0 & 0 & 0 & \cdots & \lambda_n \cdot \id \end{matrix} \right) ,
\end{equation}
where $\lambda_j = \slope_\alpha(Q, {\bf v_j})$ for each $j = 1, \ldots, n$. The eigenvalues in \eqref{eqn:mom-map-block-diagonal} are ordered so that $\lambda_1 < \lambda_2 < \cdots < \lambda_n$ (i.e. the slope increases with $j$). 

\begin{definition}
The derivative of the infinitesimal action is
\begin{align}\label{eqn:deriv-inf-action}
\begin{split}
\delta \rho_x : \mathfrak{k} \times T_x \Rep(Q, {\bf v}) & \rightarrow T_x \Rep(Q, {\bf v}) \\
(u, X) & \mapsto \left. \frac{d}{dt} \right|_{t=0} \rho_{x+tX} (u) .
\end{split}
\end{align}
\end{definition}

An explicit formula for $\delta \rho$ is 
\begin{equation}\label{eqn:explicit-deriv-inf-action}
\delta \rho_x(u)(X) = \sum_{a \in \mathcal{E}} [u, X_a] .
\end{equation}

\begin{remark}
\begin{enumerate}
\item Note that the tangent bundle of $\Rep(Q, {\bf v})$ is trivial, and therefore we can use the trivial connection on $T \Rep(Q, {\bf v})$ in the above definition.

\item From the definition of the complex structure $I$ in \eqref{eqn:symplectic-structure} we have $\delta \rho_x(u)(IX) = I \delta \rho_x(u)(X)$.
\end{enumerate}
\end{remark}

\begin{lemma}
At a critical point $x \in \Rep(Q, {\bf v})$, the Hessian $H_f : T_x \Rep(Q, {\bf v}) \rightarrow T_x \Rep(Q, {\bf v})$ has the form
\begin{equation*}
H_f(\delta x) = -I \rho_x \rho_x^* I \delta x + I \delta \rho_x (\mu(x)) (\delta x) .
\end{equation*}
\end{lemma}

\begin{proof}
Recall that the gradient of $f = \frac{1}{2} \| \mu - \alpha \|^2$ at a representation $x \in \Rep(Q, {\bf v})$ is given by
\begin{equation*}
\grad f (x) = I \rho_x(\mu - \alpha) .
\end{equation*}

Differentiating the gradient of $f$ at $x \in \Rep(Q, {\bf v})$ in the direction of a tangent vector $X \in T_x \Rep(Q, {\bf v})$ gives us
\begin{align}\label{eqn:Hessian}
\begin{split}
H_f(x)(X) = \nabla_X \grad f(x) & = I \delta \rho_x(\mu(x)-\alpha)(X) + I \rho_x d\mu(X) \\
 & = I \delta \rho_x(\mu(x)-\alpha)(X) -  I \rho_x \rho_x^* I X , 
\end{split}
\end{align}
where again we use the trivial connection on the tangent bundle of $\Rep(Q, {\bf v})$.
\end{proof}

The next two lemmas contain some identities that will be useful in characterising the negative eigenspace of the Hessian. In order to be completely clear about the sign conventions then all the details are included.

\begin{lemma}
For any $x \in \Rep(Q, {\bf v})$ and $X \in T_x \Rep(Q, {\bf v})$ we have
\begin{equation}\label{eqn:rho-I-adjoint}
d\mu_x(X) = - \rho_x^* I X,
\end{equation}
where we use the inner product on $\mathfrak{k}_{\bf v}$ to identify $\mathfrak{k}_{\bf v} \cong \mathfrak{k}_{\bf v}^*$. For any $u \in \mathfrak{k}_{\bf v}$ we also have
\begin{equation}\label{eqn:rho-I-rho}
\rho_x^* I \rho_x (u) = [\mu(x), u] .
\end{equation}
\end{lemma}

\begin{proof}
Using the moment map equation, we know that
\begin{equation*}
d \mu(X) \cdot v = \omega(\rho_x(v), X) = g(I \rho_x(v), X) = -< v, \rho_x^* I X > 
\end{equation*}
for all $v \in \mathfrak{k}_{\bf v}$, where we use $d \mu(X) \cdot v$ to denote the dual pairing $\mathfrak{k}_{\bf v}^* \times \mathfrak{k}_{\bf v} \rightarrow \C$ and $< \cdot, \cdot>$ to denote the inner product on $\mathfrak{k}_{\bf v}$. This proves \eqref{eqn:rho-I-adjoint}. Setting $X = \rho_x(u)$ gives us
\begin{equation*}
-< v,  \rho_x^* I \rho_x(u)> =  d \mu(\rho_x(u)) \cdot v 
\end{equation*}
for all $v \in \mathfrak{k}_{\bf v}$, and so we can identify $\rho_x^* I \rho_x(u) = -d \mu(\rho_x(u))$ (where we use the inner product on $\mathfrak{k}_{\bf v}$ to identify $\mathfrak{k}_{\bf v}$ with $\mathfrak{k}_{\bf v}^*$). Equivariance of the moment map with respect to the action of $K$ gives us
\begin{equation*}
\mu(e^{tu} \cdot x) = e^{tu} \mu(x) e^{-tu} \quad \Rightarrow \quad d\mu(\rho_x(u)) = \left. \frac{d}{dt} \right|_{t=0} e^{tu} \mu(x) e^{-tu} = [u, \mu(x)] .
\end{equation*}
Therefore $\rho_x^* I \rho_x(u) = [\mu(x), u]$, as required.
\end{proof}

Differentiating this result at a critical point $x$ gives us
\begin{equation*}
(\delta \rho_x)^* (I \rho_x(u), X) + \rho_x^* I \delta \rho_x(u)(X)  = [ d\mu(X), u]
\end{equation*}

Using the fact that $I \delta \rho_x(u)(X) = \delta \rho_x(u)(IX)$ then gives us
\begin{equation*}
\rho_x^* \left( \delta \rho_x (u) (X) \right) = - \rho_x^* I \delta \rho_x(u)(IX) = [\rho_x^* X, u] - (\delta \rho_x)^* (I \rho_x(u), X).
\end{equation*}

Therefore, we have proven
\begin{lemma}
\begin{align}
\rho_x^* I \delta \rho_x(u)(X) & = -[\rho_x^* I X, u] \label{eqn:rho-I-delta} \\
\rho_x^* \left( \delta \rho_x (u) (X) \right) & = [\rho_x^* X, u] - (\delta \rho_x)^* (I \rho_x(u), X) \label{eqn:rho-delta}
\end{align}
\end{lemma}

The next lemma will be used in the proof of Lemma \ref{lem:complex-image-splits}.

\begin{lemma}\label{lem:inf-action-bracket}
Let $x$ be a critical point of $f(x) = \| \mu(x) - \alpha \|^2$. Then for any $v \in \mathfrak{k}_{\bf v}$ we have
\begin{equation}\label{eqn:inf-action-bracket}
\rho_x ([\mu(x), v]) = \delta \rho_x(\mu(x)-\alpha)(\rho_x(v)) .
\end{equation}
\end{lemma}

\begin{proof}

Firstly note that $[\mu(x), v] = [\mu(x) - \alpha, v]$ since $\alpha$ is central. Therefore
\begin{align*}
\rho_x([\mu(x) - \alpha,v]) & = \rho_x(\ad_{\mu(x)-\alpha}(v)) \\
 & = \left. \frac{\partial}{\partial t} \right|_{t=0} \rho_x(\Ad_{\exp(t(\mu(x)-\alpha))}(v)) \\
 & = \left. \frac{\partial^2}{\partial s \partial t} \right|_{s,t=0} \left( e^{t(\mu(x)-\alpha)} e^{sv} e^{-t(\mu(x)-\alpha)} \right) \cdot x \\
 & = \left. \frac{\partial}{\partial s} \right|_{s=0} \left( \rho_{e^{sv} \cdot x} (\mu(x) - \alpha) - e^{sv} \cdot \rho_x(\mu(x) - \alpha) \right) ,
\end{align*}
where $e^{sv} \cdot \rho_x(\mu(x) - \alpha)$ denotes the action of $e^{sv} \in K_{\bf v}$ on the tangent vector $\rho_x(\mu(x) - \alpha) \in T_x \Rep(Q, {\bf v})$, which maps it to an element of $T_{e^{sv} \cdot x} \Rep(Q, {\bf v})$. Since $x$ is a critical point then $\rho_x(\mu(x) - \alpha) = 0$ by \eqref{eqn:critical-rep}, and so the above equation simplifies to
\begin{align*}
\rho_x([\mu(x) - \alpha,v]) & = \left. \frac{\partial}{\partial s} \right|_{s=0} \rho_{e^{sv} \cdot x} (\mu(x) - \alpha) \\
 & = \delta \rho_x(\mu(x) - \alpha)(\rho_x(v))
\end{align*}
as required.
\end{proof}

Since $H_f$ is self-adjoint then the tangent space splits into the orthogonal direct sum of eigenspaces and each eigenvalue is real. The next lemma describes the negative eigenspace of the Hessian.

\begin{lemma}\label{lem:neg-eigenspace}
Let $x \in \Rep(Q, {\bf v})$ be a critical point of $f(x) = \frac{1}{2}\| \mu(x) - \alpha \|^2$, and let $X \in T_x \Rep(Q, {\bf v})$. Suppose that $H_f(X) = \lambda X$ for some $\lambda \neq 0$. Then $X \in \ker \rho_x^*$. Moreover, if $\lambda < 0$ then $X \in \ker \rho_x^* I$ and so the negative eigenspaces of the Hessian are orthogonal to the $G_{\bf v}$-orbit through $x$.
\end{lemma}

\begin{proof}
Since $f(x) = \| \mu(x) - \alpha \|^2$ is $K_{\bf v}$-invariant then the non-zero eigenspaces of $H_f(X)$ are orthogonal to the tangent space $T_x (K_{\bf v} \cdot x)$ of the $K_{\bf v}$-orbit through $x$ and therefore $0 = \frac{1}{\lambda} \rho_x^* H_f(X) = \rho_x^* X$. One can also see this explicitly by applying $\rho_x^*$ to both sides of the equation $H_f(X) = \lambda X$ and using equations \eqref{eqn:rho-I-rho} and \eqref{eqn:rho-I-delta} to obtain
\begin{align}\label{eqn:Hessian-compact-action}
\begin{split}
\rho_x^* I \delta \rho_x(\mu-\alpha)(X) - \rho_x^* I \rho_x \rho_x^* I X & = \lambda \rho_x^* X \\
\Leftrightarrow \quad - [\rho_x^* I X, \mu-\alpha] - [\mu-\alpha, \rho_x^* I X] & = \lambda \rho_x^* X \\
\Leftrightarrow \quad 0 & = \lambda \rho_x^* X .
\end{split}
\end{align}
Since $\lambda \neq 0$ then $\rho_x^* X = 0$. Now suppose that $H_f(X) = \lambda X$ for some $\lambda < 0$. Applying $\rho_x^* I$ to both sides of the eigenvalue equation and using \eqref{eqn:rho-delta} and the critical point equation gives us
\begin{align}\label{eqn:Hessian-imaginary-action}
\begin{split}
- \rho_x^* \delta \rho_x(\mu-\alpha)(X) + \rho_x^* \rho_x \rho_x^* I X & = \lambda \rho_x^* I X \\
\Leftrightarrow \quad - [\rho_x^* X, \mu -\alpha] + \rho_x^* \rho_x (\rho_x^* I X) & = \lambda \rho_x^* I X .
\end{split}
\end{align}

Since we have already shown that $\rho_x^* X = 0$, then
\begin{equation*}
\rho_x^* \rho_x (\rho_x^* I X) = \lambda \rho_x^* I X ,
\end{equation*}
and so $\rho_x^* I X = 0$, since $\lambda < 0$ and the operator $\rho_x^* \rho_x$ is non-negative definite.
\end{proof}

\begin{corollary}\label{cor:kernel-preserved}
The Hessian $H_f$ preserves $\ker (\rho_x^\C)^*$. 
\end{corollary}

\begin{proof}
Note that $\im \rho_x^\C = \im \rho_x + \im I \rho_x$ and so $\ker (\rho_x^\C)^* = (\im \rho_x^\C)^\perp = (\im \rho_x)^\perp \cap (\im I \rho_x)^\perp = \ker \rho_x^* \cap \ker (\rho_x^* I)$. Let $X \in \ker (\rho_x^\C)^*$. Then \eqref{eqn:Hessian-compact-action} shows that 
\begin{equation*}
\rho_x^* H_f(X) =  \rho_x^* I \delta \rho_x(\mu - \alpha)(X) - \rho_x^* I \rho_x \rho_x^* I X = - [\rho_x^* I X, \mu - \alpha] - [\mu-\alpha, \rho_x^* I X] = 0 .
\end{equation*}
Moreover, \eqref{eqn:Hessian-imaginary-action} shows that
\begin{equation*}
\rho_x^* I H_f(X) = - [\rho_x^* X, \mu -\alpha] + \rho_x^* \rho_x (\rho_x^* I X) = 0 ,
\end{equation*}
since $X \in \ker \rho_x^* \cap \ker \rho_x^* I$. Therefore $H_f(X) \in \ker \rho_x^* \cap \ker \rho_x^* I = \ker (\rho_x^\C)^*$. 
\end{proof}

\begin{definition}
Given $\lambda \in \R$, let $V_{x,\lambda} = \left\{ X \in T_x \Rep(Q, {\bf v}) \, : \, H_f(X) = \lambda X \right\}$ denote the $\lambda$-eigenspace of the Hessian at a critical point $x$. The negative eigenspace is denoted $V_x^- := \oplus_{\lambda < 0} V_{x, \lambda}$.
\end{definition}

\begin{lemma}\label{lem:complex-image-splits}
$\rho_x(\mathfrak{k}) \subseteq V_{x,0}$ and $\im \rho_x^\C$ splits into eigenspaces for $H_f$, with $\displaystyle{\im \rho_x^\C \subseteq \bigoplus_{\lambda \geq 0} V_{x,\lambda}}$.
\end{lemma}

\begin{proof}
The statement that $\im \rho_x \subseteq V_0$ follows from the fact that the function $\| \mu - \alpha \|^2$ is $K_{\bf v}$-invariant. One can also explicitly see this from the calculation
\begin{align*}
H_f(\rho_x(u)) & = I \delta \rho_x(\mu(x)-\alpha)(\rho_x(u)) -  I \rho_x \rho_x^* I \rho_x(u) \\
 & = I \delta \rho_x(\mu(x)-\alpha)(\rho_x(u)) - I \rho_x([\mu(x), u]) \quad \text{by \eqref{eqn:rho-I-rho}} \\
 & = I \delta \rho_x(\mu(x)-\alpha)(\rho_x(u)) - I \delta \rho_x(\mu(x) - \alpha)(\rho_x(u)) \quad \text{by \eqref{eqn:inf-action-bracket}} \\
 & = 0 .
\end{align*}

Since $H_f$ is self-adjoint and preserves $\ker (\rho_x^\C)^*$ by Corollary \ref{cor:kernel-preserved}, then $\im \rho_x^\C = (\ker (\rho_x^\C)^*)^\perp$ is preserved also, and therefore it splits into eigenspaces for $H_f$. Lemma \ref{lem:neg-eigenspace} then shows that each eigenvalue must be non-negative.
\end{proof}

Given dimension vectors ${\bf v_1}$ and ${\bf v_2}$, with corresponding collections of vector spaces $\{ V_k^1 \}_{k \in \mathcal{I}}$ and $\{ V_k^2 \}_{k \in \mathcal{I}}$, define the spaces
\begin{align}
\Hom^0(Q, {\bf v_1}, {\bf v_2}) & := \bigoplus_{k \in \mathcal{I}} \Hom(V_k^1, V_k^2) \\
\Hom^1(Q, {\bf v_1}, {\bf v_2}) & := \bigoplus_{a \in \mathcal{E}} \Hom(V_{t(a)}^1, V_{h(a)}^2) .
\end{align}

The final result of this section is a characterisation of the negative eigenspace of the Hessian in terms of homomorphisms between the subrepresentations that appear in the splitting \eqref{eqn:critical-decomposition}.

\begin{proposition}\label{prop:neg-eigenspace}
Let $x$ be a critical point of $f(x) = \| \mu(x) - \alpha \|^2$ on $\Rep(Q, {\bf v})$ corresponding to a decomposition \eqref{eqn:critical-decomposition}. For $j, k = 1, \ldots, n$, define $\lambda_{j,k} := \lambda_j - \lambda_k = \slope_\alpha(Q, {\bf v_k}) - \slope_\alpha(Q, {\bf v_j})$, where $\lambda_j$ and $\lambda_k$ are as in \eqref{eqn:mom-map-block-diagonal}. Then if $\lambda < 0$ we have
\begin{equation}\label{eqn:neg-eigenspace}
V_{x,\lambda} = \bigoplus_{\{j,k \, : \, \lambda_{j,k} = \lambda \}} (\ker \rho_x^\C)^* \cap \Hom^1(Q, {\bf v_j}, {\bf v_k}) .
\end{equation}
Therefore the dimension of the negative eigenspace is equal to
\begin{equation}\label{eqn:neg-slice-dimension}
\dim_\C V_x^- = \sum_{\lambda < 0} \dim_\C V_{x,\lambda} = \sum_{j,k \, : \, \lambda_{j,k} < 0} \left( \dim_\C \Hom^1(Q, {\bf v_j}, {\bf v_k}) - \dim_\C \Hom^0(Q, {\bf v_j}, {\bf v_k}) \right)
\end{equation}
\end{proposition}

\begin{proof}
Lemma \ref{lem:neg-eigenspace} shows that when $\lambda < 0$ the negative eigenspace equation reduces to
\begin{equation*}
I \delta \rho_x(\mu(x)-\alpha)(X) = \lambda X .
\end{equation*}
Applying equations \eqref{eqn:mom-map-block-diagonal} and \eqref{eqn:explicit-deriv-inf-action} completes the proof of \eqref{eqn:neg-eigenspace}.

For each $j, k$, the infinitesimal action defines a homomorphism
\begin{equation*}
\Hom^0(Q, {\bf v_j}, {\bf v_k}) \stackrel{\rho_x^\C}{\longrightarrow} \Hom^1(Q, {\bf v_j}, {\bf v_k})
\end{equation*}
and \eqref{eqn:neg-eigenspace} shows that $V_{x, \lambda_{j,k}}^-$ is the cokernel of this homomorphism. Moreover, since $x = x_1 \oplus \cdots \oplus x_n$ splits into stable subrepresentations and $\lambda_{j,k} < 0$ implies that $\slope_\alpha(Q, {\bf v_j}) > \slope_\alpha(Q, {\bf v_k})$, then $\ker \rho_x^\C = \{0\}$ by \cite[Lem. 2.3]{Reineke03}, and so 
\begin{equation*}
\dim_\C \coker(\rho_x^\C) = \sum_{j,k \, : \, \lambda_{j,k} < 0} \left( \dim_\C \Hom^1(Q, {\bf v_j}, {\bf v_k}) - \dim_\C \Hom^0(Q, {\bf v_j}, {\bf v_k}) \right)
\end{equation*}
which gives us \eqref{eqn:neg-slice-dimension}.
\end{proof}

\begin{definition}\label{def:unstable-manifold}
Given a critical point $x \in \Rep(Q, {\bf v})$, define the \emph{unstable manifold}
\begin{equation*}
W_x^- = \{ y \in \Rep(Q, {\bf v}) \mid \lim_{t \rightarrow - \infty} \phi(y, t) = x \} .
\end{equation*}
\end{definition}

We conclude this section with some remarks about the relationship between the unstable manifold and the negative slice to motivate the constructions of Section \ref{subsec:scattering}. Standard ODE theory (cf. \cite{CoddingtonLevinson55}, \cite{Hartman64}) shows that $W_x^-$ is a manifold and that the negative eigenspace of the Hessian is isomorphic to the tangent space $T_x W_x^- \cong V_x^-$. Therefore a neighbourhood of $x$ in $W_x^-$ is diffeomorphic to a neighbourhood of zero in $V_x^-$. Moreover, when the critical sets are compact then the methods of Kirwan in \cite{Kirwan84} (see also \cite[Thm. 4.1]{HirschPughShub77}) show that for each critical set $C$, the unstable manifolds $\{ W_x^- \mid x \in C \}$ glue together to form a disk bundle over $C$, which we call the \emph{unstable bundle} of the critical set $C$, denoted $W_C^-$. Similarly, the negative eigenspaces $V_x^-$ glue together to form the \emph{negative slice bundle} $S_C^- := \{ V_x^- \mid x \in C\}$. In a neighbourhood of the zero section $C$, these two bundles are homeomorphic. This homeomorphism is defined abstractly and it is not clear whether it remains a homeomorphism on restricting to a singular subset. In addition, one would also like to relate the isomorphism classes in $W_x^-$ to those in $V_x^-$ (which can be classified algebraically). The main technical result of this paper is Theorem \ref{thm:homeo-slice-bundle}, which shows that one can construct a $K_{\bf v}$-equivariant homeomorphism $S_C^- \stackrel{\cong}{\longrightarrow} W_C^-$ using the action of $G_{\bf v}$ and that \emph{this remains a homeomorphism on restriction to any closed $G_{\bf v}$-invariant subset of $\Rep(Q, {\bf v})$}.

\section{Local structure of the space of representations of quivers with relations}\label{sec:singular-space}

This section contains the basic setup for quivers with relations and extends the results of the previous section to this setting in preparation for the classification of flow lines in Section \ref{sec:gradient-Nakajima}. 

\subsection{Moduli spaces of quivers with relations}\label{sec:quiver-def}

In this section we define and study the basic properties of moduli spaces of quivers with relations. The definition given here is for unframed quivers, which is also valid for framed quivers by Crawley-Boevey's construction in \cite{Crawley01}. In particular, the definition includes Nakajima quiver varieties from \cite{Nakajima94}, \cite{Nakajima98}, \cite{Nakajima01} and the handsaw quiver varieties from \cite{Nakajima12}. A good reference for quivers with relations is \cite{Brion08}.

\begin{definition}
A \emph{relation} of a quiver $Q$ is a subspace of the path algebra $kQ$ spanned by linear combinations of paths of length at least $2$ having a common head and common tail.

A \emph{quiver with relations} is a pair $(Q, \mathcal{R})$, where $Q$ is a quiver and $\mathcal{R}$ is a two-sided ideal of $kQ$ generated by relations. The path algebra of $(Q, \mathcal{R})$ is the quotient algebra $kQ / \mathcal{R}$. In the sequel $\mathcal{R}$ is also used to denote the corresponding set of relations in the path algebra.
\end{definition}

Now fix a dimension vector ${\bf v}$ for $Q$. A relation in $Q$ (denoted $r$) with tail $t(r) \in \mathcal{I}$ and head $h(r) \in \mathcal{I}$ determines a vector space homomorphism 
\begin{equation}\label{eqn:relation-homomorphism-def}
\nu_r : \Rep(Q, {\bf v}) \rightarrow \Hom(V_{t(r)}, V_{h(r)})
\end{equation}
given by composing homomorphisms along the paths in the relation.

\begin{example}\label{ex:relations}

\begin{enumerate}

\item \label{item:relations-nakajima} Let $Q$ be a quiver with vertices $\mathcal{I}$ and edges $\mathcal{E}$, and $\tilde{Q}$ the ``doubled'' quiver with vertices $\mathcal{I}$ and edges $\mathcal{E} \cup \bar{\mathcal{E}}$ introduced in \cite{Nakajima94}. For each vertex $k \in I$ there is a relation
\begin{equation*}
\sum_{a \in \mathcal{E} \, s.t. \, h(a) = k} a \bar{a} - \sum_{a \in \mathcal{E} \, s.t. \, t(a) = k} \bar{a} a 
\end{equation*}
This induces the homomorphism
\begin{align*}
\nu_k : \Rep(\tilde{Q}, {\bf v}) & \rightarrow \Hom(V_k, V_k) \\
 x & \mapsto \sum_{a \in \mathcal{E} \, s.t. \, h(a) = k} x_a x_{\bar{a}} - \sum_{a \in \mathcal{E} \, s.t. \, t(a) = k} x_{\bar{a}} x_a  
\end{align*}
The direct sum of these maps over all the vertices is the complex moment map $\mu_\C$ associated to the hyperk\"ahler structure.

\item \label{item:relations-handsaw} Let $Q$ be a ``handsaw'' quiver as in \cite{Nakajima12} with edges labeled as below.
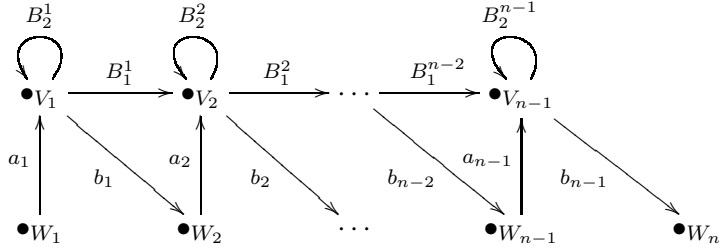
\begin{equation*}
\xymatrixrowsep{0.5in}
\xymatrixcolsep{0.5in}
\xymatrix{
\bullet_{V_1} \ar[r]^{B_1^1} \ar[dr]_{b_1} \ar@`{(10,10),(-10,10)}_{B_2^1} & \bullet_{V_2} \ar[r]^{B_1^2} \ar[dr]_{b_2} \ar@`{(30,10),(10,10)}_{B_2^2} & \cdots \ar[r]^{B_1^{n-2}} \ar[dr]_{b_{n-2}} & \bullet_{V_{n-1}} \ar[dr]_{b_{n-1}} \ar@`{(72,10),(52,10)}_{B_2^{n-1}} & \\
\bullet_{W_1} \ar[u]^{a_1} & \bullet_{W_2} \ar[u]^{a_2} & \cdots & \bullet_{W_{n-1}} \ar[u]^{a_{n-1}} & \bullet_{W_n} 
}
\end{equation*} 
 For each $k = 1, \ldots, n-2$ there is a relation
\begin{equation*}
B_1^k B_2^{k} - B_2^{k+1} B_1^k + a_{k+1} b_k = 0
\end{equation*}
Each relation induces a map $\nu_k : \Rep(Q, {\bf v}) \rightarrow \Hom(V_k, V_{k+1})$ and the direct sum of these maps for $k=1, \ldots, n-2$ is the map $\mu$ from \cite[Sec. 2]{Nakajima12}.
\end{enumerate}
\end{example}

Given a relation $r$, let $\Rel(Q, {\bf v}, r)$ denote the subspace of $\Hom(V_{t(r)}, V_{h(r)})$ consisting of homomorphisms that can be written as the composition of homomorphisms along the path defining the relation.
\begin{example}
For the quiver
\begin{equation*}
\xymatrix{
\bullet_{V_1} \ar[r] & \bullet_{V_2} \ar[r] & \bullet_{V_3}
}
\end{equation*}
let $r$ be the relation corresponding to the unique path from $V_1$ to $V_3$. Then $\Rel(Q, {\bf v}, r)$ is the subspace of $\Hom(V_1, V_3)$ consisting of homomorphisms that factor through $V_2$. Note that if $\dim V_2 < \min\{ \dim V_1, \dim V_3\}$ then $\Rel(Q, {\bf v}, r)$ will be a proper subspace of $\Hom(V_1, V_3)$.
\end{example}

Given a set of relations $\mathcal{R}$, let $\Rel(Q, {\bf v}, \mathcal{R})$ denote the vector space 
\begin{equation}\label{eqn:relation-definition}
\Rel(Q, {\bf v}, \mathcal{R}) := \bigoplus_{r \in \mathcal{R}} \Rel(Q, {\bf v}, r)
\end{equation}
All of the relations together induce a $G_{\bf v}$-equivariant map
\begin{align}\label{eqn:relation-function}
\begin{split}
\nu : \Rep(Q, {\bf v}) & \rightarrow \Rel(Q, {\bf v}, \mathcal{R}) \\
x & \mapsto \sum_{r \in \mathcal{R}} \nu_r(x) 
\end{split}
\end{align}
where $\nu_r$ is the map defined in \eqref{eqn:relation-homomorphism-def}. 

\begin{remark}\label{rem:invariance}

\begin{enumerate}

\item The examples above show that the construction of $\nu$ specialises to the complex moment map associated to a representation of a doubled quiver from \cite{Nakajima94}, and the analogous construction for handsaw quivers in \cite{Nakajima12}.

\item The space $\nu^{-1}(0)$ is always $G_{\bf v}$-invariant. To see this, note that the composition of homomorphisms $x_{a_\ell} \cdots x_{a_2} x_{a_1}$ along a path $a_1, a_2, \ldots, a_\ell$ becomes $g_{h(a_\ell)} x_{a_\ell}  \cdots x_{a_1} g_{t(a_1)}^{-1}$ under the action of $g \in G_{\bf v}$. Each relation $r$ consists of a linear combination of paths with the same head and tail vertex, denoted $h(r)$ and $t(r)$ respectively. Therefore $\nu_r(g \cdot x) = g_{h(r)} \nu_r(x) g_{t(r)}^{-1}$. 

\item If each relation $r \in \mathcal{R}$ is generated by paths of the same length $\ell(r)$ then each $\nu_r$ is a homogeneous polynomial of degree $\ell(r)$ and so the space $\nu^{-1}(0)$ is invariant under scalar multiplication $x \mapsto \lambda x$. This is true for Examples \ref{ex:relations} \eqref{item:relations-nakajima} \& \eqref{item:relations-handsaw} where each relation is generated by paths of length two.

\end{enumerate}

\end{remark}

\begin{definition}\label{def:moduli-quiver-relations}
Given an admissible stability parameter $\alpha$, the \emph{moduli space of representations of $(Q, \mathcal{R})$} is 
\begin{equation}\label{eqn:moduli-definition}
\mathcal{M}_\alpha(Q, {\bf v}, \mathcal{R}) := \left( \Rep(Q, {\bf v})^{\alpha-ss} \cap \nu^{-1}(0) \right) / \negthickspace / G_{\bf v}
\end{equation}
Proposition \ref{prop:King-Kempf-Ness} shows that the GIT quotient $\mathcal{M}_\alpha(Q, {\bf v}, \mathcal{R})$ is homeomorphic to the symplectic quotient $\left( \mu^{-1}(\alpha) \cap \nu^{-1}(0) \right) / K_{\bf v}$.
\end{definition}

There is a special choice of stability parameter which reproduces Nakajima's stability condition for framed quiver varieties from \cite[Prop. 3.5]{Nakajima94}.

\begin{definition}\label{def:aasp}
Let $Q$ be a quiver with vertices $\mathcal{I}$ and edges $\mathcal{E}$, and let ${\bf v} = ( v_i )_{i \in \mathcal{I}}$ a dimension vector such that one vertex (which we label $\infty$) has dimension $1$. Define $\mathcal{I}' = \mathcal{I} \setminus \{ \infty \}$ be the set of remaining vertices of $Q$. For such a quiver $Q$ and dimension vector ${\bf v}$, the \emph{canonical stability parameter} $\alpha(Q, {\bf v}) := ( \alpha_i )_{i \in \mathcal{I}}$ is given by
\begin{equation}\label{eqn:Nakajima-stability}
\alpha_i := \left\{ \begin{matrix} -\sum_{j \in \mathcal{I}'} v_j & i = \infty \\ 1 & i \in \mathcal{I}' \end{matrix} \right. .
\end{equation}

In this case we define
\begin{equation}\label{eqn:vect-zero}
\Vect_0(Q, {\bf v}) = \bigoplus_{k \in \mathcal{I}'} V_k
\end{equation}
to be the direct sum of all the vector spaces except for the one at the vertex $\infty$.
\end{definition}

\begin{lemma}\label{lem:semistable-equals-stable}
The $\alpha$-semistable points are all $\alpha$-stable for this choice of stability parameter. 
\end{lemma}

\begin{proof}
Note that a proper subrepresentation satisfies exactly one of the following conditions: (a) the subrepresentation does not contain the vertex $\infty$ and so it must have strictly positive slope, or (b) the subrepresentation contains the vertex $\infty$ and so it must have strictly negative slope. 

A subrepresentation of an $\alpha$-semistable representation cannot be in case (a) and therefore the slope of any subrepresentation must be strictly negative, so the representation is in fact $\alpha$-stable.
\end{proof}

When the quiver is an affine Dynkin diagram with ``doubled'' edges then the quiver varieties associated to two generic stability parameters are diffeomorphic (see \cite[Corollary 4.2]{Nakajima94}). In the sequel we need the following result relating moduli spaces where the stability parameters differ by a scalar multiple, which is valid for any set of relations where each relation is generated by paths of the same length.
\begin{lemma}\label{lem:scaling}
Let $Q$ be any quiver and suppose that each relation $r \in \mathcal{R}$ is generated by paths of the same length $\ell(r)$. If $\beta = k \alpha$ for some real scalar $k > 0$ then $\mathcal{M}_{\beta}(Q, {\bf v}, \mathcal{R}) \cong \mathcal{M}_\alpha(Q, {\bf v}, \mathcal{R})$.
\end{lemma}

\begin{proof}
Let $x \in \mu^{-1}(\alpha)$. Then $\mu(\sqrt{k} x) = \beta$, since $\mu$ is a homogeneous quadratic polynomial of degree two. Remark \ref{rem:invariance} shows that the condition on the path lengths implies that the space $\nu^{-1}(0)$ is preserved by the transformation $x \mapsto \sqrt{k} x$ and so we have a continuous map $\mathcal{M}_\alpha(Q, {\bf v}, \mathcal{R}) \rightarrow \mathcal{M}_\beta(Q, {\bf v}, \mathcal{R})$. Similarly, the inverse map is continuous and so $\mathcal{M}_{\beta}(Q, {\bf v}, \mathcal{R}) \cong \mathcal{M}_\alpha(Q, {\bf v}, \mathcal{R})$.

Equivalently, one can also note that the stability condition from Definition \ref{def:stability} is preserved if we multiply the stability parameter by a positive non-zero scalar. The same is true for the slope-stability condition from Proposition \ref{prop:slope-stability}.
\end{proof}

\begin{remark}\label{rem:crawley-boevey}
\begin{enumerate}

\item Definition \ref{def:moduli-quiver-relations} differs slightly from that given by Nakajima in \cite{Nakajima94}, which also involves a framing of the quiver. It was first pointed out by Crawley-Boevey in \cite{Crawley01} that these framed quiver varieties can be interpreted as a quiver variety of the form described above. We briefly recall this construction in the notation of this paper since it is relevant to the current section. Given a quiver $Q'$ with vertices $\mathcal{I}'$ and edges $\mathcal{E}'$, dimension vector ${\bf v}' = ( v_i )_{i \in \mathcal{I}'}$, and framed dimension vector ${\bf w'} = ( w_i )_{i \in \mathcal{I}'}$ in the notation of \cite{Nakajima94}, let $Q$ be a new quiver with vertices $\mathcal{I} = \mathcal{I}' \cup \{ \infty \}$ and edges $\mathcal{E} = \mathcal{E}' \cup \mathcal{F}$, where $\mathcal{F}$ consists of $w_i$ edges from $\infty$ to each edge $i \in \mathcal{I}'$. Also let ${\bf v} = ({\bf v}', 1)$ be the dimension vector obtained from ${\bf v}'$ by adjoining a $1$ for the new vertex $\infty$. Since the construction of $(Q, {\bf v})$ described above has a vertex with dimension one, then it has a stability parameter $\alpha(Q, {\bf v})$ as defined in Definition \ref{def:aasp}. Crawley-Boevey then shows in \cite{Crawley01} that the quotient $\mathcal{M}_\alpha(Q, {\bf v}, \mathcal{R})$ is the same as Nakajima's definition of quiver variety $\mathcal{M}(Q, {\bf v'}, {\bf w'})$ and the method works in exactly the same way for framed quivers with relations.

\item Crawley-Boevey also shows that the stability parameter $\alpha(Q, {\bf v})$ induces the same stability condition on $\nu^{-1}(0)$ as Nakajima's stability condition for the framed quiver $(Q', {\bf v'}, {\bf w'})$ from \cite[Sec. 3.ii]{Nakajima98}. To see this, note that the stability condition induced by $\alpha(Q, {\bf v})$ is that $x \in \mu_\C^{-1}(0)$ is $\alpha$-stable if and only if every subrepresentation has negative slope, which occurs if and only if every subrepresentation contains the vertex $\infty$. This is equivalent to condition $(2)$ of \cite[Lemma 3.8]{Nakajima98}.  

\end{enumerate}
\end{remark}

\subsection{Structure of the critical sets}\label{sec:singular-critical-sets}

In this section we describe the structure of representations that are critical points of $\| \mu - \alpha \|^2$ on $\nu^{-1}(0)$. First, we define what it means for a representation to be critical for $\| \mu - \alpha \|^2$ on the singular space $\nu^{-1}(0) \subset \Rep(Q, {\bf v})$.

\begin{definition}\label{def:critical-point-singular}
A point $x \in \nu^{-1}(0) \subset \Rep(Q, {\bf v})$ is \emph{critical} for $\| \mu - \alpha \|^2$ if and only if $x$ is critical for $\| \mu - \alpha \|^2$ on the ambient smooth space $\Rep(Q, {\bf v})$.
\end{definition}

Since $\nu^{-1}(0)$ is singular then this definition needs some justification. Returning to the smooth space $\Rep(Q, {\bf v})$ for the moment, recall from \eqref{eqn:group-neg-flow} that the gradient flow of $\| \mu - \alpha \|^2$ on $\Rep(Q, {\bf v})$ is generated by the action of $G_{\bf v}$. Therefore, for any $G_{\bf v}$-invariant closed subset $Z \subset \Rep(Q, {\bf v})$ (for example $Z = \nu^{-1}(0)$), if $x \in Z$ then the flow satisfies $\phi(x, t) \in Z$ for all $t$ such that $\phi(x,t)$ is defined. Since $Z$ is closed, then any limit point of the flow is also contained in $Z$. Therefore we can define the gradient flow on the subset to be the restriction of the gradient flow on the smooth space $\Rep(Q, {\bf v})$ and Theorem \ref{thm:algebraic-flow-limit} will apply. In particular, if we define the critical points of $\| \mu - \alpha \|^2$ on $\nu^{-1}(0)$ as in Definition \ref{def:critical-point-singular}, then we have a Morse stratification of the space $\nu^{-1}(0)$ by Theorem \ref{thm:algebraic-flow-limit}.

We also have the following property of critical points on the smooth space $\Rep(Q, {\bf v})$.

\begin{lemma}
\begin{enumerate}
\item Let $x \in \Rep(Q, {\bf v})$ be a critical point of $\| \mu - \alpha \|^2$. Then $x$ minimises the value of $\| \mu - \alpha \|^2$ on the orbit $G_{\bf v} \cdot x$. 

\item Given any $x \in \Rep(Q, {\bf v})$, consider the orbit closure $\overline{G_{\bf v} \cdot x}$. There is a unique $K_{\bf v}$-orbit $K_{\bf v} \cdot x_\infty$ of critical points in $\overline{G_{\bf v} \cdot x}$ that contains the limit $x_\infty$ of the downwards gradient flow of $\| \mu - \alpha \|^2$ with initial condition $x$. The representations minimising $\| \mu - \alpha \|^2$ on $\overline{G_{\bf v} \cdot x}$ are precisely those in this $K_{\bf v}$-orbit.

\end{enumerate} 
\end{lemma}

\begin{proof}

Recall that the Harder-Narasimhan type is $G_{\bf v}$-invariant and so $G_{\bf v} \cdot x$ is contained in the Harder-Narasimhan stratum of $x$. The result of \cite[Corollary 2, p334]{HaradaWilkin11} shows that the critical point $x$ minimises the value of $\| \mu - \alpha \|^2$ on the Harder-Narasimhan stratum and therefore it must do so on the $G_{\bf v}$-orbit also.

Recall Reineke's result \cite[Prop. 3.7]{Reineke03} which says the closure of a Harder-Narasimhan stratum $\Rep(Q, {\bf v})_{{\bf v^*}}$ is contained in the union
\begin{equation*}
\overline{\Rep(Q, {\bf v})_{{\bf v^*}} } \subset \bigcup_{{\bf w^*} \geq {\bf v^*}} \Rep(Q, {\bf v})_{{\bf w^*}}.
\end{equation*}
Therefore the closure $\overline{G_{\bf v} \cdot x}$ is also contained in this union. 

To see that this is minimised by a unique $K_{\bf v}$-orbit, first note that the minimum of $\| \mu - \alpha \|^2$ on $\overline{G_{\bf v} \cdot x}$ is not attained by any point in $\Rep(Q, {\bf v})_{\bf w^*}$ for ${\bf w^*} > {\bf v^*}$, since (a) the minimum of $\| \mu - \alpha \|^2$ on $\Rep(Q, {\bf v})_{\bf w^*}$ is strictly greater than the value of $\| \mu - \alpha \|^2$ on the set of critical points in $\Rep(Q, {\bf v})_{\bf v^*}$, and (b) applying Theorem \ref{thm:algebraic-flow-limit} to the gradient flow with initial condition $x \in \Rep(Q, {\bf v})_{\bf v^*}$ shows that the minimum of $\| \mu - \alpha \|^2$ on $\overline{G_{\bf v} \cdot x}$ is attained by a critical point in $\Rep(Q, {\bf v})_{\bf v^*}$.

Theorem \ref{thm:algebraic-flow-limit} shows that gradient flow induces a deformation retract of $\Rep(Q, {\bf v})_{\bf v^*}$ onto the associated critical set and that the image of the subset $G_{\bf v} \cdot x$ under this deformation retract is a single $K_{\bf v}$-orbit containing the limit of the flow with initial condition $x$. Therefore, since the deformation retract is continuous, then $\overline{G_{\bf v} \cdot x} \cap \Rep(Q, {\bf v})_{\bf v^*}$ deformation retracts onto the closure of this $K_{\bf v}$-orbit. Since the orbit is closed then this completes the proof.
\end{proof}

As a consequence of the above lemma and the fact that the flow is contained in a $G_{\bf v}$-orbit, we see that the critical points defined in Definition \ref{def:critical-point-singular} have the following properties.

\begin{corollary}\label{cor:critical-properties}
\begin{enumerate}
\item Let $x \in \nu^{-1}(0) \subset \Rep(Q, {\bf v})$ be a critical point of $\| \mu - \alpha \|^2$. Then $x$ minimises the value of $\| \mu - \alpha \|^2$ on the orbit $G_{\bf v} \cdot x \subset \nu^{-1}(0)$. 

\item Given any $x \in \nu^{-1}(0)$, consider the orbit closure $\overline{G_{\bf v} \cdot x}$. The minimum of $\| \mu - \alpha \|^2$ on $\overline{G_{\bf v} \cdot x}$ is precisely the $K_{\bf v}$-orbit of critical points in $\overline{G_{\bf v} \cdot x}$ that contains the limit of the downwards gradient flow of $\| \mu - \alpha\|^2$ with initial condition $x$.

\end{enumerate}
\end{corollary}

The rest of this section contains more details about the structure of critical points in $\nu^{-1}(0)$ with respect to the stability parameter $\alpha(Q, {\bf v})$ from Definition \ref{def:aasp}. Let $x \in \nu^{-1}(0)$ be a critical point. Recall from \eqref{eqn:critical-decomposition} that $x$ must split into subrepresentations and from \eqref{eqn:mom-map-block-diagonal} that the value of the moment map on each subrepresentation is determined by the slope. Each of the subrepresentations is semistable with respect to the induced stability parameter.

Since the vertex $\infty$ has dimension $1$ then only one of the subrepresentations (call it $x_1$) in the decomposition \eqref{eqn:critical-decomposition} can have non-zero dimension vector at this vertex. Let ${\bf v_1} = (v_i')_{i \in \mathcal{I}}$ be the dimension vector for this subrepresentation. A calculation shows that the induced stability parameter is
\begin{equation}\label{eqn:induced-Nakajima-parameter}
\alpha_i' = \left\{ \begin{matrix} \frac{1 + \sum_{j \in \mathcal{I}'} v_j}{1 + \sum_{j \in \mathcal{I}'} v_j'} & i \in \mathcal{I}' \\ -\left( \sum_{j \in \mathcal{I}'} v_j' \right) \frac{1 + \sum_{j \in \mathcal{I}'} v_j}{1 + \sum_{j \in \mathcal{I}'} v_j'} & i = \infty \end{matrix} \right.
\end{equation}
which is a positive scalar multiple of the stability parameter $\alpha(Q, {\bf v'})$. Lemma \ref{lem:semistable-equals-stable} then shows that $x_1$ is stable with respect to the induced stability parameter and Lemma \ref{lem:scaling} shows that the induced stability parameter is equivalent to the parameter from Definition \ref{def:aasp}.

From \eqref{eqn:Nakajima-stability} we see that all of the other subrepresentations must then have the same slope. Let $x_2$ denote the sum of all the subrepresentations in \eqref{eqn:critical-decomposition} that do not contain the vertex $\infty$. Then \eqref{eqn:mom-map-block-diagonal} shows that $\mu(x_2) = 0$. The above argument is summarised in the following proposition.

\begin{proposition}\label{prop:singular-critical-decomp}
Let $x \in \nu^{-1}(0)$ be a critical point of $\| \mu - \alpha \|^2$. Then $x$ splits into two subrepresentations $x_1$ and $x_2$ with respective dimension vectors ${\bf v_1}$ and ${\bf v_2}$. The induced values of the moment map are $\mu(x_1) = k \alpha(Q, {\bf v_1})$ and $\mu(x_2) = 0$, where $k = \frac{1 + \sum_{j \in \mathcal{I}'} v_j}{1 + \sum_{j \in \mathcal{I}'} v_j'} > 0$ is the scalar from \eqref{eqn:induced-Nakajima-parameter}. The subrepresentation $x_1$ is stable with respect to the induced stability parameter.

Moreover, any representation $x \in \nu^{-1}(0)$ of the form $x = x_1 \oplus x_2$ where $\mu(x_1) = k \alpha(Q, {\bf v_1})$ and $\mu(x_2) = 0$ is a critical point of $\| \mu - \alpha \|^2$.
\end{proposition}

\begin{definition}\label{def:critical-components}
Let $C_{\bf v_1}$ denote all of the critical points of $\| \mu - \alpha \|^2$ on $\nu^{-1}(0)$ for which the stable subrepresentation containing the vertex $\infty$ from the decomposition in Proposition \ref{prop:singular-critical-decomp} has dimension vector ${\bf v_1}$. 

Given a dimension vector ${\bf v_1} = (v_i')_{i \in \mathcal{I}} < {\bf v}$ and associated vector spaces $\{ V_i' \}_{i \in \mathcal{I}}$ such that $\dim_\C V_i' = v_i'$, fix an inclusion $V_i' \hookrightarrow V_i$ for each $i \in \mathcal{I}$. Let $C_{\bf v_1}^0 \subset C_{\bf v_1}$ be the subset consisting of representations of the form $x_1 \oplus x_2$ with $x_2 = 0$ such that $x$ preserves $\bigoplus_{i \in \mathcal{I}} V_i'$.
\end{definition}

\begin{lemma}\label{lem:quotient-critical}
Given the fixed inclusion $V_i' \hookrightarrow V_i$ for each $i \in \mathcal{I}$ from Definition \ref{def:critical-components}, let $K_{\bf v_1}$ denote the associated subgroup of $K_{\bf v}$. Then
\begin{equation*}
C_{\bf v_1} / K_{\bf v} \cong \mathcal{M}_\alpha(Q, {\bf v_1}, \mathcal{R}) \times \mathcal{M}_0(Q, {\bf v} - {\bf v_1}, \mathcal{R}), \quad \text{and} \quad C_{\bf v_1}^0 / K_{\bf v_1} \cong \mathcal{M}_\alpha(Q, {\bf v_1}, \mathcal{R}) .
\end{equation*}
Moreover, if each relation in $\mathcal{R}$ is generated by paths of the same length $\ell(r)$ then there is a $K_{\bf v_1}$-equivariant deformation retraction of $C_{\bf v_1}$ onto $C_{\bf v_1}^0$.
\end{lemma}

\begin{proof}
On the space $\Rep(Q, {\bf v})$, \cite[Prop. 12]{HaradaWilkin11} describes the fibre bundle structure of the critical sets. Restricting this  to $\nu^{-1}(0)$ and applying Proposition \ref{prop:singular-critical-decomp} shows that $C_{\bf v_1} / K_{\bf v} \cong \mathcal{M}_{\alpha}(Q, {\bf v_1}, \mathcal{R}) \times \mathcal{M}_0(Q, {\bf v}-{\bf v_1}, \mathcal{R})$. Similarly, we obtain $C_{\bf v_1}^0 / K_{\bf v_1} \cong \mathcal{M}_\alpha(Q, {\bf v_1}, \mathcal{R}) \times \{[0]\}$. 

If the paths defining each relation all have the same length $\ell(r)$ then the equation $\nu(x_2) = 0$ is invariant under scaling by a real parameter (see Remark \ref{rem:invariance}). The moment map $\mu(x_2)$ is a homogeneous polynomial in $x_2$ and so a $K_{\bf v_1}$-equivariant deformation retract of $C_{\bf v}$ onto $C_{\bf v}^0$ is given by $(x_1, x_2) \mapsto (x_1, tx_2)$ for $0 \leq t \leq 1$.
\end{proof}

\begin{lemma}
For each ${\bf 0} \leq {\bf v_1} \leq {\bf v}$, the critical set $C_{\bf v_1}$ is a union of connected components of the set of critical points of $\| \mu - \alpha \|^2$. In the hyperk\"ahler case studied in \cite{Nakajima94}, each $C_{\bf v_1}$ is connected.
\end{lemma}

\begin{proof}
Using the estimate in \cite[Lemma 14]{HaradaWilkin11}, we can construct a neighbourhood around each point in $C_{\bf v_1}$ that does not intersect $C_{\bf v_1'}$ for any ${\bf v_1'} \neq {\bf v_1}$. Taking the union of these neighbourhoods gives an open neighbourhood of $C_{\bf v_1}$ in $\nu^{-1}(0)$ that does not intersect $C_{\bf v_1'}$ for any ${\bf v_1'} \neq {\bf v_1}$. Therefore $C_{\bf v_1}$ is open in $\bigcup_{v_1'} C_{\bf v_1'}$ (which has the subspace topology induced from $\nu^{-1}(0)$). 

This is also true for every other critical set, and so the complement of $C_{\bf v_1}$ is open in $\bigcup_{v_1'} C_{\bf v_1'}$. Therefore $C_{\bf v_1}$ is open and closed in this set and so it must be a union of connected components of $\bigcup_{v_1'} C_{\bf v_1'}$.

In the hyperk\"ahler case of a doubled quiver with relations as in \cite{Nakajima94}, Crawley-Boevey's result from \cite{Crawley01} shows that $\mathcal{M}_\alpha(Q, {\bf v_1}, \mathcal{R})$ is connected for each ${\bf v_1}$. Lemma \ref{lem:quotient-critical} shows that $C_{\bf v_1}^0$ fibres over this space with connected fibres and so it must also be connected. Therefore $C_{\bf v_1}$ is connected, since it deformation retracts onto $C_{\bf v_1}^0$. 
\end{proof}

\subsection{Local slices around the critical points}

Returning to the smooth space $\Rep(Q, {\bf v})$ for the moment, recall that the Hermitian structure on each of the vector spaces $\{ V_k \}_{k\in \mathcal{I}}$ defines a Hermitian structure on $\Rep(Q, {\bf v})$ and $\mathfrak{g}_{\bf v}$ (cf. \cite[Sec. 2.2]{HaradaWilkin11}). We use $g( \cdot, \cdot)$ to denote the inner product on $T_x \Rep(Q, {\bf v})$ and $\left< \cdot, \cdot \right>$ to denote the inner product on $\mathfrak{g}_{\bf v}$. The adjoint of the infinitesimal action is then a homomorphism $(\rho_x^\C)^* : T_x \Rep(Q, {\bf v}) \rightarrow \mathfrak{g}_{\bf v}$ that defines a direct sum decomposition $T_x \Rep(Q, {\bf v}) \cong \im \rho_x^\C \oplus \ker (\rho_x^\C)^* \cong (\ker \rho_x^\C)^\perp \oplus \ker (\rho_x^\C)^*$. The following local slice theorem is in \cite[Lemma 18]{HaradaWilkin11}.
\begin{lemma}\label{lem:local-slice}
Let $x \in \Rep(Q, {\bf v})$. The map
\begin{align*}
\zeta : (\ker \rho_x^\C)^\perp \oplus \ker (\rho_x^\C)^* & \rightarrow \Rep(Q, {\bf v}) \\
 (u, \delta x) & \mapsto \exp(u) \cdot (x + \delta x)
\end{align*}
is a diffeomorphism from a neighbourhood of $(0,0)$ in $(\ker \rho_x^\C)^\perp \oplus \ker (\rho_x^\C)^*$ to a neighbourhood of $x$ in $\Rep(Q, {\bf v})$.
\end{lemma}

The next result is a restriction of Lemma \ref{lem:local-slice} from $\Rep(Q, {\bf v})$ to a closed $G_{\bf v}$-invariant subset $Z \subset \Rep(Q, {\bf v})$. The local slices in Lemma \ref{lem:local-slice} are sufficiently small neighbourhoods of zero in $\ker (\rho_x^\C)^*$. On the space $Z$, we replace $\ker (\rho_x^\C)^*$ with the slice $S_x$ defined below.
\begin{definition}\label{def:slice}
Let $x \in Z \subset \Rep(Q, {\bf v})$ and let $\rho_x^\C$ denote the infinitesimal action of $G_{\bf v}$ on $T_x \Rep(Q, {\bf v}) \cong \Rep(Q, {\bf v})$. The \emph{slice} through $x$ is defined to be
\begin{equation*}
S_x = \left\{ \delta x \in  \Rep(Q, {\bf v}) \, : \, \delta x \in \ker (\rho_x^\C)^* \, \text{and} \, x + \delta x \in Z \right\} .
\end{equation*}
\end{definition}

We then have the following result.
\begin{lemma}\label{cor:singular-slice-thm}
Let $x \in Z$. The map
\begin{align*}
\zeta : (\ker \rho_x^\C)^\perp \times S_x & \rightarrow Z \\
 (u, \delta x) & \mapsto \exp(u) \cdot (x + \delta x)
\end{align*}
is a homeomorphism from a neighbourhood of $(0,0)$ in $(\ker \rho_x^\C)^\perp \times S_x$ to a neighbourhood of $x$ in $Z$.
\end{lemma}

\begin{proof}
Let $y \in Z$ be sufficiently close to $x$ such that Lemma \ref{lem:local-slice} applies in $\Rep(Q, {\bf v})$. Therefore we can write
\begin{equation*}
y = \exp(u) \cdot (x + \delta x) 
\end{equation*}
for unique $u \in (\ker \rho_x^\C)^\perp$ and $\delta x \in \ker (\rho_x^\C)^*$. Since $x + \delta x = \exp(-u) \cdot y \in Z$, then $\delta x \in S_x$. Therefore $\zeta$ surjects onto a neighbourhood of $x \in Z$. Since it is the restriction of a local diffeomorphism then it is injective, continuous and has a continuous inverse. Therefore $\zeta$ is a local homeomorphism.
\end{proof}

Since the slice consists of representations orthogonal to the $G_{\bf v}$ orbit through $x$, then the equation for the moment map simplifies.

\begin{lemma}
Let $x$ be a critical point with $\beta = \mu(x)$ and let $y = x + \delta x \in S_x$. Then
\begin{equation}\label{eqn:explicit-moment-map-slice}
\mu(y) - \beta = \frac{1}{2i} \sum_{a \in E} [\delta x_a, \delta x_a^*] 
\end{equation}
and there exists a constant $C > 0$ such that 
\begin{equation}\label{eqn:moment-map-slice-quadratic}
\| \mu(y) - \beta \| \leq C \| y - x \|^2
\end{equation}
\end{lemma}

\begin{proof}
Since the moment map $\mu(x + \delta x)$ is quadratic in $\delta x$, then we have
\begin{equation*}
\mu(y) - \beta = \mu(y) - \mu(x) = d \mu_x(\delta x) + \frac{1}{2i} \sum_{a \in E} [\delta x_a, \delta x_a^*] 
\end{equation*}
The defining equation for the moment map says that for any $u \in \mathfrak{k}$ we have
\begin{equation*}
d \mu_x(\delta x) \cdot u = \omega(\rho_x(u), \delta x) = g(I \rho_x(u), \delta x) = \left< u, - \rho_x^* I \delta x \right> = 0
\end{equation*}
since $\delta x \in \ker (\rho_x^\C)^*$. This completes the proof of \eqref{eqn:explicit-moment-map-slice}. The inequality \eqref{eqn:moment-map-slice-quadratic} then follows from \eqref{eqn:explicit-moment-map-slice} and the inequality $\| [A, B] \| \leq C \| A \| \| B \|$ for matrices $A$ and $B$. 
\end{proof}

In order to understand the group action on the slice, we need the following lemma.

\begin{lemma}
Let $x$ be a critical point, and let $\beta = \mu(x)$. Then for any $\delta x \in S_x$ and any $t \in \R$ we have $e^{i \beta t} \cdot \delta x \in S_x$.
\end{lemma}

\begin{proof}
First note that the critical point equations \eqref{eqn:critical-rep} for $x$ imply that $e^{i \beta t} \cdot x = x$ for all $t \in \R$. Given any $u \in \mathfrak{g}$ and any $X \in T_x \Rep(Q, {\bf v})$, we have
\begin{align*}
e^{i \beta t} \cdot \rho_x^\C(u) & = \bigoplus_{a \in \mathcal{E}} e^{i \beta t} (u_{h(a)} x_a - x_a u_{t(a)}) e^{-i \beta t} \\
 & = \bigoplus_{a \in \mathcal{E}} \left( e^{i \beta t} u_{h(a)} e^{-i \beta t} \right) x_a - x_a \left( e^{i \beta t} u_{t(a)} e^{-i \beta t} \right) \quad \text{(since $e^{i \beta t} \cdot x = x$ when $x$ is critical)} \\
 & = \rho_x^\C \left( e^{i \beta t} \cdot u \right)
\end{align*}
Therefore, since $e^{i \beta t}$ is self-adjoint, we have for any $u \in \mathfrak{g}$ and any $X \in T_x \Rep(Q, {\bf v})$
\begin{align*}
\left< (\rho_x^\C)^* \left( e^{i \beta t} \cdot X \right), u \right> & = g \left( e^{i \beta t} \cdot X, \rho_x^\C(u) \right) = g \left( X, e^{i \beta t} \cdot \rho_x^\C(u) \right) \\
 & = g \left( X, \rho_x^\C \left( e^{i \beta t} \cdot u \right) \right) = \left< (\rho_x^\C)^* X, e^{i \beta t} \cdot u \right>
\end{align*}
and so $(\rho_x^\C)^* X = 0$ if and only if $(\rho_x^\C)^* \left( e^{i \beta t} \cdot X \right) = 0$. Since the $G_{\bf v}$-action preserves the space $\nu^{-1}(0)$ then this implies that $\delta x \in S_x$ if and only if $e^{i \beta t} \cdot \delta x \in S_x$ for all $t \in \R$.
\end{proof}

\begin{definition}\label{def:neg-slice-def}
The \emph{negative slice} is
\begin{equation}\label{eqn:neg-slice-def}
S_x^- := \left\{ \delta x \in S_x \, : \, \lim_{t \rightarrow \infty} e^{i \beta t} \cdot \delta x = 0 \right\} .
\end{equation} 
\end{definition}

\begin{remark}
In \cite[Sec. 4.3 \& 4.6]{Kirwan84}, Kirwan uses the downwards gradient flow of the function $\mu_\beta(x) = \mu(x) \cdot \beta$ to define a stratification associated to the norm-square of a moment map. Here we define the negative slice using the upwards gradient flow of $\mu_\beta$ with initial condition in the slice $S_x$.
\end{remark}

Next we study the subset of the slice corresponding to the negative eigenspace of the Hessian. Recall from Section \ref{subsec:hessian} that we have the following description of the tangent space at a critical point on the ambient smooth space $\Rep(Q, {\bf v})$.

\begin{itemize}

\item Since the Hessian is self-adjoint, then the tangent space splits into eigenspaces for the Hessian at $x$.

\item The tangent space also decomposes according to the splitting of the representation into subrepresentations from \eqref{eqn:critical-splitting}. This has the form
\begin{equation*}
T_x \Rep(Q, {\bf v}) \cong \Rep(Q, {\bf v}) \cong \bigoplus_{j,k=1}^n \Hom^1(Q, {\bf v_j}, {\bf v_k}) .
\end{equation*}

\item The negative eigenspaces of the Hessian are characterised by homomorphisms from the subrepresentations of large slope into subrepresentations of small slope. If we order the subrepresentations by increasing slope as in \eqref{eqn:mom-map-block-diagonal}, then Proposition \ref{prop:neg-eigenspace} shows that the negative eigenspaces of the Hessian are
\begin{equation}\label{eqn:smooth-neg-slice}
V(x)^- = \bigoplus_{j > k} \Hom^1(Q, {\bf v_j}, {\bf v_k}) \cap \ker (\rho_x^\C)^* .
\end{equation} 

\end{itemize}

The next result is Lemma \ref{lem:slice-eigenspace-equivalence} which shows that the negative slice corresponds to the intersection of the exponential image of the negative eigenspace of the Hessian with the singular subset $Z$. When studying the singular subset $Z$, the negative slice is more natural since it is defined in terms of the group action, however for the purpose of doing computations the negative eigenspace of the Hessian is easier to deal with since we can describe it explicitly.

First recall that Proposition \ref{prop:neg-eigenspace} shows that $X \in V_{x, \lambda}$ with $\lambda < 0$ implies that $X \in \ker (\rho_x^\C)^*$. Therefore the negative eigenspace is a subset of the slice. Moreover, the negative eigenspace equation reduces to
\begin{equation*}
I \delta \rho_x(\beta)(X) = \lambda X
\end{equation*}
and since $x$ is fixed by $e^{i \beta t}$ and $\Ad_{e^{i \beta t}} (\beta) = \beta$, we have
\begin{equation*}
I \delta \rho_x(\beta) \left(e^{i \beta t} \cdot X \right) = \lambda \left( e^{i \beta t} \cdot X \right)  
\end{equation*}
Therefore $e^{i \beta t} \cdot X \in V_{x, \lambda}$ if and only if $X \in V_{x, \lambda}$. Moreover, we have
\begin{equation*}
\left. \frac{d}{dt} \right|_{t = 0} \left\| e^{i \beta t} \cdot X \right\|^2 = 2 g \left( I \delta \rho_x( \beta)(X), X \right) . 
\end{equation*}
If $X \in \ker (\rho_x^\C)^*$ then this expression is negative if and only if $X \in \bigoplus_{\lambda < 0} V_{x, \lambda}$. Since there are only finitely many negative eigenvalues $\lambda_1 < \cdots < \lambda_k < 0 \leq \lambda_{k+1} \leq \cdots \leq \lambda_n$, then $2 \Re \left< I \delta \rho_x( \beta)(X), X \right> \leq \lambda_k \| X \|^2 < 0$, and so $e^{i \beta t} \cdot X$ converges exponentially to zero if and only if $X \in V_{x, \lambda}$. Therefore we have proved the following equivalence.

\begin{lemma}\label{lem:slice-eigenspace-equivalence}
Let $x$ be a critical point for $\| \mu - \alpha \|^2$ on $\Rep(Q, {\bf v})$. Then $S_x^- = \bigoplus_{\lambda < 0} V_{x, \lambda}$. 
\end{lemma}
Therefore we have the following equivalent definition of the negative slice on the subset $\nu^{-1}(0)$.

\begin{definition}\label{def:neg-slice}
Let $x \in Z$ be a critical point for $\| \mu - \alpha \|^2$ and let $V(x)^- = \bigoplus_{\lambda < 0} V_\lambda$ denote the negative eigenspace of the Hessian at $x$ on the smooth space $\Rep(Q, {\bf v})$. The \emph{negative slice} through $x \in Z$ is 
\begin{equation*}
S_x^- := V(x)^- \cap S_x .
\end{equation*}
\end{definition}

There is also a local slice theorem for the restriction to the negative slice, which we use in Section \ref{subsec:flow-classification}. Here we order the subrepresentations for the critical point by the condition that $j > k$ if and only if $\slope_\alpha(Q, {\bf v_j}) > \slope_\alpha(Q, {\bf v_k})$ (see \eqref{eqn:mom-map-block-diagonal}). 

\begin{lemma}\label{lem:neg-slice-thm}
Let $x \in Z$ be critical for $\| \mu - \alpha \|^2$ and let $\delta x \in \bigoplus_{j > k}\Hom^1(Q, {\bf v_j}, {\bf v_k})$ such that  $x + \delta x \in Z$ is in the neighbourhood from Corollary \ref{cor:singular-slice-thm}. Then there exists $g \in G_{\bf v}$ such that
\begin{equation*}
g \cdot (x + \delta x) - x \in S_x^- .
\end{equation*}
\end{lemma}

\begin{proof}
Let $\rho_x^-$ denote the restriction of $\rho_x^\C$ to the subspace $\mathfrak{g}_\C^- := \bigoplus_{j > k} \Hom^0(Q, {\bf v_j}, {\bf v_k})$. Note that
\begin{equation*}
\rho_x^- : \mathfrak{g}_\C^- \rightarrow \bigoplus_{j > k} \Hom^1(Q, {\bf v_j}, {\bf v_k})
\end{equation*}

We have the orthogonal decomposition $\bigoplus_{j > k}\Hom^1(Q, {\bf v_j}, {\bf v_k}) \cong \im \rho_x^- \oplus \ker (\rho_x^-)^*$ as well as the \emph{a priori} results
\begin{align}
\im \rho_x^\C & = \bigoplus_{j,k} \im \rho_x^\C \cap \Hom^1(Q, {\bf v_j}, {\bf v_k}) \label{eqn:image-decomposition} \\
\im \rho_x^- & = \im \rho_x^\C \cap \bigoplus_{j > k}\Hom^1(Q, {\bf v_j}, {\bf v_k}) \\
\ker (\rho_x^-)^* & \supseteq \ker (\rho_x^\C)^* \cap \bigoplus_{j > k}\Hom^1(Q, {\bf v_j}, {\bf v_k}) \label{eqn:kernel-inclusion}.
\end{align}
Note that any $X \in \bigoplus_{j > k} \Hom^1(Q, {\bf v_j}, {\bf v_k})$ can be written $X = X_1 + X_2$ with $X_1 \in \im \rho_x^\C$ and $X_2 \in \ker (\rho_x^\C)^*$. The equality \eqref{eqn:image-decomposition} shows that $X_1 = X_1^- + X_1^+$ with $X_1^- \in \im \rho_x^\C \cap \bigoplus_{j > k} \Hom^1(Q, {\bf v_j}, {\bf v_k})$ and $X_1^+ \in \im \rho_x^\C \cap \bigoplus_{j \leq k} \Hom^1(Q, {\bf v_j}, {\bf v_k})$. Therefore we have $X_1^+ \perp X$, $X_1^+ \perp X_1^-$ and $X_1^+ \perp X_2$, which implies that
\begin{equation*}
0 = \left< X_1^+, X - X_1^- \right> = \left< X_1^+, X_1^+ + X_2 \right> = \left< X_1^+, X_1^+ \right>
\end{equation*}
and so $X_1^+ = 0$ and $X_2 \in \ker (\rho_x^\C)^* \cap \bigoplus_{j > k} \Hom^1(Q, {\bf v_j}, {\bf v_k})$. Therefore
\begin{equation*}
\bigoplus_{j > k}\Hom^1(Q, {\bf v_j}, {\bf v_k}) = \left( \im \rho_x^\C \cap \bigoplus_{j > k}\Hom^1(Q, {\bf v_j}, {\bf v_k}) \right) \oplus \left( \ker (\rho_x^\C)^* \cap \bigoplus_{j > k}\Hom^1(Q, {\bf v_j}, {\bf v_k}) \right)
\end{equation*}
Putting all of this together gives us 
\begin{align*}
\im \rho_x^- \oplus \ker (\rho_x^-)^* & \cong \bigoplus_{j > k}\Hom^1(Q, {\bf v_j}, {\bf v_k}) \\
 & = \left( \im \rho_x^\C \oplus \ker (\rho_x^\C)^* \right) \cap \bigoplus_{j > k}\Hom^1(Q, {\bf v_j}, {\bf v_k}) \\
 & = \left( \im \rho_x^\C \cap \bigoplus_{j > k}\Hom^1(Q, {\bf v_j}, {\bf v_k}) \right) \oplus \left( \ker (\rho_x^\C)^* \cap \bigoplus_{j > k}\Hom^1(Q, {\bf v_j}, {\bf v_k}) \right) \\
 & \subseteq \im \rho_x^- \oplus \ker (\rho_x^-)^*
\end{align*}
then \eqref{eqn:kernel-inclusion} must be an equality, and so $\ker (\rho_x^-)^* \cong \ker (\rho_x^\C)^* \cap \bigoplus_{j > k}\Hom^1(Q, {\bf v_j}, {\bf v_k}) = V(x)^-$, which implies that the function
\begin{align*}
\zeta^- : (\ker \rho_x^-)^\perp \times V(x)^- & \rightarrow  \bigoplus_{j > k}\Hom^1(Q, {\bf v_j}, {\bf v_k}) \\
 (u, \delta y) & \mapsto e^u \cdot (x + \delta y)
\end{align*}
is a local diffeomorphism. Therefore, if $\delta x \in \bigoplus_{j > k}\Hom^1(Q, {\bf v_j}, {\bf v_k})$ is small enough, then $e^u \cdot (x + \delta y) = x + \delta x \in Z$ for some $u \in (\ker \rho_x^-)^\perp$ and $\delta y \in V(x)^-$. Therefore $x + \delta y \in Z$ and so $\delta y \in S_x^-$.
\end{proof}

From the definition of the negative slice, we can classify the isomorphism classes in $S_x^-$. Corollary \ref{cor:iso-classes-correspond} then shows that this classification also applies to $W_x^-$.

\begin{lemma}\label{lem:neg-slice-filtration}
Let $x$ be a critical point of $\| \mu - \alpha \|^2$ on $Z$ and let $x = x_1 \oplus \cdots \oplus x_n$ be the decomposition as a direct sum of stable representations as in \eqref{eqn:critical-decomposition}, ordered so that $\slope_\alpha(x_j) < \slope_\alpha(x_k)$ for all $j < k$. Then every $y \in S_x^-$ admits a  filtration $y_1 \subset \cdots \subset y_n$ such that each quotient $y_k / y_{k-1}$ is isomorphic to $x_k$ for $k=1, \ldots, n$. Conversely, let $y \in Z$ be a representation admitting a filtration $y_1 \subset \cdots \subset y_n$ such that each quotient $y_k / y_{k-1}$ is isomorphic to $x_k$ for $k=1, \ldots, n$. Then there exists $g \in G_{\bf v}$ such that $g \cdot y \in S_x^-$.
\end{lemma}

\begin{proof}
Given $y \in S_x^-$, the existence of the filtration follows directly from the definition of $S_x^-$ as a subset of $V_x^- \subset \Rep(Q, {\bf v})$ (cf. \eqref{eqn:smooth-neg-slice}). Conversely, given a representation $y \in Z$ admitting a filtration $y_1 \subset \cdots \subset y_n$ such that each quotient $y_k / y_{k-1}$ is isomorphic to $x_k$ for $k=1, \ldots, n$, there exists $g_1 \in G_{\bf v}$ such that $g_1 \cdot y$ is in the neighbourhood where Lemma \ref{lem:neg-slice-thm} applies, and so there exists $g_2 \in G_{\bf v}$ such that $g_2 \cdot g_1 \cdot y \in S_x^-$.
\end{proof}

\begin{definition}\label{def:slice-bundle}
Let $C \subset \Rep(Q, {\bf v})$ be a critical set for $\| \mu - \alpha \|^2$ and let $Z$ be a closed $G_{\bf v}$-invariant subset. The \emph{negative slice bundle} is $S_C^- := V_C^- \cap Z$ and the \emph{unstable bundle} is $W_C^- \cap Z$ (which we also denote by $W_C^-$). 
\end{definition}

For the special case of the stability parameter from Definition \ref{def:aasp}, the critical point $x$ induces a decomposition $\Vect(Q, {\bf v}) \cong \Vect(Q, {\bf v_1}) \oplus \Vect(Q, {\bf v_2})$ as in Proposition \ref{prop:singular-critical-decomp}.  The next lemma shows that the negative slice equations simplify on the space $\nu^{-1}(0)$ associated to a set $\mathcal{R}$ of relations. In particular, the space $S_x^-$ (which \emph{a priori} is a singular subset of the vector space $V_x^-$) is a vector space for this choice of stability parameter.

\begin{lemma}\label{lem:slice-linearises}
Let $\alpha = \alpha(Q, {\bf v})$ be the stability parameter from Definition \ref{def:aasp}. Then
\begin{equation*}
S_x^- = \Hom^1(Q, {\bf v_2}, {\bf v_1}) \cap \ker (\rho_x^\C)^* \cap \ker d \nu_x .
\end{equation*}
\end{lemma}

\begin{proof}
The definition of $S_x^-$ together with \eqref{eqn:smooth-neg-slice} shows that
\begin{equation*}
S_x^- = \left\{ \delta x \in \Hom^1(Q, {\bf v_2}, {\bf v_1}) \cap \ker (\rho_x^\C)^* \, : \, x + \delta x \in \nu^{-1}(0) \right\} .
\end{equation*}
The polynomial $\nu(x+\delta x)$ has the following expansion.
\begin{equation*}
\nu(x + \delta x) = \nu(x) + d\nu_x(\delta x) + \text{(terms quadratic or higher in $\delta x$)} .
\end{equation*}
Since $\delta x \in \Hom^1(Q, {\bf v_2}, {\bf v_1})$ then $\im \delta x \subset \ker \delta x$ and so all the higher order terms vanish. Since $x \in \nu^{-1}(0)$ then the condition $\nu(x + \delta x) = 0$ simplifies to $d \nu_x(\delta x) = 0$.
\end{proof}

\section{Gradient flow lines between critical points}\label{sec:gradient-Nakajima}

This section contains the main results of the paper relating flow lines in symplectic geometry to the Hecke correspondence in algebraic geometry. Section \ref{subsec:scattering} contains the bulk of the analysis to show that the isomorphism classes of representations in the unstable set $W_x^-$ are in one-to-one correspondence with the isomorphism classes in the negative slice $S_x^-$ (Corollary \ref{cor:iso-classes-correspond}).  The methods used to prove this result also lead to a proof that a neighbourhood of the zero section in the negative slice bundle $S_C^-$ is homeomorphic to a neighbourhood of the zero section in the unstable bundle $W_C^-$, and that this statement remains true on restricting to any closed $G_{\bf v}$-invariant subset (Theorem \ref{thm:homeo-slice-bundle} and Corollary \ref{cor:homeo-restrict-to-singular}). 

After dealing with the analytic preliminaries, in Section \ref{subsec:flow-classification} we use the homological algebra of representations of quivers to give an algebraic criterion for two critical points to be connected by a flow line and interpret the Hecke correspondence in terms of pairs of critical points at adjacent critical levels which are connected by a flow line (Theorem \ref{thm:miniscule-flow-hecke}). 

\subsection{The relationship between the unstable set and the negative slice}\label{subsec:scattering}

The goal of this section is to construct a $K_{\bf v}$-equivariant homeomorphism between the negative slice bundle $S_C^-$ and the unstable bundle $W_C^-$ (Theorem \ref{thm:homeo-slice-bundle}). Since the construction is done purely in terms of the $G_{\bf v}$ action then it is sufficient to prove the result first on the manifold $\Rep(Q, {\bf v})$ and then afterwards restrict to a closed $G_{\bf v}$-invariant subset $Z \subset \Rep(Q, {\bf v})$.

\subsubsection{The modified flow in a neighbourhood of a critical point}\label{subsec:modified-flow}

In this section we define a new flow called the \emph{modified flow} which is $K_{\bf v}$-related to the gradient flow \eqref{eqn:neg-grad-flow}. The reason for doing this is that in addition to the distance-decreasing formula for the metric flow \eqref{eqn:distance-decreasing}, the modified flow also satisfies a distance-decreasing result for the action of $e^{-i \beta t}$ (cf. Lemma \ref{lem:modified-distance-decreasing}) which we need to carry out the procedure of Section \ref{subsec:reverse-scattering}.

Let $x$ be a critical point of $\| \mu - \alpha \|^2$ on $\Rep(Q, {\bf v})$ and let $U$ be a neighbourhood of $x$ on which the local slice result of Lemma \ref{lem:local-slice} applies, so that there is a neighbourhood $V$ of $(0,0)$ in $(\ker \rho_x^\C)^\perp \oplus \ker (\rho_x)^*$ and a diffeomorphism $\zeta : V \rightarrow U$ given by $\zeta(u, \delta x) = e^u \cdot (x + \delta x)$. Let $\beta = \mu(x) \in \mathfrak{k}^* \cong \mathfrak{k}$.

Let $K_\beta \subset K_{\bf v}$ denote the isotropy subgroup of $\beta \in \mathfrak{k}$ via the adjoint action of $K_{\bf v}$ on $\mathfrak{k}$. The critical point $x$ has an associated Harder-Narasimhan filtration and we define $P_\beta \subset G_{\bf v}$ as the subgroup preserving this filtration. $K_\beta$ is the maximal compact subgroup of the Levi subgroup of $P_\beta$. The following is the analog of the distance-decreasing property of the flow from \eqref{eqn:distance-decreasing}.

\begin{lemma}\label{lem:modified-distance-decreasing}
Given $g \in P_\beta$, let $g_t = e^{-i \beta t} g e^{i \beta t}$ and define $h_t = g_t^* g_t$. Then $\frac{d}{dt} \sigma(h_t) \leq 0$.
\end{lemma}

The group $G_{\bf v}$ can be written as a fibre product $G_{\bf v} \cong K_{\bf v} \times_{K_\beta} P_\beta$. Any homomorphism from a semistable representation of slope $\nu_1$ to a semistable representation of slope $\nu_2 < \nu_1$ must be zero (cf. \cite[Lem. 2.3]{Reineke03}) and so $(\mathfrak{g}_{\bf v})_x \subset \mathfrak{p}_\beta$ (see \cite[Lem. 20]{HaradaWilkin11}). Therefore $(\mathfrak{p}_\beta)^\perp \subset (\mathfrak{g}_{\bf v})_x^\perp$. Define $(\mathfrak{p}_\beta)_x^\perp := (\mathfrak{g}_{\bf v})_x^\perp \cap \mathfrak{p}_\beta$. Since the slice equations are $K_{\bf v}$-invariant then $k \cdot S_x = S_{k \cdot x}$ for all $k \in K_{\bf v}$ and so we can define $S_\beta := K_\beta \times_{K_x} S_x$. Then $K_\beta$ acts on $K_{\bf v} \times S_\beta \times (\mathfrak{p}_\beta)_x^\perp$ by $k_\beta \cdot (k, [k', x + \delta x, u]) = (k k_\beta^{-1}, [k_\beta k' k_\beta^{-1}, k_\beta \cdot (x + \delta x), \Ad_{k_\beta}(u)])$. The local slice result Lemma \ref{lem:local-slice} can then be refined to show that the map
\begin{align}\label{eqn:refined-slice}
\begin{split}
K_{\bf v} \times_{K_\beta} \left( S_\beta \times (\mathfrak{p}_\beta)_x^\perp \right) & \stackrel{\zeta}{\rightarrow} \Rep(Q, {\bf v}) \\
[k, k', x + \delta x, u] & \mapsto k \cdot k' \cdot e^u \cdot (x + \delta x)
\end{split}
\end{align}
is a $K_{\bf v}$-equivariant local diffeomorphism in a neighbourhood of $(\id, 0, 0)$, where $k'' \in K_{\bf v}$ acts on $K_{\bf v} \times_{K_\beta} (S_\beta \times (\mathfrak{p}_\beta)_x^\perp)$ by $k'' \cdot [k, k', x + \delta x, u] = [ k'' k, k', x + \delta x, u]$.

On the image $\zeta \left( S_\beta \times (\mathfrak{p}_\beta)_x^\perp \right) \subset U$, define $\gamma(y)_-$ as the component of $i \mu(y)$ in $(\mathfrak{p}_\beta)^\perp$. Note that since the action of $K_\beta$ preserves $\zeta \left( S_\beta \times (\mathfrak{p}_\beta)_x^\perp \right)$ and conjugation by $K_\beta$ preserves $(\mathfrak{p}_\beta)^\perp$ then $\gamma(k \cdot y)_- = \Ad_k \gamma(y)_-$ for all $k \in K_\beta$. Now define 
\begin{equation}\label{eqn:def-gamma}
\gamma(y) = \gamma(y)_- - \gamma(y)_-^* \in \mathfrak{k}_{\bf v} .
\end{equation}
Again we see that $\gamma$ is $K_\beta$-equivariant. Using the fibre product structure of \eqref{eqn:refined-slice} we can extend $\gamma$ to a $K_{\bf v}$-equivariant map $U \rightarrow \mathfrak{k}$. Note also that $-i \mu(y) + \gamma(y) \in \mathfrak{p}_\beta$ for all $y \in U$.  

The following lemma shows that $\gamma$ satisfies a Lipschitz bound analogous to that of Lemma \ref{lem:mu-h-Lipschitz}.
\begin{lemma}\label{lem:gamma-h-lipschitz}
Fix a critical point $x$ for $\| \mu - \alpha \|^2$. Given any $g \in G_{\bf v}$ define $h = g^* g$. Then there exist neighbourhoods $U$ of $x$ in $\Rep(Q, {\bf v})$ and $N$ of $K_{\bf v}$ in $G_{\bf v}$, and a constant $C$ such that $y \in U$ and $g \in N$ implies that
\begin{equation}\label{eqn:gamma-lipschitz-bound}
\| \Ad_{g^{-1}} \left( \gamma(g \cdot y) \right) - \gamma(y) \| \leq C \sqrt{\sigma(h)}
\end{equation} 
where $\sigma(h) = \tr(h + h^{-1} - 2 \id)$ is the distance function from \eqref{eqn:def-sigma}.
\end{lemma}

\begin{proof}
Given $g \in G_{\bf v} = \times_{k \in I} \GL(n_k, \C)$, the Cartan decomposition determines $k \in K_{\bf v}$ and $u \in i \mathfrak{k}_{\bf v}$ such that $g = k e^u$. The $K_{\bf v}$-equivariance of $\gamma$ then gives us
\begin{equation*}
\Ad_{g^{-1}} \gamma(g \cdot y) - \gamma(y) = \Ad_{e^{-u}} \gamma(e^u \cdot y) - \gamma(y)
\end{equation*}
which is a $C^\infty$ function of $u$ and $y$, which is zero when $u = 0$. Therefore, for each $y$ there exists a neighbourhood $N'$ of zero in $i \mathfrak{k}_{\bf v}$ and a constant $C'(y)$ (depending continuously on $y$) such that 
\begin{equation*}
\| \Ad_{e^{-u}} \gamma(e^u \cdot y) - \gamma(y) \| \leq C'(y) d(\id, e^u)
\end{equation*}
for all $u \in N'$, where $d$ denotes the geodesic distance in $G_{\bf v}$ with respect to the left-invariant metric determined by the norm $\| u \| = \tr(u u^*)$ on the Lie algebra $\mathfrak{g}_{\bf v}$. Since $C'(y)$ depends continuously on $y$ then \eqref{eqn:distance-sigma1} and \eqref{eqn:distance-sigma2} imply that there is a neighbourhood $U$ of $x$ in $\Rep(Q, {\bf v})$ and a uniform constant $C$ such that
\begin{equation*}
\| \Ad_{e^{-u}} \gamma(e^u \cdot y) - \gamma(y) \| \leq C \sqrt{\sigma(h)}
\end{equation*}
for all $y \in U$.
\end{proof}

From the definition of $\gamma(y)$ as a sum of components of $\mu(y)$, we have the inequality
\begin{equation}\label{eqn:gamma-slice-bound}
\| \gamma(y) \| = \| \gamma(y) - \gamma(x) \| \leq \| \mu(y) - \mu(x) \| = \| \mu(y) - \beta \| . 
\end{equation}

\begin{definition}\label{def:modified-flow}
Given $y_0 \in U$, the \emph{modified flow} with initial condition $y_0$ is the solution to
\begin{equation}\label{eqn:modified-flow-def}
\frac{dy_t}{dt} = -I \rho_{y_t}(\mu(y_t) - \alpha) + \rho_{y_t}(\gamma(y_t)) . 
\end{equation}
for all $t$ such that $y_t \in U$.
\end{definition}
Note that if $x$ is a critical point of $\| \mu - \alpha \|^2$ then $\mu(x) = \beta \in \mathfrak{k}_\beta \subset \mathfrak{p}_\beta$ and so $\gamma(x) = 0$. Therefore $x$ is also a stationary point for the modified flow. Define
\begin{equation*}
W_{mod,x}^- := \{ y_0 \in \Rep(Q, {\bf v}) \, : \, \text{$y_t$ solves \eqref{eqn:modified-flow-def} on $(-\infty, 0]$ and} \, \, \lim_{t \rightarrow - \infty} y_t = x  \} .
\end{equation*}
Given a critical set $C$, define $W_{mod,C}^- := \bigcup_{x \in C} W_{mod,x}^-$.

In a similar way to the downwards gradient flow of $\| \mu - \alpha \|^2$ (cf. \eqref{eqn:group-neg-flow}), a solution $y_t$ to \eqref{eqn:modified-flow-def} satisfies $y_t = g_t \cdot y_0$, where $g_t$ is a solution of
\begin{equation}\label{eqn:group-modified-flow}
\frac{dg_t}{dt} g_t^{-1} = -i (\mu(g_t \cdot y_0) - \alpha) + \gamma(g_t \cdot y_0), \quad g_0 = \id.
\end{equation}
Note that since $-i (\mu(y) - \alpha) + \gamma(y) \in \mathfrak{p}_\beta$ for all $y \in U$, then $g_t \in P_\beta$ for all $t$ such that $g_t \cdot y_0 \in U$. 

\begin{lemma}\label{lem:relate-flows}
Let $y_0 \in U$, let $y_t^{mod}$ be a solution to the modified flow \eqref{eqn:modified-flow-def} with initial condition $y_0$ and let $y_t^{orig}$ be a solution to the original downwards gradient flow of $\| \mu - \alpha \|^2$ given by \eqref{eqn:neg-grad-flow}. Then $y_t^{mod} = s_t \cdot y_t^{orig}$, where $s_t \in K_{\bf v}$ is the solution of 
\begin{equation}\label{eqn:relate-flow-def}
\frac{ds}{dt} s_t^{-1} = \gamma(s_t \cdot y_t^{orig}), \quad s_0 = \id .
\end{equation}
\end{lemma}

\begin{proof}
Let $g_t^{orig}$ be the solution to \eqref{eqn:group-neg-flow} such that $y_t^{orig} = g_t^{orig} \cdot y_0$. Define $g_t^{mod} = s_t \cdot g_t^{orig}$. We then have
\begin{align*}
\frac{d g_t^{mod}}{dt} (g_t^{mod})^{-1} & = \frac{ds}{dt} s_t^{-1} + \Ad_{s_t} \left( \frac{d g_t^{orig}}{dt} (g_t^{orig})^{-1} \right) \\
 & = \gamma(s_t \cdot y_t^{orig}) - i \Ad_{s_t} \left( \mu(y_t^{orig}) - \alpha \right) \\
 & = \gamma(g_t^{mod} \cdot y_0) - i \left( \mu(g_t^{mod} \cdot y_0) - \alpha \right)
\end{align*}
and so $g_t^{mod}$ is a solution to \eqref{eqn:group-modified-flow}, hence $y_t^{mod} := g_t^{mod} \cdot y_0$ is a solution to \eqref{eqn:modified-flow-def}.
\end{proof}

\subsubsection{Exponential convergence of the reverse flow}

In this section we prove that a solution to the modified flow \eqref{eqn:modified-flow-def} which converges to a critical point as $t \rightarrow - \infty$ must converge at an exponential rate. The idea of the proof is similar to the case of the Yang-Mills-Higgs flow studied in \cite{wilkin-YMH-flow-lines}, however there are some simplifications here since for the Yang-Mills-Higgs flow we prove convergence in the $L_k^2$ norm for all $k$, but for the finite-dimensional space $\Rep(Q, {\bf v})$ we only need to prove convergence in the topology induced by the metric on $\Rep(Q, {\bf v})$. We can then use the exponential convergence to prove Proposition \ref{prop:modified-flow-homeo}, which shows that the unstable bundles for the gradient flow \eqref{eqn:neg-grad-flow} and the modified flow \eqref{eqn:modified-flow-def} are $K_{\bf v}$-equivariantly homeomorphic. 

Let $x$ be a non-minimal critical point of $f = \| \mu - \alpha \|^2$ on $\Rep(Q, {\bf v})$ and let $\beta = \mu(x)$. The local slice result of (Lemma \ref{lem:local-slice}) proves the existence of a neighbourhood $U$ of $x$ such that any $y \in U$ can be written $y = e^u \cdot (x + z)$ for $u \in \mathfrak{g}_x^\perp$ and $z \in S_x$. Moreover, this description is unique if we restrict to a neighbourhood of zero in $\mathfrak{g}_x^\perp \times S_x$. Further decompose $z = z_{\geq 0} + z_-$ according to the eigenspaces for the action of $e^{i \beta}$ on $S_x$. 

The following lemma collects some results of Kirwan \cite[Sec. 10]{Kirwan84} which will be used in the sequel.
\begin{lemma}
With respect to the decomposition $y = e^u \cdot (x + z_{\geq 0} + z_-)$ we have
\begin{enumerate}
\item $f(e^u \cdot (x + z_{\geq 0})) \geq f(x)$,

\item $\grad f(e^u \cdot (x + z_{\geq 0}))$ is tangent to $\{ z_- = 0 \}$, and

\item \cite[(10.14)]{Kirwan84} Let $y_t$ be a solution to the gradient flow \eqref{eqn:neg-grad-flow} which converges to $x$ as $t \rightarrow - \infty$, and write $y_t = e^{u_t} \cdot (x + z_{\geq 0}^t + z_-^t)$. Then there exists $\delta > 0$ and positive constants $K_1$ and $K_2$ such that if $\| y_0 - x \| < \delta $ then $\| z_-^t \| \leq K_1 e^{K_2 t}$ for all $t \leq 0$.
\end{enumerate}
\end{lemma}

Using this, we can show that the quantity $f(x) - f(y_t)$ decreases exponentially.

\begin{lemma}
There exists a neighbourhood $U$ of $x$ and positive constants $K_1'$ and $K_2'$ such that if $y_t$ is a solution to \eqref{eqn:neg-grad-flow} with initial condition $y_0 \in U$ which converges to $x$ as $t \rightarrow - \infty$, then $f(x) - f(y_t) \leq K_1' e^{K_2' t}$ for all $t \leq 0$. 
\end{lemma}

\begin{proof}
Let $y = e^u \cdot (x + z_{\geq 0} + z_-)$ as above. Since $x$ is a critical point then for all $\varepsilon > 0$ there exists $\delta > 0$ such that $\| y - x \| < \delta$ implies that 
\begin{equation*}
f(e^u \cdot (x + z_{\geq 0} + z_-)) - f(e^u \cdot (x + z_{\geq 0})) \geq - \varepsilon \| z_- \| .
\end{equation*}
Therefore
\begin{align*}
f(e^u \cdot (x + z_{\geq 0} + z_-)) - f(x) & = f(e^u \cdot (x + z_{\geq 0} + z_-)) - f(e^u \cdot (x + z_{\geq 0})) \\
 & \quad \quad + f(e^u \cdot (x  + z_{\geq 0})) -  f(x) \\
 & \geq f(e^u \cdot (x + z_{\geq 0} + z_-)) - f(e^u \cdot (x + z_{\geq 0})) \\
 & \geq - \varepsilon \| z_- \| \\
 & \geq - \varepsilon K_1 e^{K_2 t} 
\end{align*}
and so the result follows after setting $K_1' = \varepsilon K_1$ and $K_2' = K_2$.
\end{proof}

\begin{lemma}\label{lem:exponential-convergence}
There exists a neighbourhood $U$ of $x$ and positive constants $C_1$ and $C_2$ such that if $y_t$ is a solution to \eqref{eqn:neg-grad-flow} with initial condition $y_0 \in U$ which converges to $x$ as $t \rightarrow - \infty$, then $\| y_t - x \| \leq C_1 e^{C_2 t}$ for all $t \leq 0$. 
\end{lemma}

\begin{proof}
After possibly shrinking the neighbourhood $U$ of the previous lemma, we can apply the Lojasiewicz inequality method of Simon \cite{Simon83} to show that there exist constants $C > 0$ and $\theta \in (0, 1)$ such that
\begin{equation}\label{eqn:flow-length-bound}
\| y_t - x \| \leq \int_{-\infty}^t \| \grad f(y_s) \| \, ds \leq \frac{1}{C \theta} \left( f(x) - f(y_t) \right)^\theta 
\end{equation}
and the previous lemma shows that
\begin{equation}\label{eqn:energy-difference-bound}
\frac{1}{C \theta} \left( f(x) - f(y_t) \right)^\theta \leq \frac{1}{C \theta} (K_1')^\theta e^{\theta K_2' t} .
\end{equation}
The result then follows by combining \eqref{eqn:flow-length-bound} and \eqref{eqn:energy-difference-bound}.
\end{proof}

Now we can show that this result also applies to the modified flow. Let $y_t$ be a solution to the gradient flow which converges to a critical point $x$ as $t \rightarrow \infty$ and let $z_t$ be the associated solution to the modified flow \eqref{eqn:modified-flow-def} with initial condition $y_0$. Lemma \ref{lem:relate-flows} shows that $z_t = s_t \cdot y_t$ for $s_t \in K_{\bf v}$. Since the metric is $K_{\bf v}$-invariant and both $\gamma$ and $\mu$ are $K_{\bf v}$-equivariant then
\begin{align*}
\left\| \frac{dz}{dt} \right\| & = \left\| - I \rho_{z_t} (\mu(z_t) - \alpha) + \gamma(z_t) \right\| \\
 & =  \left\| s_t \left( - I \rho_{y_t} (\mu(y_t) - \alpha) + \gamma(y_t) \right) s_t^{-1} \right\| \\
 & =  \left\| - I \rho_{y_t} (\mu(y_t) - \alpha) + \gamma(y_t) \right\| \leq \left\| \frac{dy}{dt} \right\| + \left\| \gamma(y_t) \right\|
\end{align*}

Since $\gamma$ is smooth and satisfies $\gamma(x) = 0$ then (after possibly shrinking the neighbourhood $U$) there exists a Lipschitz constant $C'$ such that 
\begin{equation}\label{eqn:gamma-critical-lipschitz}
\| \gamma(y) \| \leq C' \| y - x \| \quad \text{for all $y \in U$.}
\end{equation}
Therefore, along the solution $y_t$ to \eqref{eqn:neg-grad-flow} we have $\| \gamma(y_t) \| \leq C' C_1 e^{C_2 t}$, and so there are positive constants $C_1'$ and $C_2'$ such that
\begin{equation}\label{eqn:modified-flow-length}
\int_{-\infty}^t \left\| \frac{dz_s}{ds} \right\| \, ds \leq \int_{-\infty}^t \left\| \frac{dy_s}{ds} \right\| \, ds + \int_{-\infty}^t C' C_1 e^{C_2 s} \, ds \leq C_1' e^{C_2' t}  .
\end{equation}

\begin{proposition}\label{prop:exponential-modified-convergence}
Let $y_t$ be a solution to the downwards gradient flow \eqref{eqn:neg-grad-flow} that converges to a critical point $x$ as $t \rightarrow - \infty$ and let $z_t = s_t \cdot y_t$ be the associated solution to the modified flow with initial condition $y_0$. Then $\lim_{t \rightarrow - \infty} s_t =: s_\infty \in K_{\bf v}$ exists and there exist positive constants $C_1'$ and $C_2'$ such that 
\begin{equation*}
\| z_t - s_\infty \cdot x \| \leq C_1' e^{C_2' t}
\end{equation*}
for all $t \leq 0$.
\end{proposition}

\begin{proof}
The inequality \eqref{eqn:modified-flow-length} shows that the length of the flow line $\{ z_t : t \leq 0 \}$ converges to zero exponentially as $t \rightarrow - \infty$. Therefore $z_t$ converges exponentially to a unique limit $z_\infty$. Since $z_t \in K_{\bf v} \cdot y_t$ for all $t$, $y_t \rightarrow x$ as $t \rightarrow - \infty$ and $K_{\bf v} \cdot x$ is closed then $z_\infty \in K_{\bf v} \cdot x$ also. Choose $k \in K_{\bf v}$ such that $z_\infty = k \cdot x$. Note that this implies $\gamma(z_\infty) = 0$. 

Since $\frac{ds_t}{dt} s_t^{-1} = \gamma(z_t)$ and $\| \gamma(z_t) \| = \| \gamma(z_t) - \gamma(z_\infty) \| \leq C' \| z_t - z_\infty \|$ is exponentially decreasing, then $s_t$ also converges to a unique limit $s_\infty \in K_{\bf v}$. Since $s_t \cdot y_t \rightarrow s_\infty \cdot x$ then $z_\infty = k \cdot x = s_\infty \cdot x$.
\end{proof}

Now we use the limit $s_\infty$ to define a map between the unstable sets for the gradient flow \eqref{eqn:neg-grad-flow} and the modified flow \eqref{eqn:modified-flow-def}.

\begin{lemma}\label{lem:continuous-relate-unstable-sets}
Let $y_0 \in W_x^-$ and let $s_t$ be the solution of \eqref{eqn:relate-flow-def} for $t \in (-\infty, 0]$ with initial condition $y_0$ and limit $s_\infty(y_0)$. Then $s_\infty(y_0)^{-1} \cdot y_0 \in W_{mod,x}^-$. Moreover, if $k \in K_{\bf v}$ then $s_\infty(k \cdot y_0) = \Ad_k s_\infty(y_0)$. This defines a continuous $K_{\bf v}$-equivariant map $\eta : W_C^- \rightarrow W_{mod, C}^-$. 
\end{lemma}

\begin{proof}
Proposition \ref{prop:exponential-modified-convergence} shows that the solution $z_t$ to the modified flow \eqref{eqn:modified-flow-def} with initial condition $y_0$ converges to $s_\infty \cdot x$ as $t \rightarrow - \infty$. The $K_{\bf v}$-equivariance of the flow then implies that $s_\infty^{-1} \cdot z_t$ is a solution to \eqref{eqn:modified-flow-def} with initial condition $s_\infty^{-1} \cdot y_0$. Therefore $s_\infty^{-1} \cdot z_t \rightarrow s_\infty^{-1} \cdot z_\infty = x$ and so $s_\infty^{-1} \cdot y_0 \in W_{mod,x}^-$. 

The $K_{\bf v}$-equivariance of the flow also implies that for any $k \in K_{\bf v}$ the curve $k \cdot z_t$ is a solution of \eqref{eqn:modified-flow-def} with initial condition $k \cdot y_0$. The equation for $s_t(k \cdot y_0)$ then becomes
\begin{equation*}
\frac{ds_t(k \cdot y_0)}{dt} s_t(k \cdot y_0)^{-1} = \gamma(k \cdot z_t) = \Ad_k \gamma(z_t) = \Ad_k \left( \frac{d s_t(y_0)}{dt} s_t(y_0)^{-1} \right) 
\end{equation*}
and so $s_t(k \cdot y_0) = \Ad_k s_t(y_0)$. Therefore the limits also satisfy $s_\infty(k \cdot y_0) = \Ad_k s_\infty(y_0)$.
\end{proof}

Finally, we prove that the unstable sets for the gradient flow \eqref{eqn:neg-grad-flow} and the modified flow \eqref{eqn:modified-flow-def} are $K_{\bf v}$-equivariantly homeomorphic.

\begin{proposition}\label{prop:modified-flow-homeo}
The map $\eta$ defines a $K_{\bf v}$-equivariant homeomorphism of pairs $\varphi : (W_C^-, W_C^- \setminus C) \rightarrow (W_{mod,C}^-, W_{mod,C}^- \setminus C)$.
\end{proposition}

\begin{proof}
Lemma \ref{lem:continuous-relate-unstable-sets} shows that $\eta$ is well-defined and continuous. Therefore the proof reduces to showing that there is a continuous inverse.

Given any $z_0 \in W_{mod,x}^-$, let $z_t$ denote the solution of \eqref{eqn:modified-flow-def} for $t \in (-\infty,0]$, let $y_t$ denote the solution of the gradient flow of $\| \mu - \alpha \|^2$ given by \eqref{eqn:neg-grad-flow}. Then Lemma \ref{lem:relate-flows} shows that $z_t = s_t \cdot y_t$ for $s_t$ solving \eqref{eqn:relate-flow-def} and that $s_\infty(z_0) \cdot z_0 \in W_x^-$.

The same argument as in Lemma \ref{lem:continuous-relate-unstable-sets} then shows that this defines a continuous $K_{\bf v}$-equivariant map $\eta' : W_{mod,C}^- \rightarrow W_C^-$. Given any $y_0 \in W_C^-$, we compute
\begin{align*}
\eta' \circ \eta(y_0) & = \eta'(s_\infty(y_0)^{-1} \cdot y_0) \\
 & = s_\infty(s_\infty(y_0)^{-1} \cdot y_0) \cdot s_\infty(y_0)^{-1} \cdot y_0 \\
 & = \left( s_\infty(y_0)^{-1} \cdot s_\infty(y_0) \cdot s_\infty(y_0) \right) \cdot s_\infty(y_0)^{-1} \cdot y_0 \quad \text{by Lemma \ref{lem:continuous-relate-unstable-sets}} \\
 & = y_0 .
\end{align*}
An analogous calculation shows that $\eta \circ \eta'(z_0) = z_0$. Therefore $\eta'$ is a continuous inverse to $\eta : W_C^- \rightarrow W_{mod, C}^-$  and so $\eta$ is a $K_{\bf v}$-equivariant homeomorphism.
\end{proof}

\subsubsection{Convergence of the scattering method}

In this section we show how to use the $G_{\bf v}$ action to construct a map $S_x^- \rightarrow W_{mod,x}^-$ for each critical point $x \in C$. The method is to first use the time $t$ linearised gradient flow on $S_x^-$ to flow towards the critical point $x$ (this is the upwards gradient flow of the function $\mu_\beta$ used by Kirwan in \cite{Kirwan84}), and then to flow down for time $t$ using the modified flow \eqref{eqn:modified-flow-def}. The main result of this section is Proposition \ref{prop:convergence-group} which shows that the composition of these two flows converges as $t \rightarrow \infty$. The same idea will also work if we use the gradient flow \eqref{eqn:neg-grad-flow} instead of the modified flow, however we need the modified flow for the reverse construction in the next section and so for consistency we use the modified flow throughout both sections.

A similar process is called the \emph{scattering method} in \cite{Hubbard05} and \cite{Nelson69}, since it originates in the study of wave operators in quantum mechanics. The method in this section is different to that in \cite{Hubbard05} and \cite{Nelson69}, since here we define the linearised flow intrinsically on the negative slice $S_x^-$ (the method of \cite{Hubbard05} projects the gradient flow onto the negative eigenspace of the Hessian) and we use the action of $G_{\bf v}$ and the distance-decreasing formula \eqref{eqn:distance-decreasing} in place of the coordinate transformations and the estimates of \cite{Hubbard05}. The reason for this is so that the entire process is defined intrinsically on any closed $G_{\bf v}$-invariant subset $Z \subset \Rep(Q, {\bf v})$, which is not the case for the method of \cite{Hubbard05} since it requires the subset $Z$ to be smooth in order to (a) define the projection onto the negative eigenspace of the Hessian and (b) define the coordinate transformations which map the gradient vector field to a vector field which differs from the linearised flow by a term of high enough order to make the scattering method converge (cf. \cite[Sec. 3]{Hubbard05}).

Throughout this section we fix a critical point $x$ and define $\beta = \mu(x) - \alpha \in \mathfrak{k}_{\bf v}$. To simplify the notation in this section, we use the following shifted version of the negative slice (cf. Definition \ref{def:neg-slice-def})
\begin{equation*}
S_x^- : = \{ x + \delta x \, : \, \delta x \in S_x, \, \, \lim_{t \rightarrow \infty} e^{i \beta t} \cdot \delta x  = 0 \} .
\end{equation*}
For any point $y_0 \in S_x^-$, let $y_t = \phi^{mod}(y_0, t) = g_t \cdot y_0$ denote the time $t$ modified flow \eqref{eqn:modified-flow-def} with initial condition $y_0$, where $g_t$ solves \eqref{eqn:group-modified-flow}. Now define $f_t = g_t \cdot e^{i \beta t}$ so that $y_t = f_t \cdot e^{-i \beta t} \cdot y_0$ and let $h_t := f_t^* f_t \in G_{\bf v} / K_{\bf v}$ be the associated change of metric. 

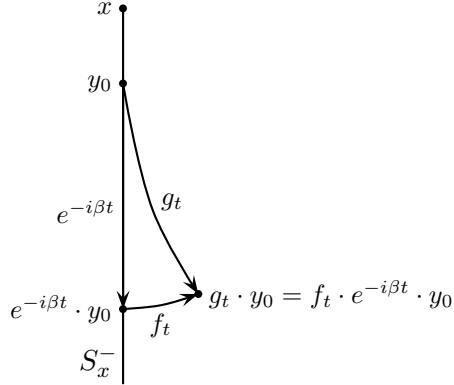
\begin{figure}[ht]
\begin{center}
\begin{pspicture}(2,0)(10,5)
\psline(4,5)(4,0)
\psline[arrowsize=5pt]{->}(4,4)(4,1)
\pscurve[arrowsize=5pt]{->}(4,4)(4.4,2.3)(5,1.2)
\pscurve[arrowsize=5pt]{->}(4,1)(4.5,1.05)(5,1.2)
\psdots[dotsize=3pt](4,5)(4,4)(4,1)(5,1.2)
\uput{4pt}[180](4,5){\small{$x$}}
\uput{4pt}[180](4,4){\small{$y_0$}}
\uput{4pt}[180](4,1){\small{$e^{-i \beta t} \cdot y_0$}}
\uput{4pt}[0](5,1.2){\small{$g_t \cdot y_0 = f_t \cdot e^{-i \beta t} \cdot y_0$}}
\uput{2pt}[180](4,0.3){$S_x^-$}
\uput{3pt}[270](4.5,1.05){\small{$f_t$}}
\uput{3pt}[30](4.4,2.3){\small{$g_t$}}
\uput{3pt}[180](4,2.3){\small{$e^{-i \beta t}$}}
\end{pspicture}
\end{center}
\end{figure}

The same calculation as in \cite[Lem. 3.14]{wilkin-YMH-flow-lines} shows that
\begin{equation}\label{eqn:gauge-difference-derivative}
\frac{df_t}{dt} f_t^{-1} = -i (\mu(g_t \cdot y_0) - \alpha) + \gamma(g_t \cdot y_0) +  f_t(i\beta) f_t^{-1}
\end{equation}
and
\begin{equation}\label{eqn:metric-derivative}
\frac{dh_t}{dt} = -2i h_t (\mu_h(e^{-i \beta t} \cdot y_0) - \alpha) + i \beta h_t + i h_t \beta .
\end{equation}

Next we derive a uniform bound on $\sigma(h_t)$ in terms of $\| y_0 - x \|$.

\begin{lemma}\label{lem:uniform-bound-metric}
There exists $\varepsilon > 0$ such that if $\| e^{-i \beta T} \cdot y_0 - x \|^2 < \varepsilon$ then there exists a constant $K$ (which is uniform over the critical set containing $x$) such that
\begin{equation}\label{eqn:uniform-sigma-bound}
\sigma(h_t) \leq K \| e^{-i \beta T} \cdot y_0 - x \|^2 .
\end{equation}
\end{lemma}

\begin{proof}
Taking the trace of \eqref{eqn:metric-derivative} leads to
\begin{equation*}
\frac{d}{dt} \tr (h_t) = - 2i \tr \left( (\mu_h(e^{-i \beta t} \cdot y_0) - \alpha - \beta) h_t \right) .
\end{equation*}
Therefore we can rewrite the derivative as
\begin{equation*}
\frac{d}{dt} \tr (h_t) = -2i \tr \left( (\mu_h(e^{-i \beta t} \cdot y_0) - \mu(e^{-i \beta t} \cdot y_0)) h_t \right) - 2i \tr \left( (\mu(e^{-i \beta t} \cdot y_0) - \alpha - \beta) h_t \right) .
\end{equation*}
Since $h_t$ is positive definite, then the inequality \eqref{eqn:mu-h-inequality} shows that
\begin{align*}
\frac{d}{dt} \tr (h_t) & \leq  2i \tr \left( (\mu(e^{-i \beta t} \cdot y_0) - \alpha - \beta) h_t \right) \\
 & \leq C \| \mu(e^{-i \beta t} \cdot y_0) - \alpha - \beta \| \tr(h_t) \\
 & \leq C \| e^{-i \beta t} \cdot y_0 - x \|^2 \tr (h_t) \quad \text{by \eqref{eqn:moment-map-slice-quadratic}.}
\end{align*}
The same idea with $\tr(h_t^{-1})$ shows that
\begin{equation*}
\frac{d}{dt} \tr (h_t^{-1}) \leq C \| e^{-i \beta t} \cdot y_0 - x \|^2 \tr (h_t^{-1})
\end{equation*}
Now write the eigenvalues of $i \beta$ in increasing order $\lambda_1 \leq \cdots \leq \lambda_k < 0 \leq \lambda_{k+1} \leq \cdots \leq \lambda_n$. On each critical set $\beta = \mu(x) - \alpha$ is unique up to conjugation by $K_{\bf v}$, and so these eigenvalues are constant on the critical set containing $x$ (we use this below to show that the constant $K$ is uniform on the critical set). The following estimate holds for any $y_0 \in S_x^-$ and any $t \leq T$
\begin{equation*}
\| e^{-i \beta t} \cdot y_0 - x \| = \| e^{i \beta (T-t)} \cdot e^{-i \beta T} \cdot y_0 - x \| \leq e^{\lambda_k (T-t)} \| e^{-i \beta T} \cdot y_0 - x \| .
\end{equation*}
Therefore the derivative of $\sigma(h_t) = \tr(h_t) + \tr(h_t^{-1}) - 2 \rank (Q, {\bf v})$ satisfies the inequality
\begin{align*}
\frac{d}{dt} \sigma(h_t) & \leq 2 C e^{2 \lambda_k (T-t)} \| e^{-i \beta T} \cdot y_0 - x \|^2 \left( \tr (h_t) + \tr (h_t^{-1}) \right) \\
 & = 2 C e^{2 \lambda_k (T-t)} \| e^{-i \beta T} \cdot y_0 - x \|^2 \sigma(h_t) + 4 C \rank(Q, {\bf v}) e^{2 \lambda_k (T-t)} \| e^{-i \beta T} \cdot y_0 - x \|^2 
\end{align*}
An application of Gronwall's inequality then shows that for all $0 \leq t \leq T$ we have
\begin{align*}
\sigma(h_t) & \leq \left( 4 C \rank(Q, {\bf v}) \| e^{-i \beta T} \cdot y_0 - x \|^2  \int_0^t e^{2\lambda_k (T-s)} \, ds \right) \exp \left( 2 C \| e^{-i \beta T} \cdot y_0 - x \|^2 \int_0^t e^{2 \lambda_k (T-s)} \, ds \right) \\
 & \leq  \left( 4 C \rank(Q, {\bf v}) \| e^{-i \beta T} \cdot y_0 - x \|^2  \frac{1}{2 |\lambda_k|} \right) \exp \left( 2 C \| e^{-i \beta T} \cdot y_0 - x \|^2 \frac{1}{2 |\lambda_k|} \right)
\end{align*}
Therefore if $\| e^{-i \beta T} \cdot y_0 - x \|^2 < \varepsilon$ then there exists a constant $K$ such that
\begin{equation*}
\sigma(h_t) \leq K \| e^{-i \beta T} \cdot y_0 - x \|^2 \quad \text{for all $0 \leq t \leq T$.}
\end{equation*}
Moreover, since the constant $K$ depends only on $\varepsilon$,  $\rank(Q, {\bf v})$ and the eigenvalue $\lambda_k$, which are all constant on each critical set, then $K$ is uniform over the critical set containing $x$.
\end{proof}

Now we can use this estimate to derive a bound on the distance from $f_t$ to the identity in $G_{\bf v}$. In the following we fix a norm on $\mathfrak{g}_{\bf v}$ given by $\| u \|^2 := \tr(u u^*)$, and use $d(g_1, g_2)$ to denote the geodesic distance between $g_1, g_2 \in G_{\bf v}$ with respect to the left-invariant metric on $G_{\bf v}$ associated to the norm on $\mathfrak{g}_{\bf v}$.

\begin{lemma}\label{lem:group-distance}
\begin{equation}\label{eqn:group-distance}
d(\id, f_t) \leq C_1 \| e^{-i \beta T} \cdot y_0 - x \| + C_2 \| e^{-i \beta T} \cdot y_0 - x \|^2 
\end{equation}
\end{lemma}

\begin{proof}
Equation \eqref{eqn:gauge-difference-derivative} shows that
\begin{align*}
f_t^{-1} \frac{df_t}{dt} & =  -i f_t^{-1} (\mu(g_t \cdot y_0) - \alpha) f_t + f_t \gamma(g_t \cdot y_0) f_t^{-1} + i \beta \\
 & = -i (\mu_{h_t}(e^{-i \beta t} \cdot y_0) - \alpha) + f_t \gamma(g_t \cdot y_0) f_t^{-1} + i \beta \\
 & = -i \left( \mu_{h_t}(e^{-i \beta t} \cdot y_0) - \mu(e^{-i \beta t} \cdot y_0) \right) -i \left( \mu(e^{-i \beta t} \cdot y_0) - \alpha - \beta \right) \\
 & \quad \quad \quad + \left( f_t \gamma(g_t \cdot y_0) f_t^{-1} - \gamma(e^{-i \beta t} \cdot y_0) \right) + \gamma(e^{-i \beta t} \cdot y_0)
\end{align*}
From the bounds \eqref{eqn:mu-h-Lipschitz}, \eqref{eqn:gamma-lipschitz-bound}, \eqref{eqn:moment-map-slice-quadratic} and \eqref{eqn:gamma-slice-bound}, there exist constants $K_1$ and $K_2$ such that
\begin{equation*}
\left\| f_t^{-1} \frac{df_t}{dt} \right\| \leq K_1 \sqrt{\sigma(h_t)} + K_2 \| e^{-i \beta t} \cdot y_0 - x \|^2 .
\end{equation*}
Now the inequality \eqref{eqn:uniform-sigma-bound} implies that for all $0 \leq t \leq T$
\begin{equation}\label{eqn:group-distance-identity}
d(\id, f_t) \leq \int_0^T \left\| f_t^{-1} \frac{df_t}{dt} \right\| \, dt \leq C_1 \| e^{-i \beta T} \cdot y_0 - x \| + C_2 \| e^{-i \beta T} \cdot y_0 - x \|^2  .
\end{equation}
\end{proof}

Now fix $y_0 \in S_x^-$ such that $\| y_0 - x \| < \varepsilon$ (where $\varepsilon$ is defined in the previous lemma) and define $y_t = e^{i \beta t} \cdot y_0$ for $t \geq 0$. Let $g_s(y_t) \in G_{\bf v}$ be the curve in $G_{\bf v}$ generating the time $s$ modified flow \eqref{eqn:group-modified-flow} with initial condition $y_t$, i.e. $\phi^{mod}(y_t, s) = g_s(y_t) \cdot y_t$. Define $f_s(y_t) = g_s(y_t) \cdot e^{i \beta s} \in G_{\bf v}$ and let $h_s(y_t) = f_s(y_t)^* f_s(y_t)$ be the associated change of metric. This is summarised in the following diagram.

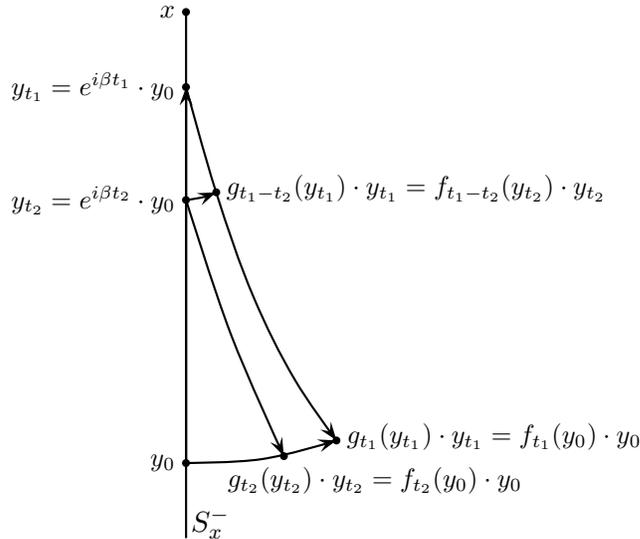
\begin{center}
\begin{pspicture}(2,-1.5)(10,6.5)
\psline(4,6)(4,-1)
\psline[arrowsize=5pt]{->}(4,0)(4,5)
\pscurve[arrowsize=5pt]{->}(4,5)(4.4,3.6)(4.9,2.2)(5.45,1.1)(6,0.3)
\pscurve[arrowsize=5pt]{->}(4,3.5)(4.2,3.53)(4.4,3.6)
\pscurve[arrowsize=5pt]{->}(4,0)(5,0.05)(6,0.3)
\pscurve[arrowsize=5pt]{->}(4,3.5)(4.6,1.7)(5.3,0.09)
\psdots[dotsize=3pt](4,6)(4,5)(4,0)(6,0.3)(4,3.5)(4.4,3.6)(5.3,0.09)
\uput{4pt}[180](4,6){\small $x$}
\uput{4pt}[180](4,5){\small $y_{t_1} = e^{i \beta t_1} \cdot y_0$}
\uput{4pt}[180](4,0){\small $y_0$}
\uput{4pt}[180](4,3.5){\small $y_{t_2}=e^{i \beta t_2} \cdot y_0$}
\uput{4pt}[10](6,0.3){\small $g_{t_1}(y_{t_1}) \cdot y_{t_1} = f_{t_1}(y_0) \cdot y_0$}
\uput{4pt}[300](5.3,0.09){\small{$g_{t_2}(y_{t_2}) \cdot y_{t_2} = f_{t_2}(y_0) \cdot y_0$}}
\uput{4pt}[10](4.4,3.6){\small{$g_{t_1-t_2}(y_{t_1}) \cdot y_{t_1} = f_{t_1-t_2}(y_{t_2}) \cdot y_{t_2}$}}
\uput{2pt}[0](4,-0.8){$S_x^-$}
\end{pspicture}
\end{center}

Using this notation, the estimate \eqref{eqn:uniform-sigma-bound} now becomes
\begin{equation}\label{eqn:new-uniform-sigma-bound}
\sigma(h_s(y_t)) \leq K \| y_t - x \|^2 \leq K e^{2 \lambda_k t} \| y_0 - x \|^2 . 
\end{equation}
and \eqref{eqn:group-distance} becomes
\begin{equation}\label{eqn:new-group-distance-identity}
d(\id, f_s(y_t)) \leq  C_1 \| y_t - x \| + C_2 \| y_t - x \|^2
\end{equation}

Now we can use the estimate from the previous lemma to prove that the sequence $f_t(y_0)$ converges in $G_{\bf v}$. The limit of this sequence defines the map $\psi_x : S_x^- \rightarrow W_x^-$ used in Section \ref{subsec:slice-homeo-unstable}.

\begin{proposition}\label{prop:convergence-group}
The sequence $f_t(y_0)$ is Cauchy in $G_{\bf v}$. Let $f_\infty(y_0)$ denote the limit. Then $f_t(y_0) \cdot y_0$ converges to $f_\infty(y_0) \cdot y_0$ in $\Rep(Q, {\bf v})$.
\end{proposition}

\begin{proof}
Consider the gradient flow with initial conditions $y_{t_2} = e^{i \beta t_2} \cdot y_0$ and $f_{t_1-t_2}(y_{t_2}) \cdot y_{t_2} = \delta_0 \cdot y_{t_2}$, where $\delta_0 = f_{t_1-t_2}(y_{t_2}) \in G_{\bf v}$. The estimate \eqref{eqn:new-group-distance-identity} shows that
\begin{equation}\label{eqn:delta-0-bound}
d(\id, \delta_0) \leq C_1 \| y_{t_2} - x \| + C_2 \| y_{t_2} - x \|^2 \leq C_1 e^{\lambda_k t_2} \| y_0 - x \| + C_2 e^{2 \lambda_k t_2} \| y_0 - x \|^2 .
\end{equation}
Denote the respective solutions by $Y_t^{(1)} = g_t^{(1)} \cdot y_{t_2}$ and $Y_t^{(2)} = g_t^{(2)} \cdot f_{t_1-t_2}(y_{t_2}) \cdot y_{t_2} = \delta_t \cdot Y_t^{(1)}$, where $\delta_t = g_t^{(2)} \cdot \delta_0 \cdot (g_t^{(1)})^{-1}$ as shown in the diagram below.

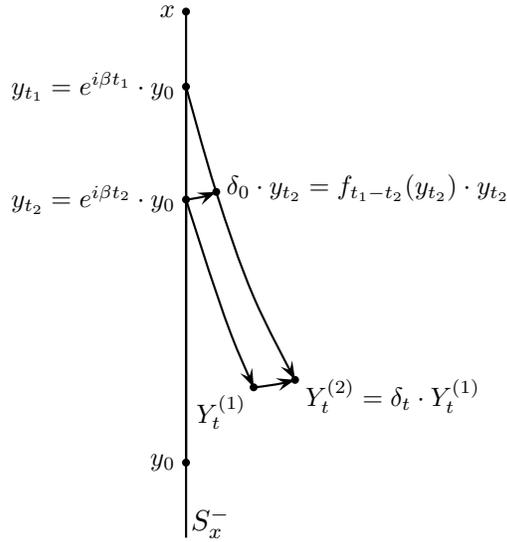
\begin{center}
\begin{pspicture}(2,-1.5)(10,6.5)
\psline(4,6)(4,-1)
\psline[arrowsize=5pt](4,0)(4,5)
\pscurve[arrowsize=5pt]{->}(4,5)(4.4,3.6)(4.9,2.2)(5.45,1.1)
\pscurve[arrowsize=5pt]{->}(4,3.5)(4.6,1.7)(4.9,1)
\pscurve[arrowsize=5pt]{->}(4,3.5)(4.2,3.53)(4.4,3.6)
\pscurve[arrowsize=5pt]{->}(4.9,1)(5.2,1.04)(5.45,1.1)
\psdots[dotsize=3pt](4,6)(4,5)(4,0)(5.45,1.1)(4,3.5)(4.4,3.6)(4.9,1)
\uput{4pt}[180](4,6){\small $x$}
\uput{4pt}[180](4,5){\small $y_{t_1} = e^{i \beta t_1} \cdot y_0$}
\uput{4pt}[180](4,0){\small $y_0$}
\uput{4pt}[180](4,3.5){\small $y_{t_2}=e^{i \beta t_2} \cdot y_0$}
\uput{4pt}[225](4.9,1){\small $Y_t^{(1)}$}
\uput{4pt}[330](5.45,1.1){\small $Y_t^{(2)} = \delta_t \cdot Y_t^{(1)}$}
\uput{4pt}[10](4.4,3.6){\small{$\delta_0 \cdot y_{t_2} = f_{t_1-t_2}(y_{t_2}) \cdot y_{t_2}$}}
\uput{2pt}[0](4,-0.8){$S_x^-$}
\end{pspicture}
\end{center}

Let $h_t = \delta_t^* \delta_t$. Then 
\begin{align*}
 \delta_t^{-1} \frac{d (\delta_t)}{dt} & = \delta_t^{-1} \left( \frac{dg_t^{(2)}}{dt} (g_t^{(2)})^{-1} \right) \delta_t - \frac{dg_t^{(1)}}{dt} (g_t^{(1)})^{-1} \\
 & = \delta_t^{-1} \left( -i \left( \mu(g_t^{(2)} \cdot f_{t_1-t_2}(y_{t_2}) \cdot y_{t_2}) - \alpha \right) + \gamma(g_t^{(2)} \cdot f_{t_1-t_2}(y_{t_2}) \cdot y_{t_2}) \right) \delta_t \\
 & \quad \quad \quad + i (\mu(g_t^{(1)} \cdot y_{t_2}) - \alpha) - \gamma(g_t^{(1)} \cdot y_{t_2}) \\
 & = -i \left( \mu_{h_t}(g_t^{(1)}\cdot y_{t_2}) - \mu(g_t^{(1)} \cdot y_{t_2}) \right) + \left( \delta_t^{-1} \gamma( \delta_t \cdot g_t^{(1)} \cdot y_{t_2} ) \delta_t - \gamma(g_t^{(1)} \cdot y_{t_2}) \right)
\end{align*}
where in the last step we use the fact that $\delta_t \alpha \delta_t^{-1} = \alpha$ since $\alpha$ is central. Applying the Lipschitz bounds \eqref{eqn:mu-h-Lipschitz} and \eqref{eqn:gamma-lipschitz-bound} and then the distance decreasing formula \eqref{eqn:distance-decreasing} gives us
\begin{equation*}
\left\| \delta_t^{-1} \frac{d (\delta_t)}{dt} \right\| \leq C \sqrt{\sigma(h_t)} \leq C \sqrt{\sigma(h_0)} .
\end{equation*}
The estimate from \eqref{eqn:new-uniform-sigma-bound} shows that $\sigma(h_0) \leq K  \| e^{i \beta t_2} \cdot y_0 - x \|^2$. Therefore, for any $T$ such that $t_1 > t_2 \geq T > - \frac{1}{\lambda_k}$, the geodesic distance $d(\delta_0, \delta_{t_2})$ in $G_{\bf v}$ is bounded above by 
\begin{multline*}
\int_0^{t_2} \left\| \delta_t^{-1} \frac{d (\delta_t)}{dt} \right\|  \, dt \leq \int_0^{t_2} C \sqrt{K} \| e^{i \beta t_2} \cdot y_0 - x \|  \, dt = t_2 C' \| e^{i \beta t_2} \cdot y_0 - x \| \\
\leq t_2 C' e^{\lambda_k t_2} K \| y_0 - x \|  \leq T e^{\lambda_k T} C_3 \| y_0 - x \|
\end{multline*}
Combining this with the estimate \eqref{eqn:delta-0-bound} shows that 
\begin{multline}\label{eqn:cauchy-bound}
d(f_{t_1}(y_0), f_{t_2}(y_0)) = d(\id, \delta_{t_2}) \leq d(\id, \delta_0) + d(\delta_0, \delta_{t_2}) \\
 \leq C_1 e^{\lambda_k T} \| y_0 - x \| + C_2 e^{2\lambda_k T} \| y_0 - x \|^2 + C_3 T e^{\lambda_k T} \| y_0 - x \| 
\end{multline}
which can can be made arbitrarily small by making $T$ large. Therefore the sequence $f_t(y_0)$ is Cauchy in $G_{\bf v}$. Let $f_\infty(y_0)$ denote the limit. Since the action of $G_{\bf v}$ on $\Rep(Q, {\bf v})$ is continuous, then $f_t(y_0) \cdot y_0$ converges to $f_\infty(y_0) \cdot y_0$ in $\Rep(Q, {\bf v})$.
\end{proof}

As a consequence of the methods used above, we can derive the following estimate which is used in Theorem \ref{thm:homeo-slice-bundle}.

\begin{corollary}\label{cor:convergence-in-rep}
Let $C$ be a critical set. Then there exists a constant $K$ such that for any $x \in C$ and any $y_0 \in S_x^-$, we have
\begin{equation}\label{eqn:convergence-in-rep}
\| f_t(y_0) \cdot y_0 - f_\infty(y_0) \cdot y_0 \| \leq \left( K_1 e^{\lambda_k t} (t+1) \| y_0 - x \| + K_2 e^{2\lambda_k t} \| y_0 - x \|^2 \right) \| y_0 \|
\end{equation}
\end{corollary}

\begin{proof}
Since $\{f_t(y_0) \}$ is Cauchy, then the inequality \eqref{eqn:cauchy-bound} shows that 
\begin{equation*}
d(f_t(y_0), f_\infty(y_0)) \leq C_1' e^{\lambda_k t} (t+1) \| y_0 - x \| + C_2 e^{2\lambda_k t} \| y_0 - x \|^2 .
\end{equation*}
Since the action of $G_{\bf v}$ is continuous then there exists a constant $K$ such that $\| f_t(y_0) \cdot y_0 - f_\infty(y_0) \cdot y_0 \| \leq K d(f_t(y_0), f_\infty(y_0)) \| y_0 \|$. The result then follows after combining these two estimates.
\end{proof}

Next we show that $f_\infty(y_0) \cdot y_0$ lies in $W_{mod, x}^-$.

\begin{lemma}
Let $\varepsilon > 0$ be as in Lemma \ref{lem:uniform-bound-metric}. For each $y_0 \in S_x^-$ such that $\| y_0 - x \| < \varepsilon$, we have $f_\infty(y_0) \cdot y_0 \in W_{mod,x}^-$.
\end{lemma}

\begin{proof}
Since $f_t(y_0) \cdot y_0 \rightarrow f_\infty(y_0) \cdot y_0$ and the finite-time flow $\phi^{mod}(\cdot, t) : \Rep(Q, {\bf v}) \rightarrow \Rep(Q, {\bf v})$ depends continuously on the initial condition, then for each $T > 0$, we have $\phi^{mod}(f_t(y_0) \cdot y_0, -T) \rightarrow \phi^{mod}(f_\infty(y_0) \cdot y_0, -T)$. Therefore we can choose $t$ large so that $t \geq T$ and
\begin{equation}\label{eqn:ineq1}
\| \phi^{mod}(f_\infty(y_0) \cdot y_0, -T) - \phi^{mod}(f_t(y_0) \cdot y_0, -T) \| \leq \| e^{i \beta T} \cdot y_0 - x \| .
\end{equation}
Moreover, since $f_t(y_0) \cdot y_0 = \phi^{mod}(e^{i \beta t} \cdot y_0, t)$ by definition, for all $t \geq T$ we have 
\begin{equation*}
\phi^{mod}(f_t(y_0) \cdot y_0, -T) = \phi^{mod}(e^{i \beta t} \cdot y_0, t-T) = \phi^{mod}(e^{i \beta (t-T)} \cdot e^{i \beta T} \cdot y_0, t-T) = f_{t-T}(e^{i \beta T} \cdot y_0) \cdot (e^{i \beta T} \cdot y_0) .
\end{equation*}

In particular, the bound of \eqref{eqn:group-distance} applies, and so we have 
\begin{align}\label{eqn:ineq2}
\begin{split}
\| \phi^{mod}(f_t(y_0) \cdot y_0, -T) - e^{i \beta T} \cdot y_0 \| & = \| f_{t-T}(e^{i \beta T} \cdot y_0) \cdot (e^{i \beta T} \cdot y_0) - e^{i \beta T} \cdot y_0 \| \\
 & \leq C d \left( \id, f_{t-T}(e^{i \beta T} \cdot y_0) \right) \| e^{i \beta T} \cdot y_0 \| \\
 & \leq C \left( \| e^{-i \beta T} \cdot y_0 - x \| + \| e^{-i \beta T} \cdot y_0 - x \|^2 \right) \| e^{i \beta T} \cdot y_0 \|
\end{split}
\end{align}
for some positive constant $C$. Combining \eqref{eqn:ineq1} \eqref{eqn:ineq2} gives us 
\begin{align*}
\| \phi^{mod}(f_\infty(y_0) \cdot y_0, -T) - x \| & \leq \| \phi^{mod}(f_\infty(y_0) \cdot y_0, -T) - \phi^{mod}(f_t(y_0) \cdot y_0, -T) \| \\
 & \quad \quad + \| \phi^{mod}(f_t(y_0) \cdot y_0, -T) - e^{i \beta T} \cdot y_0 \| + \| e^{i \beta T} \cdot y_0 - x \| \\
 & \leq \| e^{i \beta T} \cdot y_0 - x \| + \| e^{i \beta T} \cdot y_0 - x \| \\
 & \quad \quad + C \left( \| e^{-i \beta T} \cdot y_0 - x \| + \| e^{-i \beta T} \cdot y_0 - x \|^2 \right) \| e^{i \beta T} \cdot y_0 \|
\end{align*}
Since the right-hand side converges to zero as $T \rightarrow \infty$ then $\lim_{T \rightarrow \infty} \phi(f_\infty(y_0) \cdot y_0, -T) = x$ and so $f_\infty(y_0) \cdot y_0 \in W_{mod,x}^-$.
\end{proof}

Let $\psi_x : S_x^- \rightarrow W_{mod,x}^-$ be the map $\psi_x(y_0) = f_\infty(y_0) \cdot y_0$. The final result of this section shows that $\psi_x$ is continuous.

\begin{lemma}\label{lem:fibrewise-uniformly-convergent}
There exists a neighbourhood $U$ of $x$ in $S_x^-$ such that $f_t(y_0) \cdot y_0$ converges uniformly to $f_\infty(y_0) \cdot y_0$ for all $y_0 \in U$. Therefore $f_\infty(y_0) \cdot y_0$ depends continuously on $y_0$.
\end{lemma}

\begin{proof}
To simplify the notation, given $y_0 \in S_x^-$ set $f_t = f_t(y_0)$ for each $t \in \mathbb{R}_{\geq 0}$. The inequality \eqref{eqn:cauchy-bound} shows that in a neighbourhood $U$ of $x$ in $S_x^-$ we have for all $t_1 > t_2 \geq T$
\begin{equation*}
d(f_{t_1}, f_{t_2}) \leq C_1 e^{\lambda_k T} \| y_0 - x \| + C_2  e^{2 \lambda_k T} \| y_0 - x \|^2 + C_3 T e^{\lambda_k T} \| y_0 - x \| 
\end{equation*}
and so there exists a constant $K$ such that
\begin{equation}\label{eqn:series-term-bound}
\| f_{t_1} \cdot y_0 - f_{t_2} \cdot y_0 \| \leq K e^{\lambda_k T} \left( C_1 \| y_0 - x \| + C_2 \| y_0 - x \|^2 + C_3 T \| y_0 - x \| \right) \| y_0 \|
\end{equation}
Therefore if $y_0$ is in a fixed neighbourhood of $x$ then $\| f_{t_1} \cdot y_0 - f_{t_2} \cdot y_0 \| \leq K_1' e^{\lambda_k T} + K_2' T e^{\lambda_k T}$ and so the sequence $f_t \cdot y_0$ is uniformly Cauchy with respect to the variable $y_0$, and therefore uniformly convergent. Since the finite-time terms $f_t \cdot y_0$ depend continuously on $y_0$, then this implies that $\psi_x(y_0) = f_\infty \cdot y_0$ also depends continuously on $y_0$. 
\end{proof}

\subsubsection{The reverse of the scattering method}\label{subsec:reverse-scattering}

In this section we carry out the reverse of the procedure described in the previous section. Given $y_0 \in W_{mod,x}^-$, let $y_t = g_t \cdot y_0$ denote the solution to the time $t$ modified flow \eqref{eqn:modified-flow-def} and define $f_t = g_t \cdot e^{i \beta t}$. The goal is to show that $f_t$ converges to the inverse of the map defined in the previous section. The key to the proof is that $f_t \in P_\beta$ for all $t$, and so the distance decreasing formula of Lemma \ref{lem:modified-distance-decreasing} applies (this is the reason for using the modified flow instead of the original gradient flow). 

\begin{figure}[ht]
\begin{center}
\begin{pspicture}(2,0)(10,5)
\pscurve(3.9,5)(4,4)(4.4,2.3)(5,1.2)(5.5,0.5)(6,0)
\psline[arrowsize=5pt]{->}(4,4)(4,1)
\pscurve[arrowsize=5pt]{->}(4,4)(4.4,2.3)(5,1.2)
\pscurve[arrowsize=5pt]{->}(4,1)(4.5,1.05)(5,1.2)
\psdots[dotsize=3pt](3.9,5)(4,4)(4,1)(5,1.2)
\uput{4pt}[180](3.9,5){\small{$x$}}
\uput{4pt}[180](4,4){\small{$y_0$}}
\uput{4pt}[180](4,1){\small{$e^{- i \beta t} \cdot y_0$}}
\uput{4pt}[0](5,1.2){\small{$g_t \cdot y_0$}}
\uput{2pt}[30](6,0){$W_{mod,x}^-$}
\uput{3pt}[270](4.5,1.05){\small{$f_t$}}
\uput{3pt}[30](4.4,2.3){\small{$g_t$}}
\uput{3pt}[180](4,2.3){\small{$e^{- i \beta t}$}}
\end{pspicture}
\end{center}
\end{figure}

First we derive a uniform bound on the metric $h_t$ in analogy with Lemma \ref{lem:uniform-bound-metric}. The proof differs slightly from Lemma \ref{lem:uniform-bound-metric}, since the point $y_0$ is not necessarily in the negative slice and so we need to derive some extra estimates after finding the bound \eqref{eqn:reverse-gronwall-bound}.
\begin{lemma}\label{lem:reverse-uniform-metric-bound}
There exists $\varepsilon > 0$ and a constant $K$ such that for any initial condition $y_0 \in W_{mod,x}^-$ and any $t > 0$ satisfying $\| e^{-i \beta t} \cdot y_0 - x \| + \| g_t \cdot y_0 - x \| < \varepsilon$ we have
\begin{equation*}
\sigma(h_t) \leq K \left( \| e^{-i \beta t} \cdot y_0 - x \| + \| g_t \cdot y_0 - x \| \right) .
\end{equation*}
\end{lemma}

\begin{proof}
The same calcuation as \eqref{eqn:gauge-difference-derivative} and \eqref{eqn:metric-derivative} shows that
\begin{equation}\label{eqn:reverse-gauge-difference-derivative}
\frac{df_t}{dt} f_t^{-1} = -i (\mu(g_t \cdot y_0) - \alpha) + \gamma(g_t \cdot y_0) +  f_t(i\beta) f_t^{-1}
\end{equation}
and
\begin{equation}\label{eqn:reverse-metric-derivative}
\frac{dh_t}{dt} = -2i h_t (\mu_h(e^{-i \beta t} \cdot y_0) - \alpha) + i \beta h_t + i h_t \beta .
\end{equation}

Taking the trace in the same way as the proof of Lemma \ref{lem:uniform-bound-metric} gives us the estimates
\begin{equation*}
\frac{d}{dt} \tr (h_t) \leq  2i \tr \left( (\mu(e^{-i \beta t} \cdot y_0) - \alpha - \beta) h_t \right) \leq C \tr(h_t) \| e^{-i \beta t} \cdot y_0 - x \|
\end{equation*}
and
\begin{equation*}
\frac{d}{dt} \tr (h_t^{-1}) \leq  -2i \tr \left( h_t^{-1} (\mu(e^{-i \beta t} \cdot y_0) - \alpha - \beta) \right) \leq C \tr(h_t^{-1}) \| e^{-i \beta t} \cdot y_0 - x \|
\end{equation*}
and so 
\begin{equation*}
\frac{d}{dt} \sigma(h_t) \leq C \sigma(h_t) \| e^{-i \beta t} \cdot y_0 - x \| + 2C \rank(Q, {\bf v}) \| e^{-i \beta t} \cdot y_0 - x \|
\end{equation*}
Therefore Gronwall's inequality implies that
\begin{equation}\label{eqn:reverse-gronwall-bound}
\sigma(h_t) \leq 2C \rank(Q, {\bf v}) \left( \int_0^t \| e^{-i \beta s} \cdot y_0 - x \| \, ds \right) \exp \left( \int_0^t C \| e^{-i \beta s} \cdot y_0 - x \| \, ds \right) 
\end{equation}
and so the problem of finding a bound for $\sigma(h_t)$ reduces to the problem of finding a bound for $\int_0^t \| e^{-i \beta s} \cdot y_0 - x \| \, ds$ in terms of $\| g_t \cdot y_0 - x \|$ and $\| e^{-i \beta t} \cdot y_0 - x \|$.

To find this bound we decompose $y_0 - x$ into eigenspaces $( y_0 - x )_+$, $( y_0 - x )_0$ and $( y_0 - x )_-$ according to whether the action of $e^{-i \beta}$ is has eigenvalues greater than one, equal to one, or less than one. Since $\| (e^{-i \beta s} \cdot y_0 - x)_+ \|$ is exponentially increasing with $s$ then there exists a constant $C_1$ such that  
\begin{equation*}
\int_0^t \|  (e^{-i \beta s} \cdot y_0 - x)_+ \| \, ds \leq C_1 \| (e^{-i \beta t} \cdot y_0 - x)_+ \| \leq C_1 \| e^{-i \beta t} \cdot y_0 - x \|.
\end{equation*}
Since $\| (e^{-i \beta s} \cdot y_0 - x)_0 \|$ is constant then  
\begin{equation*}
\int_0^t \|  (e^{-i \beta s} \cdot y_0 - x)_0 \| \, ds = t \| (y_0 - x)_0 \| \leq C t e^{-C_2' t} \| g_t \cdot y_0 - x \|.
\end{equation*}
where we use the bound $\| y_0 - x \| \leq C_1' e^{-C_2' t} \| g_t \cdot y_0 - x \|$ from Proposition \ref{prop:exponential-modified-convergence} (note that here $t > 0$ and so the exponent is $-C_2't$). The term $\| (e^{-i \beta s} \cdot y_0 - x)_- \|$ is exponentially decreasing with $s$, and so there is a constant $C_2$ such that
\begin{equation*}
\int_0^t \|  (e^{-i \beta s} \cdot y_0 - x)_- \| \, ds \leq C_2 \| (y_0 - x)_- \| \leq C C_2 e^{-C_2' t} \| g_t \cdot y_0 - x \|.
\end{equation*}
Substituting these into \eqref{eqn:reverse-gronwall-bound} and noting that the exponential terms are bounded if $\| e^{-i \beta t} \cdot y_0 - x \| + \| g_t \cdot y_0 - x \| < \varepsilon$ for small enough $\varepsilon$ shows that 
\begin{equation}\label{eqn:reverse-sigma-bound}
\sigma(h_t) \leq K \left( \| e^{-i \beta s} \cdot y_0 - x \| + \| g_t \cdot y_0 - x \| \right)
\end{equation}
for some positive constant $K$.
\end{proof}

Now we prove a uniform bound on $f_t$.

\begin{lemma}\label{lem:reverse-uniform-group-bound}
With the same conditions as Lemma \ref{lem:reverse-uniform-metric-bound} we have
\begin{equation}\label{eqn:reverse-uniform-group-bound}
d(\id, f_t) \leq C_1 \| e^{-i \beta t} \cdot y_0 - x \|^{\frac{1}{2}} + C_2 \| e^{- i \beta t} \cdot y_0 - x \|  .
\end{equation}
\end{lemma}

\begin{proof}
After conjugating \eqref{eqn:reverse-gauge-difference-derivative} by $f_t$, we obtain the same bound as in the proof of Lemma \ref{lem:group-distance}
\begin{align*}
f_t^{-1} \frac{df_t}{dt} & =  -i f_t^{-1} (\mu(g_t \cdot y_0) - \alpha) f_t + f_t \gamma(g_t \cdot y_0) f_t^{-1} + i \beta \\
 & = -i (\mu_{h_t}(e^{-i \beta t} \cdot y_0) - \alpha) + f_t \gamma(g_t \cdot y_0) f_t^{-1} + i \beta \\
 & = -i \left( \mu_{h_t}(e^{-i \beta t} \cdot y_0) - \mu(e^{-i \beta t} \cdot y_0) \right) -i \left( \mu(e^{-i \beta t} \cdot y_0) - \alpha - \beta \right) \\
 & \quad \quad \quad + \left( f_t \gamma(g_t \cdot y_0) f_t^{-1} - \gamma(e^{-i \beta t} \cdot y_0) \right) + \gamma(e^{-i \beta t} \cdot y_0)
\end{align*}
and therefore the Lipschitz bounds \eqref{eqn:gamma-lipschitz-bound}, \eqref{eqn:gamma-critical-lipschitz} and the moment map bound \eqref{eqn:mu-h-Lipschitz} lead to the inequality
\begin{equation*}
\left\| f_t^{-1} \frac{df_t}{dt} \right\| \leq K_1 \sqrt{\sigma(h_t)} + K_2 \| e^{-i \beta t} \cdot y_0 - x \|
\end{equation*}
Therefore Lemma \ref{lem:reverse-uniform-metric-bound} implies
\begin{equation*}
d(\id, f_t) \leq \int_0^t \left\| f_t^{-1} \frac{df}{dt} \right\| \, dt \leq C_1 \| e^{-i \beta t} \cdot y_0 - x \|^{\frac{1}{2}} + C_2 \| e^{- i \beta t} \cdot y_0 - x \|  .
\end{equation*}
\end{proof}

Now we use the uniform bound on the metric together with the distance-decreasing formula to show that the sequence is Cauchy. Fix $y_0 \in W_{mod,x}^-$ and for all $t < 0$ let $y_t = g_t(y_0) \cdot y_0$ be the solution to \eqref{eqn:modified-flow-def} and define $f_t(y_0) = e^{i \beta t} \cdot g_t(y_0)$. This is summarised in the diagram below.

\begin{figure}[ht]
\begin{center}
\begin{pspicture}(2,-1)(10,6)
%\psline[linestyle=dashed](3.5,6)(3.5,-1)
\psline[arrowsize=5pt]{->}(5,3.6)(5,0.2)
\psline[arrowsize=5pt]{->}(4,5.1)(4,0)
\pscurve(3.5,6)(4,5.1)(5,3.6)(6,2.2)(7.4,0.5)
\pscurve[arrowsize=5pt]{->}(5,3.6)(4.4,3.53)(4,3.4)
\pscurve[arrowsize=5pt]{->}(7.4,0.5)(5,0.2)(4,0)
%\pscurve[arrowsize=5pt]{->}(4,3.5)(4.6,1.7)(5.3,0.09)
\psdots[dotsize=3pt](3.5,6)(4,5.1)(4,3.4)(4,0)(7.4,0.5)(5,3.6)(5,0.2)
\uput{4pt}[180](3.5,6){\small $x$}
\uput{4pt}[20](4,5){\small $y_{t_1} = g_{t_1}(y_0) \cdot y_0$}
\uput{4pt}[180](4,-0.2){\small $f_{t_1}(y_0) \cdot y_0 = e^{i \beta t_1} \cdot y_{t_1}$}
\uput{4pt}[180](4,3.4){\small $e^{i \beta (t_1-t_2)} \cdot y_{t_1}$}
\uput{4pt}[10](7.4,0.5){\small $y_0$}
\uput{5pt}[305](5,0.2){\small{$f_{t_2}(y_0) \cdot y_0 = e^{i\beta t_2} \cdot y_{t_2}$}}
\uput{4pt}[20](5,3.6){\small{$y_{t_2} = g_{t_2}(y_0) \cdot y_0$}}
%\uput{2pt}[180](3.5,-0.8){$S_x^-$}
\uput{2pt}[45](6.8,1.5){$W_{mod,x}^-$}
\end{pspicture}
\end{center}
\end{figure}

\begin{lemma}
There exists $\varepsilon' > 0$ such that the limit $f_\infty(y) := \lim_{t \rightarrow - \infty} f_t(y)$ exists in $P_\beta$ for any $y \in W_{mod,x}^-$ such that $\| y - x \| < \varepsilon'$.
\end{lemma}

\begin{proof}
First we show that for any $\varepsilon > 0$ there exists $\varepsilon' > 0$ such that $\| y_0 - x \| < \varepsilon'$ implies that $\| f_t(y_0) \cdot y_0 - x \| + \| y_0 - x \| < \varepsilon$, and therefore we can apply the estimate of Lemma \ref{lem:reverse-uniform-group-bound}.

First suppose that $\| y_0 - x \| = \frac{1}{2} \varepsilon$. If $\| f_t(y_0) \cdot y_0 - x \| < \frac{1}{2} \varepsilon$ for all such $y_0$ then we are done. If not, then there exists $t' > 0$ such that $\| e^{i \beta t'} \cdot f_t(y_0) \cdot y_0 - x \| = \frac{1}{2} \varepsilon$. Let $g_{-t'} \cdot y_0$ denote the time $t'$ upwards flow with initial condition $y_0$. Note that $\| g_{-t'} \cdot y_0 - x \| < \frac{1}{2} \varepsilon$. Then $e^{i \beta t'} \cdot f_t(y_0) \cdot y_0 = f_{t-t'}(g_{-t'} \cdot y_0) \cdot g_{-t'} \cdot y_0$ by definition. Since $\| g_{-t'} \cdot y_0 - x \| + \| f_{t-t'}(g_{-t'} \cdot y_0) \cdot g_{-t'} \cdot y_0 - x \| < \varepsilon$ then the bound of Lemma \ref{lem:reverse-uniform-group-bound} applies. 

Therefore $f_{t-t'}(g_{-t'} \cdot y_0)$ is in a bounded neighbourhood of the identity in $G_{\bf v}$. We also know that $x \in \overline{G_{\bf v} \cdot y_0} \setminus (G_{\bf v} \cdot y_0)$ and that $\| f_{t-t'}(g_{-t'} \cdot y_0) \cdot g_{-t'} \cdot y_0 - x \| = \frac{1}{2} \varepsilon$ and so there exists $\delta > 0$ such that $\| g_{-t'} \cdot y_0 - x \| > \delta$. Since $f_{t-t'}(g_{-t'} \cdot y_0)$ is uniformly bounded by Lemma \ref{lem:reverse-uniform-group-bound} then $\delta$ is uniform with respect to the choice of $y_0 \in W_{mod,x}^-$ such that $\| y_0 - x \| < \frac{1}{2} \varepsilon$. Therefore there exists $T$ such that $t' \leq T$ for any such $y_0$. 

Therefore given any $\varepsilon' > 0$ such that $\| g_{-T} \cdot y_0 - x \| < \varepsilon'$ for all $y_0 \in W_{mod,x}^-$ such that $\| y_0 - x \| < \frac{1}{2} \varepsilon$, for any $y \in W_{mod,x}^-$ such that $\| y - x \| < \varepsilon'$ we have $\| f_t(y) \cdot y - x \| + \| y - x \| < \varepsilon$ and now we can use the estimate of Lemma \ref{lem:reverse-uniform-group-bound} for any initial condition $y$ satisfying $\| y - x \| < \varepsilon'$.

Choose $y_0 \in W_{mod,x}^-$ such that $\| y_0 - x \| < \varepsilon'$. Given $t_1 < t_2 < 0$, let $\delta_0 = e^{i \beta (t_1 - t_2)} \cdot g_{t_1} \cdot g_{t_2}^{-1}$ and $\delta_{t_2} = f_{t_1}(y_0) f_{t_2}(y_0)^{-1} = e^{-i \beta t_2} \delta_0 e^{i \beta t_2}$. Lemma \ref{lem:reverse-uniform-group-bound} shows that there exists $T < 0$ such that $t_1, t_2 < T$ implies that $d(\id, \delta_0) < \varepsilon$. The distance-decreasing formula of Lemma \ref{lem:modified-distance-decreasing} then shows that 
\begin{equation*}
d(f_{t_1}(y_0), f_{t_2}(y_0)) = d(\id, \delta_{t_2}) = d(\id, e^{-i \beta t_2} \delta_0 e^{i \beta t_2}) \leq d(\id, \delta_0) < \varepsilon .
\end{equation*}
Therefore the sequence $f_t(y_0)$ is Cauchy and so it has a unique limit in $P_\beta$.
\end{proof}

For each $z_0 \in W_{mod,x}^-$ define $\varphi_x(y_0) = f_\infty(y_0) \cdot y_0$. The same proof as Lemma \ref{lem:fibrewise-uniformly-convergent} shows that the sequence $f_t(y_0)$ is uniformly Cauchy for $y_0$ in a bounded neighbourhood of $x$ in $W_{mod,x}^-$ and so $\varphi_x$ is continuous. The next result shows that $\varphi_x$ is a left inverse of the map $\psi_x : S_x^- \rightarrow W_{mod,x}^-$ defined in the previous section.

\begin{lemma}\label{lem:left-inverse}
$\varphi_x \circ \psi_x(y_0) = y_0$ for all $y_0 \in S_x^-$.
\end{lemma}

\begin{proof}
Let $z_0 = \psi_x(y_0)$. Then $z_0 = f_\infty(y_0) \cdot y_0$ and for all $\varepsilon > 0$ there exists $T$ such that $d(\id, f_\infty(e^{i \beta t} \cdot y_0)) < \varepsilon$ for all $t \geq T$. Since $f_\infty(e^{i \beta t} \cdot y_0) \in P_\beta$ then $d(\id, e^{-i \beta t} \cdot f_\infty(e^{i \beta t} \cdot y_0) \cdot e^{i \beta t}) \leq d(\id, f_\infty(e^{i \beta t} \cdot y_0))$ by the distance-decreasing property of Lemma \ref{lem:modified-distance-decreasing}. Therefore for all $t \geq T$ we have
\begin{equation*}
\| e^{-i \beta t} \cdot f_\infty(e^{i \beta t} \cdot y_0) \cdot e^{i \beta t} \cdot y_0 - y_0 \| \leq C \varepsilon
\end{equation*}
for a constant $C$ which is uniformly bounded when $y_0$ is in a fixed neighbourhood of $x$. This is true for all $\varepsilon > 0$ and so 
\begin{equation}\label{eqn:limit-inverse}
\lim_{t \rightarrow \infty} e^{-i \beta t} \cdot f_\infty(e^{i \beta t} \cdot y_0) \cdot e^{i \beta t} \cdot y_0 = y_0 .
\end{equation}

Since $\psi_x(y_0)$ is the time $t$ downwards flow of $f_\infty(e^{i \beta t} \cdot y_0) \cdot e^{i \beta t} \cdot y_0$ then 
\begin{equation*}
\varphi_x(\psi_x(y_0)) = \lim_{t \rightarrow \infty} e^{-i \beta t} \phi^{mod}(\psi_x(y_0), -t) = \lim_{t \rightarrow \infty} e^{-i \beta t} f_\infty(e^{i \beta t} \cdot y_0) \cdot e^{i \beta t} \cdot y_0 = y_0 .
\end{equation*}
and so $\varphi_x$ is a left inverse to $\psi_x$.
\end{proof}

The standard graph transform approach to constructing unstable manifolds (see for example \cite{HirschPughShub77}) determines a local homeomorphism $S_x^- \rightarrow W_{mod,x}^-$ in a neighbourhood of $x$. The next result shows that in fact there is a local homeomorphism determined by the action of $G_{\bf v}$ via the map $\psi_x$. The point of using $\psi_x$ is that (a) the standard methods do not work on a singular space, however the method used here does work since $\psi_x$ remains a homeomorphism after restricting to any singular subset $Z \subset \Rep(Q, {\bf v})$ preserved by $G_{\bf v}$, and (b) this allows us to classify the isomorphism classes in the unstable set in terms of those in the negative slice, which is used in the next section to interpret critical points connected by a flow line in terms of the Hecke correspondence.

\begin{proposition}
There is a neighbourhood $U$ of $x$ in $S_x^-$ and a neighbourhood $V$ of $x$ in $W_{mod,x}^-$ such that $\psi_{x}$ defines a homeomorphism of pairs $(U, U \setminus \{0\}) \cong (V, V \setminus \{x\})$.
\end{proposition}

\begin{proof}
Choose a ball $U$ centred at $0$ in $S_x^-$ of radius small enough such that the conditions of Lemma \ref{lem:fibrewise-uniformly-convergent} are satisfied and so the map $\psi_x$ is defined and continuous, and define $V := \psi_x(U)$. Then $\psi_x : U \rightarrow V$ has a continuous left inverse by Lemma \ref{lem:left-inverse}, and so it is a homeomorphism onto its image $V$. Therefore it only remains to show that $V$ is a neighbourhood of $x$ in $W_{mod,x}^-$.

First note that since $f_\infty(y_0)$ is uniformly bounded by Lemma \ref{lem:group-distance} and $x \in \overline{G_{\bf v} \cdot y_0} \setminus (G_{\bf v} \cdot y_0)$ then there exists $r > 0$ such that if $y_0 \in \partial U$ then $\psi_x(y_0) \notin B_r$, where $B_r$ is the ball of radius $r$ in $W_{mod,x}^-$.

Since $\psi_x(e^{i \beta t} \cdot y_0) = \phi^{mod}(\psi_x(y_0), -t)$ then we can define a continuous map $\partial U \rightarrow \partial B_r$ by using the modified flow to flow the image $\psi_x(\partial U)$ up to $\partial B_r$. This is is injective since $\psi_x$ is injective and $\psi_x(e^{i \beta t} \cdot y_0) = \phi^{mod}(\psi_x(y_0), -t)$. Therefore we have a continuous injective map between two spheres of the same dimension, which must also be surjective, since if there exists a point which is not in the image then we have a continuous map from a sphere into a ball of the same dimension from which the Borsuk-Ulam theorem contradicts injectivity. Therefore the image $V = \psi_x(U)$ contains the ball of radius $r$ in $W_{mod,x}^-$, and so it is a neighbourhood of zero.
\end{proof}

\subsubsection{The negative slice is homeomorphic to the unstable manifold}\label{subsec:slice-homeo-unstable}

Fix a critical set $C$, and for each $x \in C$ define $\psi_x : S_x^- \rightarrow W_{mod,x}^-$ by $\psi_x(y_0) :=  f_{\infty}(y_0) \cdot y_0$. Since $\psi_x$ is surjective and $\varphi_x$ is a left-inverse then $\varphi_x : W_{mod,x}^- \rightarrow S_x^-$. Recall the definition of the negative slice bundle $S_C^-$ and the unstable bundle $W_C^-$ from Definition \ref{def:slice-bundle}. The goal of this section is to show that $\psi_x$ can be extended to a homeomorphism $\psi : S_C^- \rightarrow W_C^-$. We first prove this result on the smooth space $\Rep(Q, {\bf v})$ where $S_C^-$ and $W_C^-$ are bundles over $C$ with smooth fibres, and then note that since this homeomorphism $\psi$ is defined using the action of $G_{\bf v}$, then it remains a homeomorphism on restriction to any closed $G_{\bf v}$-invariant subset of $\Rep(Q, {\bf v})$.

Lemma \ref{lem:fibrewise-uniformly-convergent} shows that $\psi$ and $\varphi$ are continuous on each fibre of $S_C^-$. The next result shows that $\psi$ and $\varphi$ are continuous on $S_C^-$ and hence define a homeomorphism.

\begin{theorem}\label{thm:homeo-slice-bundle}
There is a $K_{\bf v}$-equivariant homeomorphism of a neighbourhood of $C$ in $S_C^-$ with a neighbourhood of $C$ in $W_C^-$.
\end{theorem}

\begin{proof}
Proposition \ref{prop:modified-flow-homeo} shows that it is sufficient to show that $\psi : S_C^- \rightarrow W_{mod,C}^-$ is a homeomorphism in a neighbourhood of $C$. The results of the previous section show that $\psi$ is injective, surjective and the restriction to each fibre of $S_C^-$ is continuous. Therefore it only remains to show that $\psi$ is continuous on all of $S_C^-$ and that $\varphi$ is continuous on all of $W_C^-$. 

Given $\varepsilon > 0$ and $y_0 \in S_x^-$, use Corollary \ref{cor:convergence-in-rep} to choose $T$ such that $\| f_T(y_0) \cdot y_0 - \psi(y_0) \| < \varepsilon$. Since the action of $e^{i \beta T}$ is continuous and the time $T$ flow depends continuously on the initial condition, then there exists $\delta$ such that for all $y \in S_C^-$ such that $\| y - y_0 \| < \delta$ we have $\| f_T(y) \cdot y - f_T(y_0) \cdot y_0 \| < \varepsilon$. Since the constants in Corollary \ref{cor:convergence-in-rep} are uniform over the critical set, then (after shrinking $\delta$ if necessary) we can arrange it so that $\| f_T(y) \cdot y - \psi(y) \| < 2 \varepsilon$ for all $y$ such that $\| y - y_0 \| < \delta$.

In conclusion, given $\varepsilon > 0$ we have constructed $\delta > 0$ such that $\| y - y_0 \| < \delta$ implies that 
\begin{align*}
\| \psi(y) - \psi(y_0) \| & \leq \| \psi(y) - f_T(y) \cdot y \| + \| f_T(y) \cdot y - f_T(y_0) \cdot y_0 \| + \| f_T(y_0) \cdot y_0 - \psi(y_0) \| \\
 & < 2 \varepsilon + \varepsilon + \varepsilon = 4 \varepsilon.
\end{align*}
Therefore $\psi$ is continuous. Since the constants in Lemma \ref{lem:reverse-uniform-group-bound} are uniform over the critical set then the same proof shows that $\varphi$ is also continuous.

This shows that $S_C^-$ is homeomorphic to the unstable set $W_{mod,C}^-$ for the modified flow \eqref{eqn:modified-flow-def}. Together with the result of Proposition \ref{prop:modified-flow-homeo} this shows that $S_C^-$ is homeomorphic to the unstable set for the gradient flow \eqref{eqn:neg-grad-flow}. The $K_{\bf v}$-equivariance of this homeomorphism follows from the $K_{\bf v}$-equivariance of the modified flow \eqref{eqn:modified-flow-def}, the linearised flow $e^{i \beta t} = e^{i \mu(x) t}$ and the negative slice $k \cdot S_x^- = S_{k \cdot x}^-$ for all $k \in K_{\bf v}$.
\end{proof}

\begin{corollary}\label{cor:iso-classes-correspond}
The isomorphism classes in $S_x^-$ are in bijective correspondence with the isomorphism classes in $W_x^-$.
\end{corollary}

\begin{proof}
Since the map $\psi_x$ is defined using the action of $G_{\bf v}$, then the previous lemma together with Proposition \ref{prop:modified-flow-homeo} shows that the statement is true in a neighbourhood of $x$. Since $e^{i \beta t} \cdot y \rightarrow x$ for all $y \in S_x^-$, then any representation in $S_x^-$ is isomorphic to one in this neighbourhood. Similarly, $\lim_{t \rightarrow - \infty} \phi(y, t) = x$ for all $z \in W_x^-$ and so the same is true for $W_x^-$. Therefore the isomorphism classes of $S_x^-$ are in bijective correspondence with those of $W_x^-$. 
\end{proof}

Since $\psi(y) \in G_{\bf v} \cdot y$ for all $y \in S^-$, then $\psi$ is still a homeomorphism on restriction to any subset of $\Rep(Q, {\bf v})$ which is a union of $G_{\bf v}$-orbits. An important special case is the subset $\nu^{-1}(0)$, where $\nu : \Rep(Q, {\bf v}) \rightarrow \Rel(Q, {\bf v}, \mathcal{R})$ is the function from \eqref{eqn:relation-function} determined by a set of relations in the path algebra of $Q$.

\begin{corollary}\label{cor:homeo-restrict-to-singular}
Let $Z \subset \Rep(Q, {\bf v})$ be a closed $G_{\bf v}$-invariant subset. Then the restriction $\psi : S_C^- \cap Z \rightarrow W_C^- \cap Z$ is a $K_{\bf v}$-equivariant homeomorphism of a neighbourhood $U$ of $C$ in $S_C^- \cap Z$ with a neighbourhood $V$ of $C$ in $W_C^- \cap Z$. This determines a $K_{\bf v}$-equivariant homeomorphism of pairs $\left. \psi \right|_U (U, U \setminus C) \stackrel{\cong}{\longrightarrow} (V, V \setminus C)$. 
\end{corollary}

\begin{remark}
The condition that $Z \subset \Rep(Q, {\bf v})$ is closed and $G_{\bf v}$-invariant is sufficient to apply this result to the subset of representations satisfying a given set of relations. More generally, we only need that the subset $Z$ is preserved by (a) the gradient flow of $\| \mu - \alpha \|^2$ and (b) the action of $e^{i \beta t}$, where $\beta = \mu(x) - \alpha$ for some critical point $x$. Therefore the above result also applies after reducing the structure group of $G_{\bf v}$ to a subgroup (for example when studying real structures in analogy with the work of Hitchin on real Higgs bundles in \cite{Hitchin92}).
\end{remark}

\subsection{Gradient flow lines and the Hecke correspondence}\label{subsec:flow-classification}

The goal of this section is to classify the gradient flow lines connecting critical points in the space of representations of a quiver. Since the gradient flow is generated by the action of $G_{\bf v}$, then these results remain valid on restricting to any closed $G_{\bf v}$-invariant subset of $\Rep(Q, {\bf v})$.

Let $x_u$ and $x_\ell$ be two critical points connected by a flow line. Therefore there exists $y \in W_{x_u}^-$ such that $\lim_{t \rightarrow -\infty} \phi(y, t) = x_u$ and $\lim_{t \rightarrow \infty} \phi(y, t) = x_\ell$. Theorem \ref{thm:homeo-slice-bundle} shows that there exists a corresponding $\delta x \in S_{x_u}^-$ and $g \in G_{\bf v}$ such that $y = g \cdot (x_u + \delta x)$ and that every $\delta x \in S_{x_u}^-$ corresponds to some $y \in W_{x_u}^-$ in this way. Therefore we can determine the possible limits $x_\ell = \lim_{t \rightarrow \infty} \phi(y, t)$ by classifying the possible graded objects of the Harder-Narasimhan-Jordan-H\"older filtration of the representations $x_u + \delta x$ for $\delta x \in S_x^-$. In summary, the analytic question of classifying critical points connected by flow lines has been reduced to an algebraic question, which we will address in this section using homological algebra.

\subsubsection{Homological algebra for quivers with relations}\label{subsec:homological-algebra}

In this section we recall some results from homological algebra and use these to prove some preliminary results which will be used to study gradient flow lines in Section \ref{subsec:grad-flow-spaces}. Throughout this section, $\alpha$ is the stability parameter from Definition \ref{def:aasp} with respect to a fixed dimension vector ${\bf v}$. Fix a finite set of relations on the quiver and consider the subset $\nu^{-1}(0) \subset \Rep(Q, {\bf v})$ defined in \eqref{eqn:relation-function}. Let $x$ be a critical point in $\nu^{-1}(0)$. Recall from Proposition \ref{prop:singular-critical-decomp} that $x$ is the direct sum of subrepresentations $x_1 \oplus x_2$ with respective dimension vectors ${\bf v_1}$ and ${\bf v_2} = {\bf v} - {\bf v_1}$. Let ${\bf v_1}$ be the dimension vector which is non-zero at the vertex $\infty$ from Definition \ref{def:aasp}. Proposition \ref{prop:singular-critical-decomp} shows that $x_1$ is stable with $\mu(x_1) = k \alpha(Q, {\bf v_1}) \cdot \id$, where $\alpha(Q, {\bf v_1}) < 0$, and that $x_2$ is semistable with $\mu(x_2) = 0$. Now Lemma \ref{lem:slice-linearises} shows that $S_x^- \cong \Hom^1(Q, {\bf v_2}, {\bf v_1}) \cap \ker (\rho_x^\C)^* \cap \ker d \nu_x$. Therefore $S_x^-$ is isomorphic (as a vector space) to the middle cohomology of the complex
\begin{equation}\label{eqn:neg-def-complex}
\Hom^0(Q, {\bf v_2}, {\bf v_1}) \stackrel{\rho_x^\C}{\longrightarrow} \Hom^1(Q, {\bf v_2}, {\bf v_1}) \stackrel{d \nu_x}{\longlongrightarrow} \Rel(Q, {\bf v}, \mathcal{R})
\end{equation}
Denote the cohomology groups of this complex by $\mathcal{H}^0(Q, {\bf v_2}, {\bf v_1}) = \ker \rho_x^\C$ and $\mathcal{H}^1(Q, {\bf v_2}, {\bf v_1}) = \ker (\rho_x^\C)^* \cap \ker d \nu_x$. Given two representations $x_s \in \Rep(Q, {\bf v_s})$ and $x_q \in \Rep(Q, {\bf v_q})$, the extensions $0 \rightarrow x_s \rightarrow x \rightarrow x_q \rightarrow 0$ are parametrised up to isomorphism by the extension classes $\ker (\rho_x^\C)^* \subset \Hom^1(Q, {\bf v_q}, {\bf v_s})$. In the presence of a set of relations $\mathcal{R}$, we impose the extra condition that the homomorphisms $x_s \hookrightarrow x$ and $x \twoheadrightarrow x_q$ are compatible with the relations. This is summarised in the following lemma.

\begin{lemma}\label{lem:extension-slice}
The extension classes are parametrised by $S_x^- \cong \mathcal{H}^1(Q, {\bf v_2}, {\bf v_1})$.
\end{lemma}

Consider a representation $x \in \nu^{-1}(0) \subset \Rep(Q, {\bf v})$ given by an extension $0 \rightarrow x_s \rightarrow x \rightarrow x_q \rightarrow 0$ and let ${\bf v} = {\bf v_s} + {\bf v_q}$ be the corresponding decomposition of the dimension vector ${\bf v}$. Now let $x' \in \nu^{-1}(0) \subset \Rep(Q, {\bf v'})$ be another representation. Then there is a long exact sequence of cohomology groups
\begin{equation}\label{eqn:les-quiver-cohomology}
\mathcal{H}^0(Q, {\bf v'}, {\bf v_s}) \rightarrow \mathcal{H}^0(Q, {\bf v'}, {\bf v}) \rightarrow \mathcal{H}^0(Q, {\bf v'}, {\bf v_q}) \rightarrow \mathcal{H}^1(Q, {\bf v'}, {\bf v_s}) \rightarrow \mathcal{H}^1(Q, {\bf v'}, {\bf v}) \rightarrow \cdots
\end{equation}

In an analogous way to extensions of holomorphic bundles, the extension class $e \in \mathcal{H}^1(Q, {\bf v_q}, {\bf v_s})$ is the image of the identity homomorphism in $\mathcal{H}^0(Q, {\bf v_q}, {\bf v_q})$ by the connecting homomorphism in the long exact sequence \eqref{eqn:les-quiver-cohomology} with $x' = x_q$. Similarly, $e$ is also the image of $\id \in \mathcal{H}^0(Q, {\bf v_s}, {\bf v_s})$ by the connecting homomorphism in the long exact sequence
\begin{equation*}
\mathcal{H}^0(Q, {\bf v_q}, {\bf v_s}) \rightarrow \mathcal{H}^0(Q, {\bf v}, {\bf v_s}) \rightarrow \mathcal{H}^0(Q, {\bf v_s}, {\bf v_s}) \rightarrow \mathcal{H}^1(Q, {\bf v_q}, {\bf v_s}) \rightarrow \mathcal{H}^1(Q, {\bf v}, {\bf v_s}) \rightarrow \cdots
\end{equation*}

The following lemma is a quiver analog of \cite[Lem. 3.1]{NarasimhanRamanan69} for holomorphic bundles.

\begin{lemma}\label{lem:narasimhan-ramanan}
Let $x \in \nu^{-1}(0) \subset \Rep(Q, {\bf v})$ be an extension of representations $0 \rightarrow x_s \rightarrow x \rightarrow x_q \rightarrow 0$ with extension class $e \in \mathcal{H}^1(Q, {\bf v_q}, {\bf v_s})$, and let ${\bf v} = {\bf v_s} + {\bf v_q}$ be the corresponding decomposition of the dimension vector of the representation $x$. Given a positive dimension vector ${\bf d}$, let $x_q'$ be a subrepresentation of $x_q$ with dimension vector ${\bf v_q'} = {\bf v_q} - {\bf d}$ and associated inclusion map $i \in \mathcal{H}^0(Q, {\bf v_q'}, {\bf v_q})$. Then $x_q'$ lifts to a subrepresentation of $x$ if and only if the extension class $e$ is in the kernel of the pullback map $i^* : \mathcal{H}^1(Q, {\bf v_q}, {\bf v_s}) \rightarrow \mathcal{H}^1(Q, {\bf v_q'}, {\bf v_s})$. 

In the same way, if $x_s$ is a subrepresentation of $x_s'$ with dimension vector ${\bf v_s'} = {\bf v_s} + {\bf d}$ then $x_s'$ is a quotient representation of $x$ if and only if the extension class $e$ is in the kernel of $\mathcal{H}^1(Q, {\bf v_q}, {\bf v_s}) \rightarrow \mathcal{H}^1(Q, {\bf v_q}, {\bf v_s'})$.
\end{lemma}

\begin{proof}
We have the following commutative diagram of long exact sequences associated to $x_q' \stackrel{i}{\hookrightarrow} x_q$.
\begin{equation*}
\xymatrix{
 & & \cdots \ar[r] & \mathcal{H}^0(Q, {\bf v_q}, {\bf v_q}) \ar[d] \ar[r] & \mathcal{H}^1(Q, {\bf v_q}, {\bf v_s}) \ar[d] \ar[r] & \cdots \\
\cdots \ar[r] & \mathcal{H}^0(Q, {\bf v_q'}, {\bf v_s}) \ar[r] & \mathcal{H}^0(Q, {\bf v_q'}, {\bf v}) \ar[r] & \mathcal{H}^0(Q, {\bf v_q'}, {\bf v_q}) \ar[r] & \mathcal{H}^1(Q, {\bf v_q'}, {\bf v_s}) \ar[r] & \cdots 
}
\end{equation*}

The extension class $e \in \mathcal{H}^1(Q, {\bf v_q}, {\bf v_s})$ and the inclusion $i \in \mathcal{H}^0(Q, {\bf v_q'}, {\bf v_q})$ are the images of $\id \in \mathcal{H}^0(Q, {\bf v_q}, {\bf v_q})$ in $\mathcal{H}^1(Q, {\bf v_q}, {\bf v_s})$ and $\mathcal{H}^0(Q, {\bf v_q'}, {\bf v_q})$ respectively. The long exact sequence above shows that $i$ lifts to a homomorphism $\tilde{i} \in \mathcal{H}^0(Q, {\bf v_q'}, {\bf v})$ iff it maps to zero in $\mathcal{H}^1(Q, {\bf v_q'}, {\bf v_s})$, which occurs iff the extension class $e$ maps to zero in $\mathcal{H}^1(Q, {\bf v_q'}, {\bf v_s})$.

The second statement follows from an analogous argument with the long exact sequences associated to $x_s \hookrightarrow x_s'$.
\begin{equation*}
\xymatrix{
 & & \cdots \ar[r] & \mathcal{H}^0(Q, {\bf v_s}, {\bf v_s}) \ar[d] \ar[r] & \mathcal{H}^1(Q, {\bf v_q}, {\bf v_s}) \ar[d] \ar[r] & \cdots \\
\cdots \ar[r] & \mathcal{H}^0(Q, {\bf v_q}, {\bf v_s'}) \ar[r] & \mathcal{H}^0(Q, {\bf v}, {\bf v_s'}) \ar[r] & \mathcal{H}^0(Q, {\bf v_s}, {\bf v_s'}) \ar[r] & \mathcal{H}^1(Q, {\bf v_q}, {\bf v_s'}) \ar[r] & \cdots 
}
\end{equation*}

\end{proof}

Now we can define Hecke modifications of quivers with relations. First consider the case of a one-dimensional Hecke modification at a single vertex of $Q$. In analogy with \cite{witten-hecke} we call these \emph{miniscule Hecke modifications}.

\begin{definition}
Let $k \in \mathcal{I}$ be a vertex of the quiver $Q$. A pair of representations $(x_1, x_2) \in \nu_{\bf v_1 - e_k}^{-1}(0) \times \nu_{\bf v_1}^{-1}(0)$ is \emph{related by a Hecke modification} if $x_1$ is a subrepresentation of $x_2$ (cf. \cite[Sec. 10]{Nakajima94}, \cite[Sec. 5]{Nakajima98}). Therefore we have the subset
\begin{multline*}
\tilde{\mathcal{B}}_k(Q, {\bf v_1}, \mathcal{R}) := \\
 \{ (x_1, x_2) \in \nu_{\bf v_1 - e_k}^{-1}(0)^{\alpha-st} \times \nu_{\bf v_1}^{-1}(0)^{\alpha-st} \mid \text{$x_1$ and $x_2$ are related by a Hecke modification} \} .
\end{multline*}
The \emph{Hecke correspondence} $\mathcal{B}_k(Q, {\bf v_1}, \mathcal{R})$ (see \cite{Nakajima94}, \cite{Nakajima98}) is the subvariety of $\mathcal{M}(Q, {\bf v_1}-{\bf e_k}, \mathcal{R}) \times \mathcal{M}(Q, {\bf v_1}, \mathcal{R})$ induced from the quotient maps $\nu_{\bf v_1 - e_k}^{-1}(0)^{\alpha-st} \rightarrow \mathcal{M}(Q, {\bf v_1}-{\bf e_k}, \mathcal{R})$ and $\nu_{\bf v_1}^{-1}(0)^{\alpha-st} \rightarrow \mathcal{M}(Q, {\bf v_1}, \mathcal{R})$
\begin{equation*}
\xymatrix{
 & \tilde{\mathcal{B}}_k(Q, {\bf v_1}, \mathcal{R}) \ar[dl] \ar[dr] \ar[d] & \\
\nu_{\bf v_1 - e_k}^{-1}(0)^{\alpha-st} \ar[d] & \ar[dl] \mathcal{B}_k(Q, {\bf v_1}, \mathcal{R}) \ar[dr] & \nu_{\bf v_1}^{-1}(0)^{\alpha-st} \ar[d] \\
\mathcal{M}(Q, {\bf v_1}-{\bf e_k}, \mathcal{R}) & & \mathcal{M}(Q, {\bf v_1}, \mathcal{R})
}
\end{equation*}
\end{definition}
Equivalently, a miniscule Hecke modification corresponds to a choice of surjective homomorphism $V_k \stackrel{v}{\longrightarrow} \C$ at a single vertex $k \in \mathcal{I}$ such that the following diagram commutes. The subrepresentation $x_2$ is the induced representation in $\Rep(Q, {\bf v_1} - {\bf e_k})$.
\begin{equation*}
\xymatrix{
0 \ar[r] & \Vect(Q, {\bf v_1} - {\bf e_k}) \ar[r] \ar@{-->}[d]^{x_2} & \Vect(Q, {\bf v_1}) \ar[r]^(0.65)v \ar[d]^{x_1} & \C \ar[d]^{\id} \ar[r] & 0 \\
0 \ar[r] & \Vect(Q, {\bf v_1} - {\bf e_k}) \ar[r] & \Vect(Q, {\bf v_1}) \ar[r]^(0.65)v & \C \ar[r] & 0
}
\end{equation*}
More generally, one can define multiple Hecke modifications associated to a dimension vector ${\bf d} = ( n_k )_{k \in \mathcal{I}} > {\bf 0}$ as follows (see also \cite{Nakajima94}, \cite{Nakajima98}). Given $x_1 \in \nu_{\bf v_1}^{-1}(0)^{\alpha-st}$ and a subrepresentation $x_2 \in \nu_{{\bf v_1}-{\bf d}}^{-1}(0)^{\alpha-st}$, at each vertex $k \in \mathcal{I}$, choose a homomorphism $V_k \stackrel{v_k}{\longrightarrow} \C^{n_k}$ such that the following diagram commutes.
\begin{equation*}
\xymatrix{
0 \ar[r] & \Vect(Q, {\bf v_1} - {\bf d}) \ar[r] \ar@{-->}[d]^{x_2} & \Vect(Q, {\bf v_1}) \ar[r]^{\oplus v_k} \ar[d]^{x_1} & \bigoplus_{k \in \mathcal{I}} \C^{n_k} \ar[d]^{\id} \ar[r] & 0 \\
0 \ar[r] & \Vect(Q, {\bf v_1} - {\bf d}) \ar[r] & \Vect(Q, {\bf v_1}) \ar[r]^{\oplus v_k} & \bigoplus_{k \in \mathcal{I}} \C^{n_k} \ar[r] & 0
}
\end{equation*}

Given a stable representation $x_1$ (recall that we are using the stability parameter from Definition \ref{def:aasp}) and a Hecke modification $\bigoplus_{k \in \mathcal{I}} v_k : \Vect(Q, {\bf v}) \rightarrow \bigoplus_{k \in \mathcal{I}} \C^{n_k}$ as above, the stability of $x_2$ follows automatically, since if there exists a subrepresentation $x'$ of $x_2$ with dimension vector ${\bf v'}$ such that $\slope_\alpha(Q, {\bf v'}) > \slope_\alpha(Q, {\bf v_2}) \geq 0$ then $\slope_\alpha(Q, {\bf v'}) > 0 > \slope_\alpha(Q, {\bf v_1})$ which contradicts the fact that $x_1$ is stable. 

This is summarised in the following lemma.

\begin{lemma}
Let ${\bf d} > {\bf 0}$ be a dimension vector and let $x_1 \in \nu_{\bf v_1}^{-1}(0)$ and $x_2 \in \nu_{{\bf v_1} - {\bf d}}^{-1}(0)$ be two representations related by a Hecke modification. Then $x_1$ stable implies $x_2$ stable. 
\end{lemma}

Conversely, if $x_2$ is stable, then $x_1$ is not necessarily stable (for example, it could be the direct sum of $x_2$ with the trivial representation). 

\begin{lemma}\label{lem:non-trivial-kernel}
Let ${\bf d} > {\bf 0}$ be a dimension vector and let $x_1 \in \nu_{\bf v_1}^{-1}(0)^{\alpha-st}$ and $x_1' \in \nu_{{\bf v_1} - {\bf d}}^{-1}(0)^{\alpha-st}$ be two stable representations related by a Hecke modification such that $\slope_\alpha(x_1') < \slope_\alpha(x_1) < 0$. Let $x_2 \in \nu_{\bf v_2}^{-1}(0)^{\alpha-ss}$ be a semistable representation with positive slope. Then the kernel of $\mathcal{H}^1(Q, {\bf v_2}, {\bf v_1-d}) \rightarrow \mathcal{H}^1(Q, {\bf v_2}, {\bf v_1})$ is isomorphic to $\mathcal{H}^0(Q, {\bf v_2}, {\bf d})$. In particular, this has non-zero dimension if and only if $\mathcal{H}^0(Q, {\bf v_2}, {\bf d})$ has non-zero dimension.
\end{lemma}

\begin{proof}
Since $x_1$ and $x_1'$ are related by a Hecke modification, then there is an exact sequence
\begin{equation*}
0 \rightarrow x_1' \rightarrow x_1 \rightarrow x_1 / x_1' \rightarrow 0
\end{equation*}
This induces a long exact sequence in cohomology
\begin{equation*}
\mathcal{H}^0(Q, {\bf v_2}, {\bf v_1 - d}) \rightarrow \mathcal{H}^0(Q, {\bf v_2}, {\bf v_1}) \rightarrow \mathcal{H}^0(Q, {\bf v_2}, {\bf d}) \rightarrow \mathcal{H}^1(Q, {\bf v_2}, {\bf v_1 - d}) \rightarrow \mathcal{H}^1(Q, {\bf v_2}, {\bf v_1}) \rightarrow \cdots
\end{equation*}
Since $x_1$ and $x_2$ are semistable and $\slope_\alpha(Q, {\bf v_1}) < 0 = \slope_\alpha(Q, {\bf v_2})$ then $\mathcal{H}^0(Q, {\bf v_2}, {\bf v_1}) = 0$. Therefore the kernel of $\mathcal{H}^1(Q, {\bf v_2}, {\bf v_1 - d}) \rightarrow \mathcal{H}^1(Q, {\bf v_2}, {\bf v_1})$ is isomorphic to $\mathcal{H}^0(Q, {\bf v_2}, {\bf d})$.
\end{proof}

Now we give a simple condition for $\mathcal{H}^0(Q, {\bf v_2}, {\bf d})$ to have non-zero dimension which applies in the case of Theorem \ref{thm:miniscule-flow-hecke}. Let ${\bf d} = {\bf e_k}$ for some vertex $k \in \mathcal{I}$ and let $x_1' \in \Rep(Q, {\bf v_1} - {\bf e_k})$ be a subrepresentation of $x_1 \in \Rep(Q, {\bf v_1})$. If the quiver $Q$ has no loops at the vertex $k$ (i.e. there are no edges $e \in \mathcal{E}$ such that $h(e) = t(e) = k$), then let $x_2 \in \Rep(Q, {\bf v_2})$ be the zero representation. Then the map $\rho_x^\C$ from \eqref{eqn:neg-def-complex} is zero, and so $\mathcal{H}^0(Q, {\bf v_2}, {\bf e_k}) \cong \C^{({\bf v_2})_k}$, where $({\bf v_2})_k$ denotes the dimension of ${\bf v_2}$ at the vertex $k \in \mathcal{I}$.

If the quiver has loops $e_1, \ldots, e_\ell$ at the vertex $k$, then $x_1 / x_1' \in \Rep(Q, {\bf e_k}) \cong \C^\ell$. Write $(A_1, \ldots, A_\ell) := x_1 / x_1'$ with each $A_i \in \C$. Then define $x_2 \in \Rep(Q, {\bf v_2})$ as a representation which is equal to $A_i \cdot \id$ on the edge $e_i$, and zero elsewhere. Note that any relation in the path algebra satisfied by $x_1 / x_1'$ is also satisfied by $x_2$ and vice-versa. Then once again, the map $\rho_x^\C$ from \eqref{eqn:neg-def-complex} is zero, and so $\mathcal{H}^0(Q, {\bf v_2}, {\bf e_k}) \cong \C^{({\bf v_2})_k}$.

These results are summarised in the following lemma.

\begin{lemma}\label{lem:positive-dimension}
Let $x_1' \in \Rep(Q, {\bf v_1} - {\bf e_k})$ be a subrepresentation of $x_1 \in \Rep(Q, {\bf v_1})$. If $Q$ has no loops at $k \in \mathcal{I}$, then let $x_2 \in \Rep(Q, {\bf v_2})$ be the zero representation. If $Q$ has at least one loop at $k$, then define $x_2$ as above. Then $\mathcal{H}^0(Q, {\bf v_2}, {\bf e_k}) \cong \C^{({\bf v_2})_k}$.
\end{lemma}

\subsubsection{Spaces of gradient flow lines}\label{subsec:grad-flow-spaces}

In this section we relate the homological algebra of the previous section to the problem of constructing gradient flow lines between critical points. The main theorem is Theorem \ref{thm:miniscule-flow-hecke} which shows that critical points connected by flow lines determine the Hecke correspondence.

Throughout this section we fix a dimension vector ${\bf v}$ and consider the subvariety $\nu^{-1}(0) \subset \Rep(Q, {\bf v})$ determined by a finite set of relations on the quiver. Consider a pair of critical points $x_u = x_1^u \oplus x_2^u$ and $x_\ell = x_1^\ell \oplus x_2^\ell$ connected by a flow line $\{ y_t \, : \, t \in \R \} \subset \nu^{-1}(0)$. Let ${\bf v} = {\bf v_1^u} + {\bf v_2^u} = {\bf v_1^\ell} + {\bf v_2^\ell}$ be the corresponding decomposition of the dimension vector. We always use the convention that ${\bf v_1^u}$ and ${\bf v_1^\ell}$ are non-zero on the vertex $\infty$ from Definition \ref{def:aasp}.

Theorem \ref{thm:homeo-slice-bundle} shows that there exists $t$ such that $y_t \in W_{x_u}^-$ is isomorphic to a representation $z \in S_x^-$. Lemma \ref{lem:extension-slice} shows that for the stability parameter of Definition \ref{def:aasp}, the representation $z$ is isomorphic to an extension $0 \rightarrow x_1^u \rightarrow z \rightarrow x_2^u \rightarrow 0$. Let $e \in \mathcal{H}^1(Q, {\bf v_2}, {\bf v_1})$ be the extension class of $0 \rightarrow x_1^u \rightarrow z \rightarrow x_2^u \rightarrow 0$. 

Since $x_\ell$ is the downward limit of the flow with initial condition $z$, then $x_\ell$ is isomorphic to the graded object of the HNJH filtration of $z$ (Theorem \ref{thm:algebraic-flow-limit}) and therefore we have an extension $0 \rightarrow (x_2^\ell)' \rightarrow z \rightarrow (x_1^\ell)' \rightarrow 0$, where $(x_2^\ell)'$ is the maximal semistable subrepresentation of $z$ and $x_2^\ell$ is isomorphic to the graded object of the Jordan-H\"older filtration of $(x_2^\ell)'$. Note that by the choice of stability parameter from Definition \ref{def:aasp}, the quotient $(x_1^\ell)'$ is stable and isomorphic to $x_1^\ell$.

Since $x_u$ is critical then $x_1^u$ is stable and $x_2^u$ is semistable. Since $f(x_u) > f(x_\ell)$ then $\slope(Q, {\bf v_1^u}) < \slope(Q, {\bf v_1^\ell}) < \slope(Q, {\bf v_2^\ell}) < \slope(Q, {\bf v_2^u})$ and so $\mathcal{H}^0(Q, {\bf v_2^\ell}, {\bf v_1^u}) = 0$. Therefore $(x_2^\ell)'$ is a subrepresentation of $x_2^u$. 

If the image of $x_1^u \rightarrow z$ intersects with the image of $(x_2^\ell)' \rightarrow z$ then, since $(x_2^\ell)' \rightarrow x_2^u$ is injective, we have a non-zero map $x_1^u \rightarrow x_2^u$ contradicting the fact that $0 \rightarrow x_1^u \rightarrow z \rightarrow x_2^u \rightarrow 0$ is a short exact sequence. Therefore the image of $x_1^u \rightarrow z$ maps injectively to $(x_1^\ell)'$ and so $x_1^u$ is a subrepresentation of $(x_1^\ell)'$. This is summarised in the diagram below.
\begin{equation*}
\xymatrix@R=1em{
& & & & 0 \ar[dl] \\ 
& & & (x_2^\ell)' \ar[dl] \ar[d] \\
0 \ar[r] & x_1^u \ar[r] \ar[d] & z \ar[r] \ar[dl] & x_2^u \ar[r] & 0 \\ 
 & (x_1^\ell)' \ar[dl] \\
0
}
\end{equation*}
The result of Lemma \ref{lem:narasimhan-ramanan} shows that the extension class $e$ is in the kernel of the pullback map $\mathcal{H}^1(Q, {\bf v_2^u}, {\bf v_1^u}) \rightarrow \mathcal{H}^1(Q, {\bf v_2^u}, {\bf v_1^\ell})$. Therefore we have proved
\begin{lemma}\label{lem:necessary-flow-line}
If $x_u = x_1^u \oplus x_2^u$ and $x_\ell = x_1^\ell \oplus x_2^\ell$ are connected by a flow line then $x_1^\ell$ is isomorphic to a Hecke modification of $x_1^u$. Moreover, there exists an extension $0 \rightarrow x_1^u \rightarrow z \rightarrow x_2^u \rightarrow 0$ such that the extension class $e \in \mathcal{H}^1(Q, {\bf v_2^u}, {\bf v_1^u})$ pulls back to zero in $\mathcal{H}^1(Q, {\bf v_2^u}, {\bf v_1^\ell})$. 
\end{lemma}

Conversely, suppose that $x_1^u$ and $x_1^\ell$ are stable and related by a Hecke modification, and that ${\bf v_1^\ell} = {\bf v_1^u} + {\bf e_k}$. If $\dim_\C \mathcal{H}^0(Q, {\bf v_2^u}, {\bf e_k}) > 0$ then Lemma \ref{lem:non-trivial-kernel} shows that $\ker \left( \mathcal{H}^1(Q, {\bf v_2^u}, {\bf v_1^u}) \rightarrow \mathcal{H}^1(Q, {\bf v_2^u}, {\bf v_1^\ell}) \right)$ is non-trivial, and so we can choose an extension $0 \rightarrow x_1^u \rightarrow z \rightarrow x_2^u \rightarrow 0$ such that the extension class $e$ pulls back to zero in $\mathcal{H}^1(Q, {\bf v_2^u}, {\bf v_1^\ell})$.  Lemma \ref{lem:positive-dimension} shows that there always exists $x_2^u$ such that $\dim_\C \mathcal{H}^0(Q, {\bf v_2^u}, {\bf e_k}) > 0$ if ${\bf v_2^u}$ has non-zero dimension at the vertex $k \in \mathcal{I}$, which is always true since ${\bf v_1^\ell} = {\bf v_1^u} + {\bf e_k}$ implies ${\bf v_2^\ell} = {\bf v_2^u} - {\bf e_k} \geq {\bf 0}$.

Since $z$ is a nontrivial extension, then the Harder-Narasimhan type of $z$ is strictly less than that of $x_u = x_1^u \oplus x_2^u$. Moreover, since $e$ pulls back to zero in $\mathcal{H}^1(Q, {\bf v_2^u}, {\bf v_1^\ell})$ then $x_1^\ell$ is a quotient of $z$ by Lemma \ref{lem:narasimhan-ramanan}. Therefore there is an extension $0 \rightarrow x_2^\ell \rightarrow z \rightarrow x_1^\ell \rightarrow 0$ and using the same argument as the proof of Lemma \ref{lem:necessary-flow-line} we can show that $x_2^\ell$ is a subrepresentation of $x_2^u$. If $x_2^\ell$ is unstable, then the maximal semistable subrepresentation of $x_2^\ell$ has slope greater than $\slope(x_2^\ell)$. Since ${\bf v_2^\ell} = {\bf v_2^u} - {\bf e_k}$, then this contradicts the fact that the Harder-Narasimhan type of $z$ is strictly less than that of $x_u$. Therefore $x_2^\ell$ is the maximal semistable subrepresentation of $z$. If $x_2^\ell$ is not polystable, then replace $x_2^\ell$ by the graded object of the Jordan-H\"older filtration of $x_2^\ell$. Therefore $x_\ell = x_1^\ell \oplus x_2^\ell$ is isomorphic to the graded object of the HNJH filtration of $x$. Therefore we have proved

\begin{lemma}\label{lem:sufficient-miniscule-flow-line}
Let $x_1^u$ be a minimiser of $\| \mu - \alpha \|^2$ on $\nu_{\bf v_1^u}^{-1}(0)$, let $x_1^\ell \in \nu_{{\bf v_1^u} + {\bf e_k}}^{-1}(0)^{st}$ be a Hecke modification of $x_1^u$ and choose $x_2^u$ such that $\dim_\C \mathcal{H}^0(Q, {\bf v_2^u}, {\bf e_k}) > 0$ and $x_u = x_1^u \oplus x_2^u$ is critical for $\| \mu - \alpha \|^2$ on $\nu_{\bf v}^{-1}(0)$ where ${\bf v} = {\bf v_1^u}+{\bf v_2^u}$. Then there exists $x_2^\ell$ such that $x_\ell = x_1^\ell \oplus x_2^\ell$ is isomorphic to a critical point connected by a flow line to $x_u$.
\end{lemma}

Now we can use Lemmas \ref{lem:necessary-flow-line} and \ref{lem:sufficient-miniscule-flow-line} to show that the Hecke correspondence is determined by pairs of critical points connected by a flow line. 

\begin{definition}\label{def:flow-line-spaces}
Let $C_\ell$ and $C_u$ be two critical sets with $f(C_\ell) < f(C_u)$. First define the \emph{space of representations that flow up to $C_u$ and down to $C_\ell$}
\begin{equation*}
\tilde{\mathcal{F}}_{\ell, u} := \{ y \in \nu^{-1}(0) \mid \lim_{t \rightarrow \infty} \phi(y, t) \in C_\ell, \lim_{t \rightarrow -\infty} \phi(y, t) \in C_u \} .
\end{equation*}
The gradient flow $\phi(y, \cdot)$ defines a natural $\mathbb{R}$-action on $\tilde{\mathcal{F}}_{\ell, u}$. The \emph{space of flow lines} connecting $C_\ell$ and $C_u$ is $\mathcal{F}_{\ell, u} := \tilde{\mathcal{F}}_{\ell, u} / \mathbb{R}$. The \emph{space of pairs of critical points connected by a flow line} is
\begin{equation*}
\mathcal{P}_{\ell, u} := \{ (x_\ell, x_u) \in C_\ell \times C_u \mid \exists y \in \tilde{\mathcal{F}}_{\ell, u} \, \, \text{such that} \, \lim_{t \rightarrow \infty} \phi(y, t) = x_\ell, \lim_{t \rightarrow -\infty} \phi(y, t) = x_u \}
\end{equation*}
\end{definition}
The canonical projection maps taking a flow line to its upper and lower endpoints fit into the following commutative diagram
\begin{equation}\label{eqn:flow-projection}
\xymatrix{
 & \tilde{\mathcal{F}}_{\ell, u} \ar[d] \ar@/_1.2pc/[dddl] \ar@/^1.2pc/[dddr] & \\
 & \mathcal{F}_{\ell, u} \ar[d] \ar@/_0.5pc/[ddl] \ar@/^0.5pc/[ddr] & \\
 & \mathcal{P}_{\ell, u} \ar[dl] \ar[dr] & \\
C_\ell & & C_u 
}
\end{equation}

Lemma \ref{lem:quotient-critical} shows that in the special case of the stability parameter from Definition \ref{def:aasp}, each critical set (modulo isomorphism) splits as a product
\begin{equation*}
C_{\bf v_1^\ell} / K_{\bf v} \cong \mathcal{M}_\alpha(Q, {\bf v_1^\ell}, \mathcal{R}) \times \mathcal{M}_0(Q, {\bf v} - {\bf v_1^\ell}, \mathcal{R})
\end{equation*}
and so there is a projection map $C_{\bf v_1^\ell} \rightarrow \mathcal{M}_\alpha(Q, {\bf v_1^\ell}, \mathcal{R})$. Given a vertex $k \in \mathcal{I}$ such that ${\bf v_1^\ell}$ has positive dimension at $k$, the critical sets $C_{\bf v_1^\ell}$ and $C_{\bf v_1^\ell - e_k}$ satisfy $f(C_{\bf v_1^\ell}) < f(C_{\bf v_1^\ell - e_k})$. Since the gradient flow is $K_{\bf v}$-equivariant, then there is an induced subvariety $\mathcal{M}_{\bf v_1^\ell, v_1^\ell - e_k} \subset \mathcal{M}_\alpha(Q, {\bf v_1^\ell}, \mathcal{R}) \times \mathcal{M}_\alpha(Q, {\bf v_1^\ell} - {\bf e_k}, \mathcal{R})$ such that the following diagram commutes
\begin{equation*}
\xymatrix{
 & \mathcal{P}_{\bf v_1^\ell, v_1^\ell - e_k} \ar[dl] \ar[dr] \ar[d] & \\
C_{\bf v_1^\ell} \ar[d] & \mathcal{M}_{\bf v_1^\ell, v_1^\ell - e_k} \ar[dl] \ar[dr] & C_{\bf v_1^\ell - e_k} \ar[d] \\
\mathcal{M}_\alpha(Q, {\bf v_1^\ell}, \mathcal{R}) & & \mathcal{M}_\alpha(Q, {\bf v_1^\ell} - {\bf e_k}, \mathcal{R})
}
\end{equation*}

\begin{theorem}\label{thm:miniscule-flow-hecke}
$\mathcal{M}_{\bf v_1^\ell, v_1^\ell - e_k}$ is the Hecke correspondence.
\end{theorem}

\begin{proof}
Lemma \ref{lem:necessary-flow-line} shows that if $x_u = x_1^u \oplus x_2^u$ and $x_\ell = x_1^\ell \oplus x_2^\ell$ are connected by a flow line, then the pair of isomorphism classes  $([x_1^\ell], [x_1^u])$ is in the Hecke correspondence $\mathcal{B}_k(Q, {\bf v_1^\ell}, \mathcal{R})$. Conversely, if $([(x_1^u)'], [(x_1^\ell)']) \in \mathcal{B}_k(Q, {\bf v_1^\ell}, \mathcal{R})$ then Lemma \ref{lem:sufficient-miniscule-flow-line} shows that there exist representatives $x_1^u$ and $x_1^\ell$ for $[(x_1^u)']$ and $[(x_1^\ell)']$ respectively, and there exist $x_2^u$ and $x_2^\ell$ such that $x_u = x_1^u \oplus x_2^u$ and $x_\ell = x_1^\ell \oplus x_2^\ell$ are critical points connected by a flow line.
\end{proof}

%\bibliographystyle{plain}
%\bibliography{../ref}

\end{document}